\documentclass[a4paper]{amsart}
\usepackage{hyperref}
\usepackage{color}
\input{xypic}

\usepackage[a4paper]{geometry}
\usepackage[all]{xy}
\usepackage{amsthm}
\usepackage{amsopn}
\usepackage{amsmath}
\usepackage{amssymb}
\usepackage{amsaddr}
\usepackage{mathrsfs}
\usepackage{enumitem}

\usepackage{tikz-cd}

\DeclareMathOperator{\ad}{ad}
\DeclareMathOperator{\Aut}{Aut}

\DeclareMathOperator{\coker}{coker}
\DeclareMathOperator{\cone}{Cone}
\DeclareMathOperator{\coord}{Coord}
\DeclareMathOperator{\Der}{Der}
\DeclareMathOperator{\DR}{DR}

\DeclareMathOperator{\en}{End}

\DeclareMathOperator{\gr}{gr}
\DeclareMathOperator{\Hoch}{HH}
\DeclareMathOperator{\Hom}{Hom}
\DeclareMathOperator{\id}{Id}
\DeclareMathOperator{\Lie}{Lie}
\DeclareMathOperator{\LC}{LC}
\DeclareMathOperator{\LM}{LM}
\DeclareMathOperator{\LT}{LT}
\DeclareMathOperator{\ran}{Ran}
\DeclareMathOperator{\res}{res}

\DeclareMathOperator{\sn}{sn}
\DeclareMathOperator{\Spec}{Spec}
\DeclareMathOperator{\Specm}{Specm}
\DeclareMathOperator{\sing}{SS}
\DeclareMathOperator{\Sym}{Sym}
\DeclareMathOperator{\Tor}{Tor}

\DeclareMathOperator{\zhu}{Zhu}

\newcommand{\g}{\mathfrak{g}}
\newcommand{\ZZZ}{\mathfrak{Z}}
\newcommand{\al}{\alpha}
\newcommand{\Om}{\Omega}
\newcommand{\om}{\omega}
\newcommand{\D}{\Delta}
\def\G{\Gamma}
\newcommand{\La}{\Lambda}
\newcommand{\la}{\lambda}
\def\C{\mathbb{C}}

\newcommand{\Z}{\mathbb{Z}}
\newcommand{\CA}{\mathcal{A}}
\newcommand{\CC}{\mathcal{C}}
\newcommand{\CD}{\mathcal{D}}
\newcommand{\CF}{\mathcal{F}}
\newcommand{\HH}{\mathcal{H}}
\newcommand{\CK}{\mathcal{K}}

\newcommand{\CM}{\mathcal{M}}
\newcommand{\OO}{\mathcal{O}}

\newcommand{\CS}{\mathcal{S}}
\newcommand{\CV}{\mathcal{V}}
\newcommand{\mx}{\mathring{X}}
\newcommand{\XtotheS}{X^{\mathcal{S}}}

\newcommand{\zuta}{\underline{\zeta}}
\newcommand{\envela}{A^{\text{e}}}
\newcommand{\ov}[1]{\overline{#1}}
\newcommand{\vac}{\mathbf{1}}
\newcommand{\vir}{\text{Vir}}

\swapnumbers
\newtheoremstyle{exps}{\topsep}{\topsep}{}{0pt}{\bfseries}{.}{0pt}{}

\newtheorem*{thm*}{Theorem}
\newtheorem*{prop*}{Proposition}
\newtheorem*{lem*}{Lemma}
\newtheorem*{cor*}{Corollary}
\newtheorem*{rem*}{Remark}
\newtheorem{thm}{Theorem}[section]
\newtheorem{prop}[thm]{Proposition}
\newtheorem{lem}[thm]{Lemma}

\newtheorem{cor}[thm]{Corollary}

\theoremstyle{definition}
\newtheorem*{defn*}{Definition}
\newtheorem*{exer*}{Exercise}

\newtheorem*{problem*}{Problem}
\newtheorem{rem}[thm]{Remark}
\newtheorem{nolabel}[thm]{ }

\theoremstyle{exps}

\numberwithin{equation}{section}

\setenumerate[1]{label={\alph*)}} 

\title[Chiral Homology of Elliptic Curves and the Zhu Algebra]{Chiral Homology of Elliptic Curves and the Zhu Algebra}
\author[*]{Jethro van Ekeren}
\author[**]{Reimundo Heluani}

\begin{document}


\begin{center}
{\LARGE \bf Chiral Homology of Elliptic Curves and the Zhu Algebra} \par \bigskip

\renewcommand*{\thefootnote}{\fnsymbol{footnote}}
{\normalsize
Jethro van Ekeren\footnote{email: \texttt{jethrovanekeren@gmail.com}}\textsuperscript{1}
Reimundo Heluani \textsuperscript{2},
}

\par \bigskip

\textsuperscript{1}{\footnotesize Instituto de Matem\'{a}tica e Estat\'{i}stica (GMA), UFF, Niter\'{o}i RJ, Brazil}

\par 

\textsuperscript{2}{\footnotesize Instituto Nacional de Matem\'{a}tica Pura e Aplicada, Rio de Janeiro, RJ, Brazil}

\par \bigskip
\end{center}

\vspace*{10mm}

\noindent
\textbf{Abstract.} We study the chiral homology of elliptic curves with coefficients in a quasiconformal vertex algebra $V$. Our main result expresses the nodal curve limit of the first chiral homology group in terms of the Hochschild homology of the Zhu algebra of $V$. A technical result of independent interest regarding the relationship between the associated graded of $V$ with respect to Li's filtration and the arc space of the $C_2$-algebra is proved. 

\vspace*{10mm}


\section{Introduction}

\begin{nolabel}
Spaces of conformal blocks play a central role in the mathematical approach to conformal field theory based on vertex algebras, and are the point of contact in a fruitful interaction between representation theory and the geometry of moduli spaces.

Let $\g$ be a finite dimensional semisimple Lie algebra, $\widehat{\g} = \g((t)) \oplus \C K$ the associated affine Kac-Moody algebra, $V$ a $\widehat{\g}$-module, and $X$ a smooth complex algebraic curve. We denote by $\widehat{\g}_{\text{out}}$ the Lie algebra of meromorphic $\g$-valued functions on $X$ with possible pole at a point $x \in X$. If we choose a coordinate $t$ at $x$ then we may use it to expand in Laurent series and obtain a morphism $\widehat{\g}_{\text{out}} \rightarrow \widehat{\g}$. The space of conformal blocks is then (the dual of) the vector space of coinvariants
\begin{align}\label{liealg.coinvariants}
H(X, x, \g, V) = \frac{V}{\widehat{\g}_{\text{out}} \cdot V},
\end{align}
which is well defined independently of the choices made. Spaces of conformal blocks appear in connection with moduli spaces of $G$-bundles over $X$ in the guise of nonabelian theta functions.


The construction (\ref{liealg.coinvariants}) blends the notions of Lie algebra
homology and de Rham cohomology of an algebraic variety. This perspective finds
natural expression in Beilinson-Drinfeld's definition of chiral homology. A
vertex algebra, and more generally a chiral algebra, may be interpreted as a Lie
algebra within a certain category of $\CD$-modules over the Ran space of an
algebraic curve. The chiral homology is then defined as the de Rham cohomology
of the Chevalley complex of this Lie algebra. The space of conformal blocks is
recovered as the dual of the zeroth chiral homology.

Higher chiral homology groups are of interest in the geometric Langlands program, specifically in connection with the construction of Hecke eigensheaves. Indeed the chiral Hecke algebra $A_k(\g)$ is defined as a certain vertex algebra extension of the simple affine vertex algebra $V_k(\g)$ at level $k \in \Z_{\geq 0}$, and Hecke eigensheaves with eigenvalue specified by a local system $E$ are realised in the higher chiral homology groups of a twist of $A_k(\g)$ by $E$. In this context Beilinson and Drinfeld ask {\cite[4.9.10]{BD.Chiral}} whether the higher chiral homology groups of $V_k(\g)$ vanish.

Following the seminal work of Zhu \cite{Z96}, spaces of conformal blocks of elliptic curves, and especially their behaviour over families degenerating to a nodal curve, have come to play an important role in the representation theory of vertex algebras. Let $X_q$ denote the elliptic curve $\C / \Z + \Z\tau$ where $q = e^{2\pi i \tau}$, let $V$ be a quasiconformal vertex algebra and let $\CA_V$ denote the associated chiral algebra on $X_q$. The space of conformal blocks is given by
\begin{align}\label{cb.quot.intro}
H_0^{\text{ch}}(X_q, \CA_V) \cong \frac{V}{V_{(0)}V + V_{(\wp)}V},
\end{align}
where $\wp = \wp(z, q)$ is the Weierstrass elliptic function. Zhu observed that the specialisation of (\ref{cb.quot.intro}) at $q = 0$ recovers the zeroth Hochschild homology of a certain associative algebra $\zhu(V)$ whose representation theory is strongly connected with that of $V$. 
The modular nature of graded dimensions of regular vertex algebras is ultimately explained by these facts.

The main results of this paper concern the relationship between chiral homology of elliptic curves and Hochschild homology of the Zhu algebra in higher degrees. We specialise Beilinson-Drinfeld's general definition of chiral homology to develop explicit complexes computing the chiral homology groups of the elliptic curve $X_q$ with coefficients in $\CA_V$. We then focus attention on $H_i^{\text{ch}}(X_q, \CA_V)$ for $i=0,1$, deriving a small three term complex $A^\bullet(q)$ which computes these groups in Section \ref{sec:ch.hom.ell}, specifically Proposition \ref{prop:chhom.equals.A}. We recover the presentation (\ref{cb.quot.intro}) of $H_0^{\text{ch}}$ and similarly present $H_1^{\text{ch}}$ as an explicit subquotient of a sum of tensor powers of $V$ involving residues of elliptic functions.

We then use $A^\bullet(q)$ to analyse the behaviour of $H_1^{\text{ch}}$ in the $q \rightarrow 0$ limit. What occurs is roughly speaking the following: a large subcomplex $B^\bullet$ decouples from $A^\bullet(q)$ in the $q \rightarrow 0$ limit, and the quotient $A^\bullet(0)/B^\bullet$ computes the Hochschild homology of $\zhu(V)$. This remarkable decoupling is ultimately due to the elliptic function identity
\begin{equation} \label{eq:heat}
8 \pi^2 q \frac{d\zeta}{dq} = 2 \zeta \wp + \wp',
\end{equation}
which is a consequence of the heat equation for theta functions.

Before stating the main result we recall the notions of singular support and associated scheme of a vertex algebra. Li introduced a filtration $F$ on a vertex algebra $V$ whose associated graded $\gr^F(V)$ is naturally a $\Z_{\geq 0}$-graded commutative algebra (indeed a Poisson vertex algebra) \cite{Lifilt}. The spectrum of $\gr^F(V)$ is known as the singular support of $V$ and is denoted $\sing(V)$. The algebra $\gr^F(V)$ is generated by its component of degree $0$, which is just Zhu's $C_2$-algebra, and whose spectrum is known as the associated scheme $X_V$ of $V$. Thus there is a natural embedding $\sing(V) \hookrightarrow J{X_V}$, where in general $JY$ denotes the arc space of the scheme $Y$.
\begin{thm}\label{H1.thm.intro}
Let $V$ be a finitely strongly generated quasiconformal vertex algebra. If the natural embedding of the singular support of $V$ into the arc space of its associated scheme is an isomorphism, i.e., if
\begin{align}\label{SS.condition}
\sing(V) \cong J{X_V},
\end{align}
then
\begin{align}\label{H1.isom.intro}
\lim_{q \rightarrow 0} H^{\text{ch}}_1(X_{q}, \CA_V) \cong \Hoch_1(\zhu(V)).
\end{align}
\end{thm}
The meaning of the left-hand side of (\ref{H1.isom.intro}) is to be understood in terms of the specialisation to $q=0$ of the complex $A^\bullet(q)$. The precise statement is given as Theorem \ref{thm:main-theorem} below. It remains to discuss for which vertex algebras $V$ the condition (\ref{SS.condition}) holds.

The condition (\ref{SS.condition}) is satisfied for the Heisenberg, universal Virasoro and universal affine vertex algebras, which are all universal enveloping vertex algebras, and is also satisfied for the universal $W$-algebra $W^k(\g, f)$. 

For the class of Virasoro minimal models $\vir_{p, p'}$, which are rational vertex algebras, we show that (\ref{SS.condition}) holds in some cases, namely $(p, p') = (2, 2k+1)$, but demonstrably fails to hold in other cases. Indeed for $V = \vir_{2, 2k+1}$ the Hilbert series of both sides of (\ref{SS.condition}) coincide with the function
\[
\prod_{\substack{m \geq 1, m \not\equiv 0, \pm 1 \\ \mod(2k+1)}} \frac{1}{1-q^m}
\]
which famously appears in Gordon's generalisation of the Rogers-Ramanujan identity. In Section \ref{sec:examples} we prove the following
\begin{thm*}[\ref{thm:minimal}] Let $V = \vir_{p,p'}$ be the Virasoro minimal model with central charge 
\[
c = c_{p, p'} = 1 - 6 \frac{(p-p')^2}{pp'}.
\]
If $(p,p')=(2,2k+1)$ where $k \geq 1$, then the natural embedding $\sing(V) \hookrightarrow JX_V$ is an isomorphism of schemes. In particular
\begin{align*}
\lim_{q \rightarrow 0} H^{\text{ch}}_1(X_{q}, \CA_V) = 0.
\end{align*}
If $p,p' \geq 3$ then the embedding is not an isomorphism. In all cases, however, the reduced schemes of $\sing(V)$ and $JX_V$ are isomorphic, both consisting of a single closed point.
\end{thm*}
Regarding simple affine vertex algebras we prove the following result.
\begin{thm*}[\ref{thm:affine}]
Let $V$ be the simple affine vertex algebra $V_k(\mathfrak{sl}_2)$ at positive integral level $k \in \mathbb{Z}_{+}$. The natural embedding $\sing(V) \hookrightarrow JX_V$ is an isomorphism of schemes. In particular
\begin{align*}
\lim_{q \rightarrow 0} H^{\text{ch}}_1(X_{q}, \CA_V) = 0.
\end{align*}
\end{thm*}
Our proof of this theorem uses a result of Meurman-Primc \cite{MPbook} on PBW-type bases of integrable $\widehat{\mathfrak{sl}}_2$-modules.

In general the question of for which vertex algebras $V$ the condition (\ref{SS.condition}) holds appears to be a subtle and interesting one. We find it particularly interesting that Rogers-Ramanujan type $q$-series identities appear in the comparison of the graded dimension of $V$ with that of $JR_V$. 

\end{nolabel}


\begin{nolabel}
We briefly explain why condition \eqref{SS.condition} appears in Theorem \ref{H1.thm.intro}; the key point is the following vanishing condition on Koszul homology of arc spaces. Let $A$ be a commutative $\C$-algebra and let $\tau : \Omega^1_{A/\C} \rightarrow A$ be a derivation. We have the associated Koszul complex $K_\bullet = K_\bullet^A$ defined by $K_\bullet^A := \Sym \Omega^1_{A/\C}[1]$. Now let $A^0$ be a commutative algebra of finite type, $X=\Spec A^0$ and let $JA^0$ be the coordinate ring of the arc space $JX$, that is $\Spec JA^0 = JX$. The algebra $JA^0$ comes equipped with a canonical derivation and it turns out that the corresponding complex $K_\bullet^{JA^0}$ is acyclic away from degree $0$. In Section \ref{sec:hoch} we prove the following converse
\begin{thm*}[\ref{thm:ssjets}]
Let $A = \bigoplus_{n \in \Z_{\geq 0}} A^n$ be a $\mathbb{Z}_{+}$-graded commutative algebra with a derivation $\tau$ of degree $+1$, and let $(K^A_\bullet, \iota_\tau)$ be the Koszul complex associated with $A$ as above. We assume $A$ is generated by $A^0$ as a differential algebra, and that $A^0$ is an algebra of finite type. Then $H_{-1}(K_\bullet^A, \iota_\tau) = 0$  if and only if $A \cong JA^0$.
\end{thm*}
The subcomplex $B^\bullet \subset A^\bullet(q=0)$ mentioned above acquires a filtration induced by the Li filtration on $V$. We identify the associated graded complex $\gr^F(B^\bullet)$ with the Koszul complex $(K^{\gr^F(V)}_\bullet, \iota_T)$ associated as above with the singular support of $V$. Thus condition \eqref{SS.condition} arises naturally as a sufficient condition for vanishing of $H^{-1}(B^\bullet)$. 
\label{no:reason-condition}
\end{nolabel}

\begin{nolabel}
Although relatively short, this article uses a variety of somewhat involved tools from different fields. Here is a brief summary.
\begin{enumerate}[label=\alph*]
\item
We use the theory of chiral algebras and chiral homology as developed in \cite{BD.Chiral}. We obtain the complex $A^\bullet(q)$ presented in Section \ref{sec:ch.hom.ell}, which computes chiral homology of the elliptic curve $X_q$ in low degrees, by a translation to linear algebra of Beilinson-Drinfeld's complex for chiral homology with supports. The main ingredients here are the use of the marked point (the polarisation) of the elliptic curve and a well-defined \'{e}tale coordinate. The support is taken to be the marked point, with the vacuum module insertion.
\item We use Totaro's theorem \cite{Totaro} on cohomology rings of configuration spaces of a manifold $X$. This is explained in Section \ref{sec:totaro}. We also use explicit representatives of cohomology classes in the case of $X=X_q$ an elliptic curve. These are described in terms of elliptic functions.
\item We use elliptic function identities, such as \eqref{eq:heat}, to study the behaviour of $A^\bullet(q)$ in the nodal curve limit $q \rightarrow 0$, specifically to find the subcomplex $B^\bullet \subset A^\bullet({q=0})$. 
\item We use Li's filtration \cite{Lifilt} to endow $B^\bullet$ with a filtration and we consider the corresponding spectral sequence. The associated graded of $B^\bullet$ is identified with the Koszul complex of the singular support of $V$ as explained above. 
\item We use a description \cite{Zhu.Note} of the multiplication in the Zhu algebra $\zhu(V)$ as a nodal curve limit $q \rightarrow 0$ of the operation $a_{(\zeta)}b := \res_z \zeta(z,q) a(z)b$. In this way we identify the quotient $A^\bullet(q=0)/B^\bullet$ with the bar complex computing the Hochschild homology of $\zhu(V)$. 
\item We use a result of Bruschek, Mourtada and Schepers \cite{mourtada} on the Hilbert series of certain arc spaces, and their relation with Rogers--Ramanujan identities, to identify which minimal models $\vir_{p,p'}$ satisfy condition \eqref{SS.condition} and which ones do not.
\item We use Gr\"{o}bner basis techniques, and a result of Meurman and Primc
\cite{MPbook}, to compute the Hilbert series of the arc space of the associated
scheme of the simple affine vertex algebra $V_k(\mathfrak{sl}_2)$.
\end{enumerate}
\label{no:list-tools}
\end{nolabel}

\begin{nolabel}
Although the ultimate objective of the techniques proposed in this article is to answer and generalise Beilinson and Drinfeld's question {\cite[4.9.10]{BD.Chiral}}, the article is written in such a way that readers better acquainted with the theory of vertex algebras than that of chiral algebras may understand the main statements and proofs. The complex constructed in Section \ref{sec:ch.hom.ell} does not use the chiral algebra formalism and is canonically associated to a vertex algebra and an elliptic curve. Of course in order to check that this complex indeed computes chiral homology, one needs to compare with Beilinson and Drinfeld's construction. The reader acquainted with chapter $4$ of \cite{BD.Chiral} or chapter $20$ of \cite{FBZ.Book} will find this comparison self-evident.

We use Weierstrass's $\wp$-function and its integral $\zeta$, and the differential equation \eqref{eq:heat} and some explicit algebraic equations satisfied by these functions in Section \ref{sec:elliptic.functions} as well as their properties as $q \rightarrow 0$. It would be interesting to have a purely algebro-geometric formulation of our results here in terms of sections of $\mathcal{O}_X(2)$ instead of elliptic functions and differential equations. 
\label{no:compare-BD}
\end{nolabel}

\begin{nolabel}
Here are some subjects that are not treated on this article and that we plan to address in the following articles of this series.
\begin{enumerate}
\item Chiral homology of the general elliptic curve $X_q$, not only in the limit $q \rightarrow 0$. Zhu's technique for handling $H^{\text{ch}}_0$ is based on the construction of explicit representatives of homology classes, as traces on $V$-modules. These representatives are shown to be flat sections of the connection on the space of conformal blocks, i.e., Zhu shows that the trace functions satisfy certain explicit differential equations. Beilinson and Drinfeld endow their chiral complexes with Gauss-Manin connections along the moduli space of curves, hence we have at our disposal differential equations generalising those considered by Zhu. What remains is to develop higher analogues of \emph{traces of modules}. We expect that extensions of $V$-modules provide examples of non-trivial classes in chiral homology. 
\item Chiral homology in degree $2$ and higher. To compute higher chiral homologies we use an explicit description of the de Rham cohomology classes of configuration spaces of an elliptic curve. It is surprising that until quite recently even the Betti numbers were not available (see for example \cite{knudsen,maguire,sch16}). It is relatively easy to find a quotient complex of the chiral chain complex isomorphic to the bar complex of $\zhu(V)$. However proving acyclicity of the kernel becomes more difficult due to the more involved combinatorics of spaces of configurations of more than $3$ points.
\item Chiral homology of higher genus curves. The sheaves of conformal blocks (at least in the rational case) can be extended to the boundary of the moduli spaces \cite{tsuchiya.ref}. We are not aware of a similar result for chiral homology. With the tools we develop in hand we expect to be able to reduce the computation of chiral homologies of vertex algebras to the case of elliptic curves.
\end{enumerate}
\label{no:project}
\end{nolabel}

\begin{rem}
The original question of Beilinson and Drinfeld for affine Kac-Moody vertex algebras at integral level has been answered positively by Dennis Gaitsgory \cite{gaitsgory-private} using a theorem of Teleman about the geometry of the affine Grassmannian. In the case of the universal affine Kac-Moody algebra, Sam Raskin is able to show that (the $q=0$ limit of) chiral homology coincides with Hochschild homology by studying directly the Beilinson-Drinfeld Grassmanian over the nodal curve \cite{sam-private}. 

\label{rem:known-results}
\end{rem}

\begin{nolabel}{\em Acknowledgements.} The authors would like to thank T. Arakawa, D. Gaitsgory, S. Kanade, A. Moreau and S. Raskin for useful remarks. We would especially like to express our gratitude to S. Kanade for drawing our attention to the work of Meurman and Primc \cite{MPbook}, which is crucial to the proof of Theorem \ref{thm:affine}. We would also like to thank the anonymous referees for very detailed and constructive remarks that greatly improved the exposition, especially in Section \ref{sec:ch.hom.ell}. JvE was supported by CNPq grants 409582/2016-0 and 303806/2017-6 and by the Serrapilheira Institute (grant number Serra – 1912-31433). RH was supported by CNPq grants 409582/2016-0 and 305688/2019-7.
\end{nolabel}


\section{Vertex Algebras}\label{sec:va}

\begin{nolabel}
For background we refer to the text \cite{Kac.VA.Book}. The formal delta function is defined to be $\delta(z, w) = \sum_{n \in \Z} z^{-n-1} w^n$, and the formal residue of the power series $f(z) = \sum_{n \in \Z} f_n z^n$ is defined by $\res_z f(z) = f_{-1}$. The symbol $i_{z,w}$, resp. $i_{w,z}$, denotes expansion of an element of $\C[[z, w]][z^{-1}, w, w^{-1}, (z-w)^{-1}]$ as a Laurent series in $w$, resp. $z$. For $A$ an endomorphism of a vector space, we shall generally use the notation $A^{(j)}$ for $A^j / j!$.
\end{nolabel}

\begin{nolabel}
A vertex algebra is a vector space $V$ equipped with a vacuum vector $\vac$ and a collection of bilinear products indexed by integers. The $n^{\text{th}}$ such product of $a, b \in V$ is denoted $a{(n)}b$. These products are to satisfy the quantum field property
\begin{align*}
\text{$a(n)b = 0$ for $n \gg 0$}, 
\end{align*}
the Borcherds identity
\begin{align} \label{Borcherds.identity}
\begin{split}
&\sum_{j \in \Z_{\geq 0}} \binom{m}{j} \left(a(n+j)b\right)(m+k-j)c \\
 & \phantom{.............} = \sum_{j \in \Z_{\geq 0}} (-1)^j \binom{n}{j} \left[ a(m+n-j)b(k+j)c - (-1)^n b(n+k-j)a(m+j)c \right]
\end{split}
\end{align}
for all $a, b, c \in V$ and $m, n, k \in \Z$, and the unit identity
\begin{align}\label{unit.identity}
\vac_{(n)}a = \delta_{n, -1}a \quad \text{for $n \in \Z$} \quad \text{and} \quad a_{(n)}\vac = \delta_{n, -1} a \quad \text{for $n \in \Z_{\geq -1}$}.
\end{align}
It is customary to associate with $a \in V$ its \emph{quantum field}
\[
Y(a, z) = a(z) = \sum_{n \in \Z} z^{-n-1} a(n) \in \en({V})[[z, z^{-1}]].
\]
The \emph{translation operator} $T \in \en({V})$ is defined by $Ta = a(-2)\vac$.

When written in terms of quantum fields the Borcherds identity becomes the following Jacobi identity
\begin{align}\label{VAJacobi.identity}
[a(x)b](w)c = \res_z \left(a(z)b(w)c \,i_{z,w} - b(w)a(z)c \,i_{w,z}\right) \delta(x, z-w).
\end{align}

The following are some useful consequences of the definitions: The translation invariance condition
\begin{align}\label{translation.identity}
[Ta](z) = \partial_z a(z), \quad \text{equivalently} \quad [Ta](n) = -na(n-1),
\end{align}
the skew-symmetry formula
\begin{align}\label{skewsymmetry.identity}
b(z)a = e^{zT}a(-z)b, \quad \text{equivalently} \quad b{(n)}a = -\sum_{j \in \Z_{\geq 0}} (-1)^{n+j} T^{(j)}(a{(n+j)}b),
\end{align}
and the commutator formula
\begin{align}\label{commutator.identity.1}
[a(z), b(w)] &= \sum_{j \in \Z_{\geq 0}} [a(j)b](w) \partial_w^{(j)}\delta(z, w), \\
\text{equivalently} \quad [a(m), b(n)] &= \sum_{j \in \Z_{\geq 0}} \binom{m}{j} \left[a(j)b\right](m+n-j).\label{commutator.identity.2}
\end{align}
\end{nolabel}

\begin{nolabel}
For later convenience we recall the $f$-product notation
\begin{align*}
a_{(f)}b = \res_z f(z) a(z)b.
\end{align*}
When necessary we abuse this notation, writing for example $a_{(x f(x))}b$ rather than introducing $g(x) = xf(x)$ and writing $a_{(g)}b$. As a consequence of the skew-symmetry identity we have the following lemma.
\end{nolabel}

\begin{lem}\label{lem:skew-symmetry}
Let $f \in \C((x))$. Then
\begin{equation}
a_{(f(x))}b + b_{(f(-x))}a = \sum_{j \in \Z_{\geq 0}} \frac{(-1)^j}{(j+1)!} T^{j+1}\left(a_{(x^{j+1} f(x))}b\right), \qquad \text{for all $a, b \in V$}.
\label{eq:lemma-sym}
\end{equation}
\end{lem}
\begin{nolabel}
Motivated by this identity we introduce the symbol
\begin{align}\label{Borcherds.product}
\int{\{a_{(f)}b\}} = \sum_{j \in \Z_{\geq 0}}\frac{(-1)^j}{(j+1)!} T^{j}(a_{(x^{j+1} f(x))}b).
\end{align}
We remark that the $f(x) = x^n$ case of this product already appeared in Borcherds' paper \cite{B86} where it was denoted $a \boldsymbol{\times}_n b$.
\end{nolabel}

\begin{nolabel}\label{sec:c2.def}
Let $V$ be a vertex algebra. The quotient $V / TV$ is well known to carry the structure of a Lie algebra with bracket given by $[a, b] = a(0)b$. The quotient $R_V = V / V(-2)V$ is known as Zhu's $C_2$-algebra. It is naturally a Poisson algebra with the commutative product given by $a \cdot b = a(-1)b$ and the Poisson bracket by $\{a, b\} = a(0)b$.
\end{nolabel}

\begin{nolabel}
Following \cite{FBZ.Book} we define the topological algebras $\OO = \C[[t]]$ and
$\CK = \C((t))$, and we consider the Lie algebras $\Der_0{\OO} \subset \Der{\OO} \subset \Der{\CK}$, defined by $\Der_0{\OO} = \bigoplus_{n \geq 0} \C L_{n}$, $\Der{\OO} = \bigoplus_{n \geq -1} \C L_{n}$ and $\Der{\CK} = \bigoplus_{n \in \Z} \C L_{n}$ where $L_n = -t^{n+1}\partial_t$. The group $\Aut{\OO}$ of continuous algebra automorphisms of $\OO$ has Lie algebra $\Der_0{\OO}$.

The Virasoro algebra $\vir = \Der{\CK} \oplus \C C$ is the universal central extension of $\Der{\CK}$. Explicitly
\[
[L_m, L_n] = (m-n) L_{m+n} + \frac{m^3-m}{12} \delta_{m, -n} C.
\]

A conformal vector of a vertex algebra $V$ is a vector $\om \in V$ whose associated quantum field $L(z) = \om(z) = \sum L_n z^{-n-2}$ furnishes $V$ with a representation of $\vir$ such that
\begin{itemize}
\item $C$ acts by a constant, called the central charge of $V$,

\item $L_0$ acts semisimply on $V$ with non-negative integral eigenvalues and finite dimensional eigenspaces,

\item $L_{-1}$ coincides with $T$.
\end{itemize}
A conformal vertex algebra is a vertex algebra together with a choice of conformal vector. A quasiconformal vertex algebra is a vertex algebra furnished with a representation of $\Der{\OO}$ such that
\begin{itemize}
\item For all $b \in V$ one has
\[
[L_m, b(n)] = \sum_{j \geq 0} \binom{m+1}{j} [L_{j-1}b](m+n+1-j)
\]
(cf. equation (\ref{commutator.identity.2})),

\item $L_0$ acts semisimply on $V$ with non negative integral eigenvalues, and $\Der_{>0}{\OO}$ acts locally nilpotently on $V$,

\item $L_{-1}$ coincides with $T$.
\end{itemize}
Let $V$ be a (quasi)conformal vertex algebra. We denote by $V_n$ the eigenspace of $L_0$ with eigenvalue $n$, and we say that elements of $V_n$ have \emph{conformal weight} $n$. Occasionally we denote by $\D(a)$ the conformal weight of an eigenvector $a$ of $L_0$.

The definition of (weak) $V$-module as a vector space $M$ equipped with the action of quantum fields $Y^M(a, z) \in \en(M)[[z, z^{-1}]]$ for all $a \in V$, satisfying appropriate analogues of the Borcherds identity and other axioms, is standard. A positive energy module over a quasiconformal vertex algebra $V$ is one equipped with a compatible grading by finite dimensional eigenspaces of $L_0$ with eigenvalues bounded below.
\end{nolabel}

\begin{nolabel}
We recall the definitions of the enveloping algebra $U(V)$ and the Zhu algebra $A(V)$ of a quasiconformal vertex algebra $V$. Firstly one defines the Lie algebra
\begin{align*}
\Lie(V) = V[t, t^{-1}] / (T + \partial_t)V[t, t^{-1}]
\end{align*}
with the Lie bracket $[at^m, bt^n] = \sum_{j \in \Z_{\geq 0}} \binom{m}{j} \left(a(j)b\right)t^{m+n-j}$. Next one equips $\Lie(V)$ with a $\Z$-grading by putting $\deg(at^m) = m+1-\D(a)$ and extends the grading to the universal enveloping algebra $U(\Lie(V))$. Using the grading it is possible to form a degreewise completion of $U(\Lie(V))$ in which the equality (\ref{Borcherds.identity}) of infinite sums makes sense. Finally one obtains the topological algebra $U(V)$ as the quotient of this completion by the relations (\ref{Borcherds.identity}) and (\ref{unit.identity}). The category of $V$-modules is naturally equivalent to the category of smooth $U(V)$-modules.

Now let $M$ be a $V$-module. Then the space of invariants $M^{U(V)_{>0}}$ is immediately seen to carry an action of the algebra
\begin{align*}
A(V) = U(V)_0 / (U(V) \cdot U(V)_{>0})_0.
\end{align*}
Zhu proved that the functor $M \mapsto M^{U(V)_{>0}}$ from the category of $V$-modules to the category of $A(V)$-modules induces a bijection between the sets of isomorphism classes of irreducible $A(V)$-modules and irreducible positive energy $V$-modules.

In fact Zhu introduced $A(V)$ in terms of the following very different presentation \cite{Z96}. For any $a, b \in V$ put
\begin{align*}
a \circ b = \res_w w^{-2} (1+w)^{\D(a)} a(w)b \,dw \quad \text{and} \quad
a * b = \res_w w^{-1} (1+w)^{\D(a)} a(w)b \,dw.
\end{align*}
Then $V \circ V \subset V$ turns out to be an ideal with respect to the operation $*$, which in turn descends to an associative product on the quotient $V / V \circ V$. Furthermore $A(V) \cong V / V \circ V$ as associative algebras.
\end{nolabel}

\begin{nolabel}\label{subsec.OO.def}
Geometrically $\OO$ represents the algebra of functions defined on the \emph{standard disc} $D = \Spec{\OO}$. A generator of the unique maximal ideal $t \C[[t]] \subset \OO$ represents a coordinate on $D$ vanishing at its unique closed point. Let us denote by $\coord$ the set of coordinates on $D$; this set is naturally an $\Aut\OO$-torsor and may be identified with $\C^\times t + t^2\C[[t]]$. Thus for any $f(t) \in \C^\times t + t^2\C[[t]]$ there exists a unique element $g \in \Aut\OO$ such that $g \cdot t = f(t)$. Introducing a second formal parameter $z$, we write $f_z(t) = f(z+t)-f(z) \in \C[[z, t]]$. 

Now let $V$ be a quasiconformal vertex algebra. The restriction of the $\Der{\OO}$-action on $V$ to $\Der_0{\OO}$ can be exponentiated to define an action of $\Aut\OO$. Given $f(t) \in \C^\times t + t^2\C[[t]]$ and $g \in \Aut{\OO}$ as above, we denote by $R(f)$ the image of $g$ under the map $\Aut\OO \rightarrow \en(V)$. The behaviour of the vertex operation $Y$ in $V$ under change of coordinate $z$ is governed by Huang's formula {\cite[Section 7.4]{Huang.2dCFT.Book}}
\begin{align}\label{huang.formula}
Y(a, z) = R(f) Y(R(f_z)^{-1}a, f(z)) R(f)^{-1}.
\end{align}
In \cite{FBZ.Book} Frenkel and Ben-Zvi construct chiral algebras from quasiconformal vertex algebras, and in this construction the formula (\ref{huang.formula}) plays a key role. We review the construction in Section \ref{FBZconstruction} below.



\end{nolabel}

\begin{nolabel}\label{huang.isom}
Let $V$ be a quasiconformal vertex algebra and let $\phi(z)$ denote the formal power series $e^{2\pi i z}-1$. In \cite{Z96} Zhu introduced the modified vertex operation $Y[-, z]$ on $V$ defined by
\begin{align*}
Y[a,z] = Y(e^{2\pi i z L_0}a, \phi(z)).
\end{align*}
Let us write $R = R(\phi)$. Then since $\phi_z(t) = e^{2\pi i z}\phi(t)$ we have $R(\phi_z) = R e^{-2\pi i z L_0}$, and (\ref{huang.formula}) yields
\begin{align}\label{eq:R.VA.isom}
Y[a,z] = R^{-1}Y(Ra,z)R.
\end{align}
So in fact $R : (V, Y[-,z]) \rightarrow (V, Y(-,z))$ is a vertex algebra isomorphism. 
This was used by Huang to uncover the following presentation of $A(V)$ \cite{Huang.duality}. Put
\begin{align}\label{f.and.g.def}
\begin{split}
f(z) &= 2\pi i \cdot \frac{e^{2\pi i z}}{e^{2\pi i z}-1} = z^{-1} + \pi i - \frac{\pi^2}{3}z - \frac{\pi^4}{45}z^3 \cdots, \\
g(z) &= (2\pi i)^2 \cdot \frac{e^{2\pi i z}}{(e^{2\pi i z}-1)^2} = z^{-2} + \frac{\pi^2}{3} + \frac{\pi^4}{15}z^2 + \cdots,
\end{split}
\end{align}
and
\begin{align*}
\zhu(V) = V / V_{(g)}V \quad \text{with product} \quad a b = a_{(f)}b.
\end{align*}
Then $\zhu(V)$ is an associative unital algebra isomorphic to $A(V)$. Indeed from (\ref{eq:R.VA.isom}) above we have $(Ra)_{(f)}(Rb) = R(a * b)$ and $(Ra)_{(g)}(Rb) = R(a \circ b)$, from which it follows that the linear isomorphism $R : V \rightarrow V$ descends to an isomorphism of algebras $A(V) \rightarrow \zhu(V)$. Before closing this section we record the formula \cite[equations (3.11) and (3.15)]{Zhu.Note}.
\begin{align}\label{eq:Zhu.commutator}
a_{(f)}b - b_{(f)}a = 2\pi i \, a(0)b \pmod{V_{(g)}V},
\end{align}
which will be used later in the paper. 
\end{nolabel}

\section{Chiral Algebras}\label{sec:chia}
\begin{nolabel}\label{subsec:Dmod}
First we recall some background material on $\CD$-modules \cite{Hotta.Book}. Let $X$ be a smooth complex (analytic or algebraic) variety. We denote by $\OO_X$ the structure sheaf of $X$, by $\Theta_X$ and $\Om_X$ the tangent and cotangent sheaves, by $\om_X$ the canonical sheaf (which is isomorphic to $\wedge^n\Om_X$) and by $\CD_X$ the sheaf of differential operators. The sheaf $\OO_X$ is a left $\CD_X$-module essentially by definition, while the Lie derivative
\begin{align}\label{eq:Lie.derivative}
\Lie_\tau (f_0 \cdot df_1 \wedge \ldots \wedge df_k) = \left( \tau f_0 \right) \cdot df_1 \wedge \ldots \wedge df_k + \sum_{i=1}^k f_0 \cdot df_1 \wedge \ldots \wedge d(\tau f_i) \wedge \ldots \wedge df_k
\end{align}
defines a right $\CD_X$-module structure on $\om_X$. For $\tau$ a vector field and $\nu$ a section of $\om_X$ this action is given by $\nu \cdot \tau = -\Lie_\tau(\nu)$. The categories of left and right $\CD_X$-modules are equivalent via the following side-changing operation: for $M$ a left $\CD_X$-module the corresponding right $\CD_X$-module is $M \otimes_{\OO_X} \om_X$ equipped with
\begin{align*}
(m \otimes \nu) \cdot \tau = -(\tau \cdot m) \otimes \nu + m \otimes (\nu \cdot \tau).
\end{align*}
We denote by $f^\bullet$ and $f_{\bullet}$ the sheaf theoretic pullback and pushforward along a morphism $f : X \rightarrow Y$. The pullback and pushforward of $\CD$-modules along $f$ is in general effected by the transfer bimodule $\CD_f$, which is the $(\CD_X, f^{\bullet}\CD_Y)$-bimodule over $X$ defined by
\begin{align*}
\CD_{f} = \OO_X \otimes_{f^{\bullet}\OO_Y} f^{\bullet}\CD_Y.
\end{align*}
The right action is tautological and the left action of $\CD_X$ is given in local coordinates by the formula
\begin{align*}
\tau(g \otimes s) = \tau(g) \otimes s + g \sum_{i=1}^{\dim(Y)} \tau(y_i \circ f) \otimes \frac{\partial s}{\partial{y_i}}.
\end{align*}
The pullback of the left $\CD_Y$-module $M$ is
\begin{align*}
f^*(M) = \CD_{f} \otimes_{f^{\bullet}\CD_Y} f^{\bullet}(M),
\end{align*}
and the functor $f^*$ is right exact. The pushforward of the right $\CD_X$-module $M$ is
\begin{align*}
f_*(M) = f_{\bullet}(M \otimes_{\CD_X} \CD_{f}).
\end{align*}
Since it involves the composition of a right exact functor with a left exact one, the functor $f_*$ is neither right nor left exact in general and the definition of its derived functors in general requires derived categories.

The pushforward along the closed embedding of a smooth subvariety has the following local description. Let $i : X \rightarrow Y$ be a closed embedding whose image is defined locally by the equations $y_i = 0$ for $i=k+1,\ldots,\dim(Y)$ and let $M$ be a right $\CD_X$-module. Then
\begin{align*}
i_*(M) = i_\bullet(M) \otimes_\C \C[\partial_{y_{k+1}}, \ldots, \partial_{y_{\dim(Y)}}]
\end{align*}
with the action of $\partial_{y_i}$ given by $\partial_{x_i} \otimes 1$ for $i=1,\ldots, k$, and by $1 \otimes \partial_{y_i}$ for $i=k+1,\ldots,\dim(Y)$, and the action of functions by restriction to $i(X)$. The most important example is the diagonal embedding $\D : X \rightarrow X^2$. Then
\begin{align}\label{eq:diagonal.pushforward}
\D_*(M) = \D_\bullet(M) \otimes_{\C[\partial]} \C[\partial_1, \partial_2]
\end{align}
where the action of $\C[\partial]$ on $\C[\partial_1, \partial_2]$ is by $\partial = \partial_1+\partial_2$.

The other operation we require is the $*X$ construction {\cite[p. 97]{Grothendieck-dR}}. For $X$ a divisor in $Y$ and $\CF$ an $\OO_Y$-module, the $\OO_Y$-module $\CF(*X)$ is defined to be
\begin{align*}
\CF(*X) = \lim_{\rightarrow}{\Hom_{\OO_Y}(I^n, \CF)},
\end{align*}
where $I$ is an ideal of definition of $X$. This construction is independent of the choices made, and works in both the algebraic and analytic categories. In the algebraic category we have $M(*X) \cong j_*j^*M$ and so, following \cite{BD.Chiral}, we use the latter notation. In Section \ref{sec:ch.hom.ell} we pass from the algebraic to analytic context, and it pays to note that the image under the analytification functor of $j_*j^*\CF$ is $\CF(*X)$ and not $(j^{\text{an}})_*(j^{\text{an}})^*\CF$, which is very much larger.
\end{nolabel}

\begin{nolabel}\label{no:general.diagonal.def}
Next we recall the key notion of chiral operation from \cite{BD.Chiral}. Let $I$ be a finite set and let $X$ be a smooth complex algebraic curve. The open embedding into $X^I$ of
\begin{align}\label{eq:standard.open}
U^{(I)} = \{\left(x_i\right)_{i \in I} \in X^I | \text{$x_{i_1} \neq x_{i_2}$ whenever $i_1 \neq i_2$}\}
\end{align}
is denoted $j^{(I)}$, and the closed embedding into $X^I$ of the diagonal
\begin{align*}
\D^{(I)} = \{\left(x_i=x\right)_{i \in I} | x \in X \}.
\end{align*}
is denoted $\D^{(I)}$.

Now let $\{L_i\}_{i \in I}$ be a collection of right $\CD_X$-modules parametrised by the finite set $I$, let $M$ be another right $\CD_X$-module, and write $j = j^{(I)}$ and $\D = \D^{(I)}$. 
A \emph{chiral $I$-operation} {\cite[3.1]{BD.Chiral}} from $\{L_i\}_{i \in I}$ to $M$ is by definition a morphism of right $\CD_{X^I}$-modules
\begin{align}\label{ch.I.op}
j_*j^*\left(\boxtimes_{i \in I} L_i\right) \rightarrow \D_* M.
\end{align}
The vector space of chiral $I$-operations (\ref{ch.I.op}) is denoted $P^{\text{ch}}(\{L_i\}_{i \in I}, M)$.

The Grothendieck residue morphism is defined, for any right $\CD_X$-module $M$, to be the natural surjection
\begin{align}\label{res.def}
\res : j_*j^*(\om_X\boxtimes M) \rightarrow \frac{j_*j^*(\om_X\boxtimes M)}{\om_X\boxtimes M} \cong \D_*(M).
\end{align}
This is the archetypal example of a chiral operation, i.e., $\res \in P^{\text{ch}}(\{\om_X, M\}, M)$.
\end{nolabel}

\begin{nolabel}
Chiral operations may be composed and in this way the category of right $\CD_X$-modules becomes what is known as a pseudo-tensor category {\cite[1.1]{BD.Chiral}}. In such a category, which generalises the notion of symmetric monoidal category, it is possible to define the notion of algebra over an operad. Thus a (non-unital) chiral algebra over $X$ is formally defined to be an algebra over the Lie operad in the pseudo-tensor category of right $\CD_X$-modules and chiral operations {\cite[Section 3.3]{BD.Chiral}}. In more concrete terms a chiral algebra over $X$ is a right $\CD_X$-module $\CA$ together with a chiral operation $\mu \in P^{\text{ch}}(\{\CA, \CA\}, \CA)$ satisfying analogues of the usual skew-symmetry and Jacobi identities. This is spelled out in Section \ref{sec:chevalley-cousin} below.

\end{nolabel}

\begin{nolabel}\label{FBZconstruction}
Let $X$ be a smooth complex algebraic curve and $V$ a quasiconformal vertex algebra. We review the construction from $V$ of a chiral algebra over $X$ \cite{FBZ.Book} (see also \cite{Zhu.global}). A coordinate at a closed $\C$-point $x \in X$ is, as in Section \ref{subsec.OO.def}, a generator of the unique maximal ideal in the local ring $\OO_{x} \cong \OO$. The set of pairs $(x, t_x)$ consisting of a point $x \in X$ and a coordinate $t_x$ at $x$ is the set of $\C$-points of a scheme $\coord_X$ which is, furthermore, an $\Aut{\OO}$-torsor over $X$. Applying the associated bundle construction to $\coord_X$ and the $\Aut{\OO}$-module $V$ yields the vector bundle
\begin{align*}
\CV = \coord_X \times_{\Aut{\OO}} V
\end{align*}
over $X$. The bundle $\CV$ carries a connection $\nabla : \CV \rightarrow \CV \otimes \Om_X$, thus a left $\CD_X$-module structure. The connection is defined relative to a choice of local coordinate by
\begin{align*}
\nabla_{\partial_z} = \partial_z + T,
\end{align*}
but is independent of this choice. The right $\CD_X$-module obtained from $\CV$ by side-changing is denoted $\CA$. It carries the natural structure of a chiral algebra which we now recall {\cite[Theorem 19.3.3]{FBZ.Book}}.

For $x \in X$ one denote by $D_x = \Spec \OO_x$ the disc centred at $x$. A choice of coordinate at $x$ induces an isomorphism $D_x \cong D$. The chiral operation $\mu$ of $\CA$ is determined by its restrictions to $D_x^2$ for $x \in X$, and these restrictions are expressed in terms of the vertex operation and the residue map (\ref{res.def}) as follows. Let $z$ be a coordinate at $x$ and $(z_1, z_2)$ the coordinate on $D_x^2$ induced by $z$. The restriction of the chiral operation
\begin{align*}
\mu : j_*j^*(\CA \boxtimes \CA) \rightarrow \frac{j_*j^*(\om \boxtimes \CA)}{\om \boxtimes \CA}
\end{align*}
to $D_x^2$ is given by
\begin{align}\label{FBZ.chiral.op}
\mu\left( f(z_1, z_2) a dz_1 \boxtimes b dz_2 \right) = f(z_1, z_2) dz_1 \boxtimes a(z_1-z_2)b dz_2
\end{align}
for all $a, b \in V$ and $f \in \C[[z_1,z_2]][(z_1-z_2)^{-1}]$. On the other hand, in terms of the description of the $\CD$-module pushforward given in (\ref{eq:diagonal.pushforward}) the chiral product $\mu$ looks as follows (cf. \cite[19.1.5]{FBZ.Book}). We have the identification
\begin{align*}
\G(D_x^2, \D_*\CA) = \G(D_x, \CA) \otimes_{\C[\partial]}\C[\partial_1, \partial_2],
\end{align*}
where $\partial$ acts on $\G(D_x, \CA)$ as by $\partial_z$, $\partial_i = \partial_{z_i}$, and $\partial = \partial_1 + \partial_2$. In these terms $\mu$ is given by
\begin{align}\label{FBZ.lambda.op}
\mu\left( f(z_1, z_2) a dz_1 \otimes b dz_2 \right) = \res_{z_1=z_2} e^{(z_1-z_2)\vec{\partial}_1} f(z_1,z_2) a(z_1-z_2)b \otimes 1,
\end{align}
where the residue symbol $\res_{z_1=z_2}$ indicates to write $z_1$ as $(z_1-z_2)+z_2$, expand the expression in positive powers of $z_1-z_2$ and extract the coefficient of $(z_1-z_2)^{-1}$. The notation $\vec{\partial}_1$ indicates to remove all powers of $\partial_1$ to the right-hand side of the tensor product symbol.

The Jacobi identity satisfied by $\mu$, i.e., the vanishing of the composition (\ref{Cousin.complex}) below, corresponds at the level of the vertex algebra $V$ to the Borcherds identity (\ref{Borcherds.identity}), while skew-symmetry of $\mu$ corresponds to the skew-symmetry identity (\ref{skewsymmetry.identity}). 

A choice of coordinate $z$ at a point $x \in X$ now induces identifications 
\begin{align}
\label{mu.to.f}
\begin{split}
\xymatrix{
\G(D_x^\times, \CA) \otimes \CA_x \ar@{->}[r]^-{\mu} \ar@{->}[d]^{\cong} & \CA_x \ar@{->}[d]^{\cong} \\
V((z))dz \otimes V \ar@{->}[r] & V, \\
}
\end{split}
\end{align}
where the morphism in the lower line is the $f$-product
\begin{align*}
a f(z) dz \otimes b \mapsto \res_z f(z) a(z)b dz = a_{(f)}b.
\end{align*}

\end{nolabel}

\section{Conformal Blocks}\label{sec:conformal.blocks}

\begin{nolabel}
Let $X$ be a smooth complex algebraic curve and $V$ a quasiconformal vertex algebra. The inclusion $D_x^\times \hookrightarrow X \backslash x$ induces a map
\begin{align*}
\G(X, \CA(*x)) \rightarrow \G(D_x^\times, \CA),
\end{align*}
and one may consider the vector space of coinvariants
\begin{align}\label{eq:cb.def.raw}
H(X, x, V) = \frac{\CA_x}{\G(X, \CA(*x)) \cdot \CA_x}.
\end{align}
The dual of this space is known as the space of conformal blocks associated with $X, x$ and $V$. The conformal blocks of an elliptic curve is the central object of Zhu's paper \cite{Z96}.
\end{nolabel}

\begin{nolabel}
The vertex algebra $V$ carries an increasing filtration $A^\D{V} = \bigoplus_{n \leq \D} V_n$, which is stable under the linear automorphism $R(f)$ associated with any change of coordinate $f$ as in Section \ref{subsec.OO.def}. Therefore $A^\D{V}$ induces a filtration $A^\D\CV$ of $\CV$ by vector subbundles of finite rank. It follows from the construction of $\CV$ that the successive quotients are direct sums of tensor powers of the tangent bundle, i.e.,
\begin{align}\label{c.w.filtr}
0 \rightarrow A^{\D-1}\CV \rightarrow A^{\D}\CV \rightarrow (\Theta_X^{\otimes \D})^{\oplus \dim{V_\D}} \rightarrow 0,
\end{align}
and so the associated graded takes the form
\begin{align*}
\gr^A\CV = \bigoplus_{\D \in \Z_{\geq 0}}\gr^A_\D\CV \quad \text{with} \quad \gr^A_\D \CV \cong (\Theta_X^{\otimes \D})^{\oplus \dim{V_\D}}.
\end{align*}
\end{nolabel}

\begin{nolabel}
We now work with $X$ a smooth elliptic curve and $x$ its marked point. We may fix $\tau \in \HH$ where $\HH$ denotes the complex upper half plane, and define $X$ in the analytic topology as $\C/\La$ where $\La = \Z + \Z\tau \subset \C$. The marked point $0=[\La] \in X$ and the addition induced from that in $\C$ gives $X$ the structure of an elliptic curve. Of course $X$ is algebraic, and the function
\begin{align}\label{ell.curve.embedding}
z \mapsto [u : v : w] = [\wp(z) : \wp'(z) : 1], \quad 0 \mapsto [0 : 1 : 0],
\end{align}
where $\wp(z)$ is the classical Weierstrass function (\ref{eq:def.Weierstrass.p}), descends to an embedding $X \subset \C P^2$, thus presenting $X$ as the zero set of a homogeneous cubic. To describe the coinvariants (\ref{eq:cb.def.raw}) explicitly, it is convenient to work with the local analytic coordinate induced from the standard coordinate $z$ on $\C$.


Since $X$ is an elliptic curve we have $\Theta_X \cong \OO_X$, and so $\gr^A\CV \cong V \otimes \OO_X$ as $\OO_X$-modules. In fact as $\OO_X$-modules we have
\begin{align}\label{eq:trivialisation.V}
\CV \cong \CA \cong V \otimes \OO_X.
\end{align}
Indeed the standard coordinate $z$ of $\C$ induces compatible trivialisations of $\CV_x$ at all points $x \in X$, so the assignment
\begin{align*}
a \otimes (dz)^{\otimes \D} \mapsto (z, a)
\end{align*}
(where on the right-hand side $z$ denotes, by abuse of notation, the formal coordinate induced by $z$ at $x$) defines a morphism $V_\D \otimes \Theta_X^{\otimes \D} \rightarrow A^\D\CV$. It follows that (\ref{c.w.filtr}) is split, and so we have $\CA \rightarrow V \otimes \OO_X$ as $\OO_X$-modules. Thus we obtain identifications
\begin{align*}
\G(X, \CA(*0)) \cdot \CA_x \cong \left(V \otimes \G(X, \OO_X(*0)) \right) \cdot V \cong \left<V_{(f)}V | f \in \G(X, \OO_X(*0))\right>,
\end{align*}
and so the space of coinvariants is
\begin{align}\label{Hdef}
H(X, 0, V) \cong \frac{V}{\left<V_{(f)}V | f \in \G(X, \OO_X(*0))\right>}.
\end{align}
We now describe the relation with Zhu's conformal blocks. Let $(V, Y[-, z])$ be Zhu's modified vertex algebra structure. The space of conformal blocks as defined and used by Zhu in \cite{Z96} is now the vector space dual to the space of coinvariants $H(X, 0, (V, Y[-, z]))$ given in (\ref{Hdef}). This is the content of {\cite[Proposition 5.2.1]{Z96}} parts (3) and (4). Since $R : (V, Y[-, z]) \rightarrow (V, Y(-,z))$ is an isomorphism of vertex algebras, the spaces of coinvariants associated with $(V, Y[-, z])$ and $(V, Y(-, z))$ are, of course, naturally isomorphic.
\end{nolabel}

\section{Preliminaries on Homology}\label{sec:homology}
In this section we gather some well-known homological algebra constructions for ease of reference and to fix conventions. These are the Chevalley complex and Lie algebra homology, the bar complex and Hochschild homology, and the de Rham cohomology of a $\CD$-module over a smooth variety.
\begin{nolabel}
Firstly we recall a few standard notational conventions. Let $A^\bullet$ be a (cohomological) complex in an abelian category, then $A^\bullet[k]$ signifies the complex $A^\bullet$ shifted $k$ places to the left, i.e., $A^n[k] = A^{n+k}$, with $d_{A[k]} = (-1)^k d_{A}$. The cone on a morphism $f : A^\bullet \rightarrow B^\bullet$ of complexes is defined as $\cone(f) = A^\bullet[1] \oplus B^\bullet$ with differential $d(a, b) = (-da, db + f(a))$.

Now let $\CC$ be a tensor category, i.e., an abelian symmetric monoidal category. The tensor product of the complexes $A^\bullet$, $B^\bullet$ in $\CC$ is the complex $(A^\bullet \otimes B^\bullet)^n = \bigoplus_{i+j=n} A^i \otimes B^j$ with differential $d$  specified by $d(a \otimes b) = (da)\otimes b + (-1)^{|a|}a \otimes (db)$ for $a$ and $b$ homogeneous of homological degrees $|a|$ and $|b|$ respectively. The symmetric algebra $\Sym(A^\bullet) = \bigoplus_{n \in \Z_{\geq 0}} \Sym_n(A^\bullet)$ on $A^\bullet$ is the sum of the quotients $\Sym_n(A^\bullet) = (A^\bullet)^{\otimes n}/\Sigma_n$ by symmetric group actions in which $a \otimes b$ is identified with $(-1)^{|a| \cdot |b|} b \otimes a$. The product in the symmetric algebra is given by $(a \otimes b) \cdot c = a \otimes b \otimes c$, etc.
\end{nolabel}

\begin{nolabel}\label{sec:lie.homology}
Let $\g$ be a Lie algebra over $\C$ and $U(\g)$ its universal enveloping algebra. The augmentation morphism is the morphism of unital algebras $U(\g) \rightarrow \C$ that sends $\g \subset U(\g)$ to $0$. The homology of $\g$ with coefficients in a $\g$-module $M$ is by definition the torsion
\begin{align*}
H_n(\g, M) = \Tor_n^{U(\g)}(\C, M).
\end{align*}
Lie algebra homology is computed by the (reduced) Chevalley complex defined as follows. Let the linear map $\delta : \g^* \rightarrow \La^2(\g^*)$ denote the transpose of the Lie bracket $[\cdot,\cdot] : \g \otimes \g \rightarrow \g$. Then $\delta$ may be extended uniquely to a derivation of the exterior algebra $\Sym(\g^*[-1])$. The Jacobi identity on $[\cdot,\cdot]$ implies that $\delta^2=0$. There is a perfect pairing between the exterior algebras $\Sym(\g[1])$ and $\Sym(\g^*[-1])$ extending that between $\g$ and $\g^*$, and the dual $d$ of $\delta$ with respect to this perfect pairing gives $\Sym(\g[1])$ the structure of a complex. The differential $d : \g \rightarrow \C$ vanishes, so as a complex $\Sym(\g[1])$ splits as the direct sum of $\C$ concentrated in degree $0$ and a complex $C^\bullet(\g)$ concentrated in negative degrees. This complex $(C^\bullet(\g), d)$ is the \emph{reduced Chevalley complex} of $\g$. Explicitly $d$ is given by
\begin{align}\label{CE.diff}
d \left(x_1 \wedge \ldots \wedge x_n\right)
= \sum_{ i < j} (-1)^{i+j} [x_i, x_j] \wedge x_1 \wedge \ldots \wedge \widehat{x}_i \wedge\ldots \wedge \widehat{x}_j \wedge \ldots \wedge x_n.
\end{align}
For $n\geq 1$ we have $H^{-n}(C^\bullet(\g), d) \cong H_n(\g, \C)$ the homology of $\g$ with coefficients in the trivial $\g$-module $\C$.

If $\CC$ is any symmetric monoidal category then the definitions of Lie algebra and of reduced Chevalley complex can be formulated in $\CC$. That is to say for $\g$ an algebra in $\CC$ over the Lie operad one may write down $C^\bullet(\g)$ as a complex in $\CC$ with differential (\ref{CE.diff}).
\end{nolabel}
\begin{nolabel}
Let $A$ be an associative unital $\C$-algebra. An $(A, A)$-bimodule is the same thing as a left module over the algebra $\envela = A \otimes_\C A^\text{op}$. The Hochschild homology of $A$ with coefficients in the $(A, A)$-bimodule $M$ is then the torsion
\begin{align*}
\Hoch_n(A, M) = \Tor_n^{\envela}(A, M).
\end{align*}
We denote the Hochschild homology $\Hoch_\bullet(A, A)$ of the bimodule $A$ simply by $\Hoch_\bullet(A)$. Hochschild homology is computed by the bar complex defined as follows. The free product $F = A * \C[\varepsilon]$ of unital algebras is made into an $(A, A)$-bimodule in the obvious way, and is made into a differential graded algebra by putting $\deg(A) = 0$, $\deg(\varepsilon) = -1$, $da = 0$ for all $a \in A$, and $d\varepsilon = 1_A$. We denote by $F_n$ the component of $F$ of degree $n$, and define the \emph{bar complex} $B_\bullet A$ of $A$ by $B_n{A} = F_{n-1}$ for all $n \leq 0$, with differential $d$. Via the augmentation map $F_{-1} \rightarrow F_0 = A$ the bar complex is a resolution of $A$ by free $\envela$-modules. From this resolution of $A$ we obtain the standard complex computing Hochschild homology $\Hoch_\bullet(A, M)$ of the $(A, A)$-bimodule $M$ as follows. For $n \in \Z_{\geq 0}$ the component in degree $-n$ of the complex is $A^{\otimes n} \otimes_\C M$, and the differentials are:
\begin{align}\label{eq:bar.complex.conv}
\begin{split}
d(a_0 \otimes \ldots \otimes a_n \otimes m)
= {} & a_1 \otimes \ldots \otimes a_n \otimes ma_0 \\
&+ \sum_{i=1}^{n} (-1)^{i} a_0 \otimes \ldots \otimes a_{i-2} \otimes a_{i-1} a_i \otimes a_{i+1} \otimes \ldots \otimes a_n \otimes m \\
& -(-1)^n a_0 \otimes a_1 \otimes \ldots \otimes a_{n-1} \otimes a_{n}m.
\end{split}
\end{align}
\end{nolabel}

\begin{nolabel}\label{no:derham.variety.defn}
Let $Y$ be a smooth complex variety of complex dimension $n$. The de Rham cohomology groups of the right $\CD_Y$-module $M$ are by definition the cohomology groups of the object
\begin{align*}
R\Gamma_{\DR}(Y, M) = R\Gamma(Y, \DR(M)) = R\Gamma(Y, M \otimes_{\CD_Y}^L \OO_Y)
\end{align*}
of the derived category of the category of sheaves of $\C$-vector spaces on $Y$.

The left $\CD_Y$-module $\OO_Y$ may be resolved as $\CD_Y \otimes_{\OO_Y} \Sym(\Theta_Y[1])$, and the object $M \otimes_{\CD_Y}^L \OO_Y$ thus represented by the complex of sheaves
\begin{align}\label{eq:DR.rep}
0 \rightarrow M \otimes_{\OO_Y} \wedge^n \Theta_Y \rightarrow \cdots \rightarrow M \otimes_{\OO_Y} \Theta_Y \rightarrow M \otimes_{\OO_Y} \OO_Y \rightarrow 0
\end{align}
(nonzero in degrees $-n$ through $0$) with differentials
\begin{align*}
d(m \otimes \xi_1 \wedge \cdots \wedge \xi_k)
= {} & \sum_i (-1)^i (m \xi_i) \otimes \xi_1 \wedge \cdots \wedge \widehat{\xi}_i \wedge \cdots \wedge \xi_k \\
&+ \sum_{i < j} (-1)^{i+j} m \otimes [\xi_i, \xi_j] \wedge \xi_1 \wedge \cdots \wedge \widehat{\xi}_i \wedge \cdots \wedge \widehat{\xi}_j \wedge \cdots \wedge \xi_k.
\end{align*}

\end{nolabel}

\section{Chiral Homology}\label{sec:ch.hom}

In this section we give a brief account of Beilinson and Drinfeld's theory of chiral homology \cite[Section 4]{BD.Chiral}.
\begin{nolabel}
Denote by $\mathcal{S}$ the category whose objects are non-empty finite sets and whose morphisms are surjective functions, and by $Q(I)$ the set of equivalence relations on the finite set $I$. The Ran space $\ran(X)$ of a topological space $X$ is the set of non-empty finite subsets of $X$, equipped with the strongest topology under which the obvious functions $X^I \rightarrow \ran(X)$ become continuous. To work with the Ran space of an algebraic variety $X$ using tools from geometry the notions of ``$!$-sheaf on $\XtotheS$'' and ``right $\CD$-module on $\XtotheS$'' are introduced as technical substitutes for the notions of ``sheaf of vector spaces on $\ran(X)$'' and ``right $\CD$-module on $\ran(X)$'', respectively.

Although in this article we adopt the framework introduced in the book \cite{BD.Chiral}, we remark that after its publication the theory of $\CD$-modules on Ran spaces has been substantially developed by Gaitsgory and collaborators, using the language of $(\infty, 1)$-categories. See for example \cite{Francis-Gaitsgory}.

Let $\pi : I \rightarrow R$ be a surjection of finite sets. As in {\cite[3.4.4]{BD.Chiral}} the open embedding of
\begin{align}\label{eq:U.j.i}
U^{[I/R]} = \{(x_i)_{i \in I} \in X^I | \text{$x_{i_1} \neq x_{i_2}$ whenever $\pi(i_1) \neq \pi(i_2)$}\}
\end{align}
into $X^I$ is denoted $j^{[I/R]}$. The natural inclusion $X^R \rightarrow X^I$ with image
\begin{align}\label{eq:D.j.i}
\D^{(I/R)} = \{(x_i)_{i \in I} \in X^I | \text{$x_{i_1} = x_{i_2}$ whenever $\pi(i_1) = \pi(i_2)$}\}
\end{align}
is written $\D^{(\pi)}$ or $\D^{(I/R)}$ by abuse of notation.

A $!$-sheaf $F$ on $\XtotheS$ {\cite[4.2.1]{BD.Chiral}} consists of a sheaf $F_{X^I}$ of vector spaces on $X^I$ for each finite set $I$ and a morphism $\theta^{(\pi)} : \D^{(\pi)}_*(F_{X^R}) \rightarrow F_{X^I}$ for each surjection $\pi : I \rightarrow R$, subject to the compatibility conditions
\begin{align}\label{shriek.compat}
\theta^{(\pi_1\pi_2)} = \theta^{(\pi_2)}\circ \D^{(\pi_2)}_*(\theta^{(\pi_1)}) \quad \text{and} \quad \theta^{(\text{id}_I)} = \text{id}.
\end{align}
Similarly a right $\CD$-module $M$ on $\XtotheS$ {\cite[3.4.10]{BD.Chiral}} consists of a right $\CD_{X^I}$-module $M_{X^I}$ for each finite set $I$ and a morphism $\theta^{(\pi)} : \D^{(\pi)}_*(M_{X^R}) \rightarrow M_{X^I}$ (where now $\D^{(\pi)}_*$ denotes pushforward of $\CD$-modules) for each surjection $\pi : I \rightarrow R$, satisfying (\ref{shriek.compat}).

Let $M$ be a right $\CD_X$-module. The assignment $M_{X^I} = \D^{(I)}_*{M}$, $\theta^{(\pi)} = \text{id}_{M_{X^J}}$ defines a right $\CD$-module on $\XtotheS$. We denote this assignment by $\D^{(\mathcal{S})}_*$.

Let $\{L_r\}_{r \in R}$ be a finite non-empty set of right $\CD$-modules on $\XtotheS$. One defines {\cite[3.4.10]{BD.Chiral}} a new right $\CD$-module $\bigotimes_{r \in R}^{\text{ch}} L_r$ on $\XtotheS$ by putting
\begin{align}\label{tensor.ch.def}
\left(\bigotimes{}^{\text{ch}}_{r \in R} L_r\right)_{X^I} = \bigoplus_{\pi : I \rightarrow R} j^{[I/R]}_* {j^{[I/R]}}^* \left( \boxtimes \left(L_r\right)_{X^{\pi^{-1}(r)}} \right).
\end{align}
A permutation of $R$ induces a permutation of the factors of the direct sum on the right-hand side of (\ref{tensor.ch.def}) and thus natural commutativity isomorphisms between differently ordered tensor products. In this way the category of right $\CD$-modules on $\XtotheS$ becomes a symmetric tensor category.
\end{nolabel}
\begin{nolabel}
Let $\CA$ be a chiral algebra on $X$. Then its image $\D^{(\mathcal{S})}_*\CA$ is a Lie algebra in the category of right $\CD$-modules on $\XtotheS$. The \emph{Chevalley-Cousin complex} $C(\CA)$ of $\CA$ is by definition the Chevalley complex of the Lie algebra $\D^{(\mathcal{S})}_*\CA$ as in Section \ref{sec:homology}.
\end{nolabel}

\begin{nolabel}\label{sec:chevalley-cousin}
It is clear that $C(\CA)_X = \CA[1]$, and that in general $C(\CA)_{X^I}$ is concentrated in cohomological degrees $-|I|, \ldots, -1$. For instance $C^{-1}(\CA)_{X^2} = \D_*\CA$ and
\begin{align*}
C^{-2}(\CA)_{X^2} = \left((\CA \otimes^{\text{ch}} \CA)_{X^2}\right)_{\Sigma_2} \cong j_*j^*(\CA \boxtimes \CA).
\end{align*}
Indeed $(\CA \otimes^{\text{ch}} \CA)_{X^2}$ is a direct sum of factors indexed by the two bijections $\pi : \{1, 2\} \rightarrow \{1, 2\}$ and the passage from the tensor algebra to the symmetric algebra corresponds to the application of coinvariants by the action of $\Sigma_2$ which exchanges these factors. Thus $C(\CA)_{X^2}$ is the complex
\begin{align*}
j_*j^*(\CA[1] \boxtimes \CA[1]) \rightarrow \D_*\CA[1],
\end{align*}
where the morphism is just the chiral product $\mu$. Similarly $C(\CA)_{X^3}$ is
\begin{align}\label{Cousin.complex}
j_*j^*(\CA[1]^{\boxtimes 3}) \rightarrow \bigoplus_{1 \leq k < \ell \leq 3} j^{[k, \ell]}_* {j^{[k, \ell]}}^*(\CA[1] \boxtimes \D^{(2)}_*\CA[1]) \rightarrow \D^{(3)}_*\CA[1],
\end{align}
where $j : X \rightarrow X^3$ is the embedding of the small diagonal, and where $j^{[k, \ell]}$ denotes the open embedding (\ref{eq:U.j.i}) associated with the surjection $\{1, 2, 3\} \rightarrow \{1,2\}$ that sends the (distinct) elements $k$ and $\ell$ to $2$ and sends the remaining element of $\{1, 2, 3\}$ to $1$. The morphisms are built from the chiral operation $\mu$, and to say that $\mu$ satisfies the Jacobi identity now just means that the composition (\ref{Cousin.complex}) vanishes. In general $C(\CA)_{X^I}$ is the following sum over the set $Q(I)$ of equivalence relations
\begin{align}\label{XJ.Cousin.complex}
C(\CA)_{X^I} = \bigoplus_{R \in Q(I)} C(\CA)_{I, R}, \quad \text{where} \quad C(\CA)_{I, R} = \D^{(I/R)}_* j^{[I/R]}_* {j^{[I/R]}}^* \CA^{\boxtimes R},
\end{align}
with $C(\CA)_{I, R}$ in degree $-|R|$, and differential built from copies of $\mu$. The Jacobi identity ensures that $C(\CA)_{X^I}$ is a complex.
\end{nolabel}

\begin{nolabel}\label{sec:colimit.S}
Let $M$ be a right $\CD$-module on $\XtotheS$. Then $\DR(M)$ is defined by the assignment
\begin{align*}
\DR(M)_{X^I} = \DR(M_{X^I}),
\end{align*}
regarded either as an object of the derived category of $!$-sheaves on $\XtotheS$ or else, making use of the explicit representative (\ref{eq:DR.rep}), as a complex of $!$-sheaves on $\XtotheS$. 

Let $F$ be a $!$-sheaf on $\XtotheS$. Then
\begin{align*}
I \mapsto \Gamma(X^I, F_{X^I})
\end{align*}
defines an $\mathcal{S}^{\text{op}}$-diagram in $\text{Vect}$:
\begin{align*}
\xymatrix{
{\G(X, F_{X})} \ar[r] &  {\G(X^2, F_{X^2})} \ar[r] \ar@/^0.5pc/[r] \ar@/^-0.5pc/[r] & {\G(X^3, F_{X^3})} & \cdots \\
}
\end{align*}
By definition the space of global sections $\Gamma(\XtotheS, F)$ is the colimit of the diagram, i.e., the initial object among all objects which receive a compatible system of morphisms from the diagram. The derived global sections functor $R\Gamma(\XtotheS, -)$, being the derived functor of a colimit, can be described as a homotopy colimit. We do not discuss homotopy colimits here, however, as they will ultimately not be required in this work.

The functor $R\G_{\DR}$ of de Rham cohomology for $\CD$-modules on $\XtotheS$ is now defined, as for varieties, as the derived composition
\begin{align*}
R\Gamma_{\DR}(\XtotheS, M) = R\Gamma(\XtotheS, \DR(M)),
\end{align*}
of $\DR$ and $\G(\XtotheS, -)$ {\cite[4.2.6(iv)]{BD.Chiral}}.

\end{nolabel}
\begin{nolabel}
Let $\CA$ be a chiral algebra on $X$. By definition {\cite[4.2.11]{BD.Chiral}} the chiral homology of $\CA$ is the de Rham cohomology of the Chevalley-Cousin complex $C(\CA)$. That is
\begin{align*}
H^{\text{ch}}_n(X, \CA) = H^{-n}(C^{\text{ch}}(X, \CA))
\end{align*}
where
\begin{align*}
C^{\text{ch}}(X, \CA) = R\Gamma_{\DR}(\XtotheS, C(\CA)).
\end{align*}
\end{nolabel}

\section{Chiral Homology of Elliptic Curves}\label{sec:ch.hom.ell}

\begin{nolabel}
We now apply the general theory of Beilinson and Drinfeld to the specific case of $X$ a smooth elliptic curve and $\CA$ the chiral algebra on $X$ associated with a quasiconformal vertex algebra $V$. We begin by describing the Chevalley-Cousin complexes and chiral chain complexes at the level of global sections algebraically. We then do the same for the complexes with coefficients in a module supported at a point. Although this variant is more complicated, it has the advantage of making it easier to apply the de Rham functor by eliminating contributions from higher derived global sections functors. Ultimately we obtain an explicit complex $A^\bullet$, given in Proposition \ref{prop:chhom.equals.A} below, which computes chiral homology of $X$ in degrees $0$ and $1$. The ingredients in the definition of $A^\bullet$ are the $f$-products in the vertex algebra $V$ and some spaces of functions $\mathring{\Gamma}_1 \cong \Gamma(X \backslash \{0\}, \OO)$ and $\mathring{\Gamma}_2 \cong \Gamma((X  \backslash \{0\})^2 \backslash \Delta, \OO)$. See (\ref{eq:G.ring.def}) below. In subsequent sections we describe these spaces, and the complex $A^\bullet$, in greater detail in terms of classical elliptic functions.


Let $\C$ be the complex plane and $z$ its standard global coordinate. From now on $X$ will denote the elliptic curve $\C / \Z +\Z\tau$ with marked point $0$. The vector field $\partial / \partial z$ on $\C$ induces a global vector field on $X$ which we denote $\xi$. Under the embedding (\ref{ell.curve.embedding}) into projective space $\xi$ corresponds to the algebraic vector field $v \partial / \partial u$.

We have an isomorphism of algebras $\G(X, \CD_X) \cong \C[\la]$ where $\la$ represents $-\xi$ and in general, for a finite set $I$, we have an isomorphism of algebras $\G(X^I, \CD_{X^I}) \cong \C[\la_i]_{i \in I}$ where $\la_i$ represents the pullback $-\pi_i^*(\xi)$ along the $i^{\text{th}}$ projection $\pi_i : X^I \rightarrow X$. Let $M$ be a right $\CD_{X^I}$-module, then the space of global sections $\G(X^I, M)$ becomes a $\C[\la_i]_{i \in I}$-module. Now let $M$ be a right $\CD_X$-module, so that $\G(X, M)$ is a $\C[\la]$-module and $\G(X^2, \D_*M)$ is a $\C[\la_1, \la_2]$-module. There is a natural $\C[\la]$-action on $\C[\la_1, \la_2]$ given by putting $\la = \la_1+\la_2$. One then has
\begin{align}\label{eq:diag.pushforward.lambda}
\G(X^2, \D_*M) \cong \G(X, M) \otimes_{\C[\la]} \C[\la_1, \la_2]
\end{align}
as in equation (\ref{eq:diagonal.pushforward}). Henceforth, for $I$ a finite set and $j : U^{(I)} \rightarrow X^I$ the embedding (\ref{eq:standard.open}), we write
\begin{align*}
\G_I = \G(X^I, j_*j^*\OO_{X^I}).
\end{align*}
Sometimes we abuse notation and write $\Gamma_n$ to stand for $\Gamma_I$ where $I$ is any set of $n$ elements.

Now let $V$ be a quasiconformal vertex algebra and $\CA = \CA_V$ the chiral algebra on $X$ associated with $V$. As in Section \ref{sec:conformal.blocks} we have a trivialisation (\ref{eq:trivialisation.V}) of $\CA$ as an $\OO_X$-module and identifications 
\begin{align*}
\G(X^I, j_*j^*\CA^{\boxtimes I}) \cong \G_I \otimes V^{\otimes I}.
\end{align*}
The action of $\G(X^I, \CD_{X^I})$ on $\G(X^I, j_*j^*\CA^{\boxtimes I})$ is identified with the action of $\C[\la_i]_{i \in I}$ on $\G_I \otimes V^{\otimes I}$ in which $\la_i$ acts by
\begin{align*}
\partial_{z_i} \otimes 1 + 1 \otimes T_i.
\end{align*}
Here $T_i \in \en(V^{\otimes I})$ is shorthand for the tensor product of $T$ on the $i^{\text{th}}$ component and the identity on all other components. We have a commutative diagram
\begin{align*}
\xymatrix{
\G(X^2, j_*j^* \CA^{\boxtimes 2}) \ar@{->}[r]^\mu \ar@{->}[d] & \G(X^2, \D_* \CA) \ar@{->}[d] \\
\G_2 \otimes V^{\otimes 2} \ar@{->}[r] & (\G_1 \otimes V) \otimes_{\C[\la]} \C[\la_1, \la_2], \\
}
\end{align*}
where, by (\ref{FBZ.lambda.op}), the lower horizontal arrow is the unique morphism of $\C[\lambda_1, \lambda_2]$-modules specified by
\begin{align}\label{2.1.witharrow}
f \otimes (a^1 \otimes a^2) \mapsto \left( \res_{z_1=z_2} f \cdot e^{\vec{\la}_1 (z_1-z_2)} a^{1}(z_{1}-z_{2})a^{2} \right) \otimes_{\C[\la]} 1.
\end{align}
Here the notation $\vec{\la}_1$ means that all copies of $\la_1$ obtained upon expansion of the exponential are to be moved to the right-hand factor of the tensor product.
\end{nolabel}


\begin{nolabel}
As remarked in Section \ref{sec:chevalley-cousin} the chiral operation may be thought of as the $X^2$ component of the Chevalley-Cousin complex $C(\CA)$. We now extend the description (\ref{2.1.witharrow}) above of the chiral operation to a similar description of the $X^I$ component of the Chevalley-Cousin complex for arbitrary finite set $I$.
\end{nolabel}

\begin{nolabel}\label{subsec:equiv.rel}
Let $I$ be a finite set. We denote by $Q(I)$ the set of all equivalence relations on $I$, as in Section \ref{sec:ch.hom}, and by $Q(I, k)$ the set of equivalence relations on $I$ consisting of exactly $k$ equivalence classes. If $\pi : I \twoheadrightarrow R$ is a surjection with $|R|=k$ then we denote by $[R] \in Q(I,k)$ the corresponding equivalence relation and by $I_r$ the preimage $\pi^{-1}(r)$ of $r \in R$. If $\pi' : I \twoheadrightarrow R$ is another surjection then we have $[R] = [R']$ if there exists a bijection $R \cong R'$ such that the following diagram commutes
\begin{align*} 
\xymatrix{
& I \ar[dl] \ar[dr] \\ 
R \ar[rr]^{\cong} & & R'. \\
}
\end{align*}
We denote by $\C[I]$ the $\C$-algebra of polynomials on generators $\la_i$ indexed by $i \in I$. A surjection $\pi : I \twoheadrightarrow R$ induces a morphism $\varphi_\pi : \C[R] \rightarrow \C[I]$ defined by $\varphi_\pi(\la_r) = \sum_{i \in I_r}\la_i$, and thus a $\C[R]$-module structure on $\C[I]$.

Let $\pi : I \rightarrow R$ be a surjection, so we have the diagonal embedding $\D^{(I/R)} : X^R \hookrightarrow X^I$ defined in (\ref{eq:D.j.i}). The image of $\D^{(I/R)}$ clearly depends on $\pi$ only through $[R]$. If $M$ is a right $\CD_{X^R}$-module then, generalising (\ref{eq:diag.pushforward.lambda}), we have
\begin{align*}
\G(X^I, \D^{(I/R)}M) \cong \G(X^R, M) \otimes_{\C[R]} \C[I],
\end{align*}
with the $\C[R]$-action on $\C[I]$ as described in the preceding paragraph.


Now for $[R] \in Q(I)$ we define the $\C[I]$-module
\begin{align}\label{eq:cousin-part-1}
C_{I,[R]} = \Bigl( \Gamma_R \otimes V[1]^{\otimes R} \Bigr) \otimes_{\C[R]} \C[I]
=
\frac{\Gamma_R \otimes V[1]^{\otimes R} \otimes \C[I]}{\left\langle -(\partial_{z_r} + T_r) + \sum_{i \in I_r} \lambda_{i} \right\rangle_{r \in R}}.
\end{align}
Note that we applied a shift to $V$ so that $C_{I,[R]}$ sits in cohomological degree $-|R|$. The symbol $\left\langle \cdots \right\rangle$ denotes $\C[I]$-submodule generated, and we abuse notation writing $\partial_{z_r} + T_r$ in place of $\partial_{z_r} \otimes 1 \otimes 1 + 1 \otimes T_r \otimes 1$. This construction is independent of the choice of representative $R$ of $[R]$ up to canonical isomorphism. At last we define the graded $\C[I]$-module $C_I$ to be
\begin{equation} \label{eq:cousin1}
C_I = \bigoplus_{k \geq 1} C_I^{-k} \quad \text{where} \quad C_I^{-k} = \bigoplus_{[R] \in Q(I, k)} C_{I,[R]}.
\end{equation}
\label{no:labeladded}
\end{nolabel}

\begin{nolabel}\label{no.coeff.differential}
At the level of global sections the complex $C(\CA)_{X^I}$ given in equation (\ref{XJ.Cousin.complex}) becomes the graded $\C[I]$-module $C_I$. We now describe the differential $d$ on $C_I$. The nonzero components of the differential are of the form $d_{[R], [R']} : C_{I, [R]} \rightarrow C_{I, [R']}$ where $[R] \in Q(I, k+1)$ and $[R'] \in Q(I, k)$ is obtained from $[R]$ by joining two of its equivalence classes into one. So let $R = \ov R \cup \{r_1, r_2\} \in Q(I, k+1)$ and let $R' = \ov R \cup \{r_2\} \in Q(I, k)$ be the quotient obtained from $R$ by mapping $r_i \mapsto r_i$ for $i \neq 1$ and $r_1 \mapsto r_2$. We now characterise
\begin{align*}
d_{[R], [R']} : C_{I, [R]} \rightarrow C_{I, [R']}
\end{align*}
as the unique map of $\C[I]$-modules satisfying
\begin{multline} \label{eq:differential-1-raw}
f \otimes \Bigl( \bigotimes_{r \in R} a^r \Bigr) \otimes 1 \mapsto
\left( \res_{z_{r_1} = z_{r_2}} f \cdot e^{\vec{\lambda}_{r_1} (z_{r_1}-z_{r_2})} a^{r_1}(z_{r_1} - z_{r_2})a^{r_2} \right) \otimes \Bigl( \bigotimes_{\ov{r} \in \ov{R}} a^{\ov{r}} \Bigr) \otimes 1.
\end{multline}
Here the notation $\vec{\lambda}_{r_1}$ means all copies of $\lambda_{r_1}$ are passed to the right-hand tensor factor and are replaced by $\sum_{i \in I_{r_1}} \lambda_i$. The residue $\res_{z_{r_1}=z_{r_2}}$ eliminates the dependence on the variable $z_{r_1}$ and the result is identified with an element of $\G_{R'}$. The fact that $(C_I, d)$ is a complex, i.e., that $d^2 = 0$, follows from the corresponding fact for the Chevalley-Cousin complex $C(\CA)_{X^I}$ and may also be demonstrated directly from the Jacobi identity for $V$.
\end{nolabel}

\begin{nolabel}\label{no:global.derham}
In the previous sections we have described the Chevalley-Cousin complex $C(\CA_V)$ of the chiral algebra $\CA_V$ in vertex algebraic terms. We now describe the chiral complex $C^{\text{ch}}(X, \CA_V)$ in similar terms. However we will stop short of using this description to perform computations, opting instead to work with a modified complex (with coefficients in a module supported at a point) for which higher derived global sections and homotopy colimits are easier to handle.

The de Rham functor of a smooth variety $Y$ was described in Section \ref{no:derham.variety.defn} as tensoring over $\CD_Y$ with the resolution of $\OO_Y$ by free $\CD_Y$-modules. We take $Y = X^I$ and we now describe the resolution at the level of global sections. We have $\G(X^I, \OO_{X^I}) = \C$ and $\G(X^I, \CD_{X^I}) = \C[I]$ and the resolution of $\C$ by free $\C[I]$-modules is as follows. We consider the augmentation map $\C[I] \rightarrow \C$ which sets $\la_i$ to $0$ for all $i \in I$. In this way $\C$ acquires the structure of a $\C[I]$-module, and we consider the graded commutative algebra
\begin{align*}
D_I^\bullet = \Sym(\cone(\id_{\C^{I}})).
\end{align*}
Clearly the degree $0$ component of $D_I^\bullet$ is naturally isomorphic to $\C[I]$, hence $D_I^\bullet$ is a $\C[I]$-module and $D_I^\bullet$ is our resolution of $\C$ by free $\C[I]$-modules. It is concentrated in degrees $-n, \ldots, 0$. If $M$ is a right $\CD_{X^I}$-module then
\begin{align*}
\G(X^I, \DR(M)) \cong \G(X^I, M) \otimes_{\C[I]} D_I^\bullet.
\end{align*}
If $M^\bullet$ is a complex of right $\CD_{X^I}$-modules then $\G(X^I, \DR(M^\bullet))$ is computed by the hypercohomology spectral sequence whose $E_1$-page is
\begin{align}\label{hypercoho.spec.seq}
E_1^{p, q} = H^q(\G(X^I, M^p) \otimes_{\C[I]} D_I^\bullet).
\end{align}
If $M^\bullet$ is a complex of right $\CD$-modules on $\XtotheS$ then the object $\G(\XtotheS, \DR(M^\bullet))$ is the homotopy colimit of the $\mathcal{S}^{\text{op}}$-diagram
\begin{align}\label{eq:GammaDR.M}
I \mapsto \G(X^I, M_{X^I}^\bullet) \otimes_{\C[I]} D_I^\bullet.
\end{align}
However to recover the chiral complex
\[
C^{\text{ch}}(X, \CA_V) = R\Gamma(\XtotheS, \DR(C(\CA_V)))
\]
it is not enough to put $M^\bullet = C(\CA_V)$ in (\ref{hypercoho.spec.seq}) and (\ref{eq:GammaDR.M}), since $\Gamma$ has nontrivial higher derived functors. One solution to this difficulty would be to replace each $C(\CA_V)_{X^I}$ by a $\Gamma$-acyclic replacement. Another approach, which is the one we follow in the sections below, is to work with $\Gamma$-acyclic complexes of chiral homology with \emph{coefficients} which, by a theorem of Beilinson and Drinfeld, are quasi-isomorphic to plain chiral homology.
\end{nolabel}

\begin{nolabel}\label{sec:supp.point}
We now describe chiral homology with coefficients in a chiral module $\CM$, following {\cite[Section 4.2.19]{BD.Chiral}}. We emphasise the case of $\CM$ a module supported at a point, and subsequently we extend the description of chiral homology of the elliptic curve developed in the preceding sections to this case.

We return briefly to the general context of $X$ a smooth algebraic curve and $\CA$ a chiral algebra on $X$. A chiral $\CA$-module is a $\CD_X$-module $\CM$ together with an action $\mu_{\CM} : j_*j^*(\CA \boxtimes \CM) \rightarrow \D_*\CM$ satisfying a natural analogue of (\ref{Cousin.complex}). As usual one may extend the chiral algebra structure from $\CA$ to $\CA \oplus \CM[-1]$ using $\mu_\CM$ and by declaring the product of two elements of $\CM$ to be zero. In fact $\CA \oplus \CM$ becomes a graded chiral algebra once we declare $\deg(\CA)=0$ and $\deg(\CM) = +1$.

By the term \emph{pointed set} we shall mean a set $I$ together with a choice of element $i_0$ referred to as the \emph{marked element}. We similarly define a \emph{pointed equivalence relation} on a set $I$ to be an equivalence relation on $I$ together with a choice of \emph{marked equivalence class}. We denote by $Q^*(I)$ the set of pointed equivalence relations, and by $Q^*(I, k)$ the subset of those composed of exactly $k$ equivalence classes.

The Chevalley-Cousin complex of the chiral algebra $\CA \oplus \CM[-1]$ carries an internal $\Z_{\geq 0}$-grading by number of copies of $\CM$, in addition to the cohomological degree. Following {\cite[4.2.19]{BD.Chiral}} the complex $C(\CA, \CM)$ is defined to be the component of $C(\CA \oplus \CM[-1])$ of internal degree $+1$. The explicit description of $C(\CA, \CM)$ is similar to that of $C(\CA)$ given in equation (\ref{XJ.Cousin.complex}). Indeed $C(\CA, \CM)_{X^J}$ is the following sum over the set $Q^*(J)$ of pointed equivalence relations
\begin{align*}
C(\CA, \CM)_{X^J} = \bigoplus_{I \in Q^*(J)} C(\CA, \CM)_{J, I}, \quad \text{where} \quad C(\CA)_{J, I} = \D^{(J/I)}_* j^{[J/I]}_* {j^{[J/I]}}^* \left( \CA^{\boxtimes I \backslash *} \boxtimes \CM \right),
\end{align*}
with $C(\CA)_{J, I}$ in degree $-|I|+1$. The differential is now a sum of components of two types; those in which two copies of $\CA$ are multiplied via $\mu$, and those in which a copy of $\CA$ acts on a copy of $\CM$ via $\mu_{\CM}$.

As before the chiral homology of $\CA$ with coefficients in $\CM$ is defined to be the de Rham cohomology of $C(\CA, \CM)$. That is
\begin{align*}
H^{\text{ch}}_n(X, \CA, \CM) = H^{-n}(C^{\text{ch}}(X, \CA, \CM))
\end{align*}
where
\begin{align*}
C^{\text{ch}}(X, \CA, \CM) = R\Gamma_{\DR}(\XtotheS, C(\CA, \CM)).
\end{align*}

Let $M$ be a $V$-module graded by non negative integer eigenvalues of $L_0$. This restriction on $M$ is very strong, but ultimately we shall only need the case $M = V$. By the localisation procedure (Section \ref{FBZconstruction}) there is an associated $\CA$-module $\CM$ on $X$ which, like $\CA$, is flat as an $\OO_X$-module. Now let $i_x : * \rightarrow X$ be the embedding of a closed point $x \in X$, and let $j_x$ denote the open embedding of the complement $U_x = X \backslash \{x\}$. We may form the $\CA$-module
\begin{align*}
\CM_x = \coker(\CM \rightarrow {j_x}_*{j_x}^*\CM)
\end{align*}
supported at the point $x$. We shall consider the chiral homology with coefficients in $\CM_x$ in particular for $M = V$. Proposition 4.4.3 of \cite{BD.Chiral} asserts the existence of a canonical quasi-isomorphism
\begin{align}\label{eq:quasi443}
C^{\text{ch}}(X, \CA) \cong C^{\text{ch}}(X, \CA, {\CA}_x).
\end{align}
This is an expression of the physics notion of ``propagation of vacua''.


The support of $\CA^{\boxtimes I \backslash *} \boxtimes \CM$ is $X^{I \backslash *} \times \{x\} \subset X^I$, and the preimage of this support under the embedding $j^{[J/I]}$ of $X^{I \backslash *} \times \{x\} \subset X^I$ is an affine variety, isomorphic to a complement of diagonals in $U_x^{I \backslash *}$. Ultimately it is this which allows one to compute chiral homology while avoiding higher derived image functors, see comments in \cite[4.2.19(ii)]{BD.Chiral}.
\end{nolabel}

\begin{nolabel}\label{subsec:pointed.equiv.rel}
We now return to the setting of $X$ an elliptic curve, taking $x = 0$ as marked point, $M$ a $V$-module as above and $\CM_0 = \coker(\CM \rightarrow {j_0}_*{j_0}^*\CM)$. At the level of global sections we have
\begin{align}\label{eq:Gamma.M}
\G(X, \CM_0) = M \otimes_\C \C[\lambda] = M[\lambda],
\end{align}
where the action of the global vector field $-\xi$ is multiplication by $\lambda$.

To describe the complex we first introduce the space of functions on $X^I$ with poles at diagonals and where any of the coordinates coincides with the marked point. We let $\mathring{X} = X \backslash \{0\}$, we write $\mathring{j} : \mathring{X}^I \backslash \Delta \rightarrow X^I$ for the canonical embedding, and we define
\begin{align}\label{eq:G.ring.def}
\mathring{\G}_I = \G(X^I, {\mathring{j}}_*{\mathring{j}}^*\OO_{X^I}).
\end{align}
In particular we allow that $I$ might be empty, in which case $\mx^I = *$ and $\mathring{\G}_I = \C$. Sometimes we abuse notation and write $\mathring{\G}_n$ to stand for $\mathring{\G}_I$ where $I$ is any set of $n$ elements. In fact the group structure of the elliptic curve $X$ can be used to write a bijection $\mathring{\G}_I \rightarrow \G_{I \cup *}$, namely $f \mapsto \hat{f}$ where $\hat{f}(\{x_i\}) = f(\{x_i - x_*\})$. We will use a similar observation in Section \ref{sec:totaro} below when we come to describe the de Rham cohomology of $\mathring{X}^n \backslash \Delta$ for small values on $n$ in terms of classical elliptic functions, but for now the observation plays no direct role.

Let $R = \ov R \cup \{r_0\}$ be a pointed set with marked element $r_0$. We give
\begin{align*}
\mathring{\G}_{\ov{R}} \otimes V[1]^{\otimes \ov{R}} \otimes M[\lambda]
\end{align*}
the structure of a $\C[R]$-module as follows: $\lambda_{r_0}$ acts as multiplication by $\lambda$, while $\lambda_r$ acts by
\[
\partial_{z_r} \otimes 1 \otimes 1 + 1 \otimes T_r \otimes 1
\]
for $r \in \ov R$. Now let $I$ be a non-empty finite set and let $[R] \in Q^*(I)$. Let $R = \ov{R} \cup \{r_0\}$ where $r_0$ is the marked equivalence class. We define the $\C[I]$-module
\begin{align*}
\mathring{C}_{I, [R]} = \left( \mathring{\G}_{\ov{R}} \otimes V[1]^{\otimes \ov{R}} \otimes M[\lambda] \right) \otimes_{\C[R]} \C[I],
\end{align*}
where the $\C[R]$-module structure on the first tensor factor is as described in the preceding paragraph. Now we define the graded $\C[I]$-module $\mathring{C}_I$ to be
\begin{equation} 
\mathring{C}_I = \bigoplus_{k \geq 0} \mathring{C}_I^{-k}, \quad \text{where} \quad \mathring{C}_I^{-k} = \bigoplus_{[R] \in Q^*(I, k+1)} \mathring{C}_{I,[R]}.
\label{eq:cousin.coeff}
\end{equation}
\end{nolabel}

\begin{nolabel}
As in Section \ref{no.coeff.differential} the nonzero components of the differential in $\mathring{C}_I$ are of the form $d_{[R], [R']} : \mathring{C}_{I, [R]} \rightarrow \mathring{C}_{I, [R']}$ where $[R] \in Q^*(I, k+2)$ and $[R'] \in Q^*(I, k+1)$ is obtained from $[R]$ by joining two of its equivalence classes into one. The components are now of two types.

Suppose firstly that $R = \ov R \cup \{r_0, r_1, r_2\} \in Q^*(I, k+3)$ (so that both $r_1$ and $r_2$ are distinct from the marked equivalence class $r_0$), and let $R' = \ov R \cup \{r_0, r_2\} \in Q^*(I, k+2)$ be the quotient obtained from $R$ by mapping $r_i \mapsto r_i$ for $i \neq 1$ and $r_1 \mapsto r_2$. Then 
\begin{align*}
d_{[R], [R']} : \mathring{C}_{I, [R]} \rightarrow \mathring{C}_{I, [R']}
\end{align*}
is the unique map of $\C[I]$-modules satisfying
\begin{multline} \label{eq:cousin-coeff-part-1}
f \otimes \Bigl( \bigotimes_{r \in R \backslash \{r_0\}} a^r \Bigr) \otimes m \otimes 1 \mapsto
\left( \res_{z_{r_1} = z_{r_2}} f \cdot e^{\vec{\lambda}_{r_1} (z_{r_1}-z_{r_2})} a^{r_1}(z_{r_1} - z_{r_2})a^{r_2} \right) \otimes \Bigl( \bigotimes_{\ov{r} \in \ov{R}} a^{\ov{r}} \Bigr) \otimes m \otimes 1.
\end{multline}
Here and in (\ref{eq:cousin-coeff-part-2}) below, as before, the notation $\vec{\lambda}_{r_1}$ means all copies of $\lambda_{r_1}$ are passed to the right-hand tensor factor and are replaced by $\sum_{i \in I_{r_1}} \lambda_i$. Now, suppose on the other hand that $R = \ov R \cup \{r_0, r_1\} \in Q^*(I, k+2)$, and let $R' = \ov R \cup \{r_0\} \in Q^*(I, k+1)$ be the quotient obtained from $R$ by mapping $r_i \mapsto r_i$ for $i \neq 1$ and $r_1 \mapsto r_0$. Then $d_{[R], [R']}$ is the unique map of $\C[I]$-modules satisfying
\begin{multline} \label{eq:cousin-coeff-part-2}
f \otimes \Bigl( \bigotimes_{r \in R \backslash \{r_0\}} a^r \Bigr) \otimes m \otimes 1 \mapsto
\Bigl( \bigotimes_{\ov{r} \in \ov{R}} a^{\ov{r}} \Bigr) \otimes \left( \res_{z_{r_1} = 0} f \cdot e^{\vec{\lambda}_{r_1} z_{r_1}} a^{r_1}(z_{r_1})m \right) \otimes 1.
\end{multline}


\end{nolabel}

\begin{nolabel}
We now compute the chiral complex $C^{\text{ch}}(X, \CA_V, M)$ in terms of the complexes $\mathring{C}_I$. The procedure is the same as in Section \ref{no:global.derham} above. The varieties $\mathring{X}$ and hence $\mathring{X}^{I} \backslash \D$ are affine. It follows that, for the complex of $\CD_{X^I}$-modules $C(\CA_V, \CM)_{X^I}$ with support $X^{I \backslash *} \times \{x\}$, the de Rham cohomology $R\G(X^I, \DR(C(\CA_V, \CM)_{X^I}))$ is computed by the hypercohomology spectral sequence whose $E_1$-page is
\begin{align}\label{hypercoho.spec.seq.coeff}
E_1^{p, q} = H^q(\G(X^I, C(\CA_V, \CM)_{X^I}^p) \otimes_{\C[I]} D_I^\bullet) = H^q(\mathring{C}_I^p \otimes_{\C[I]} D_I^\bullet).
\end{align}
The chiral complex $C^{\text{ch}}(X, \CA_V, M)$ is then given as before by the homotopy colimit of the $\mathcal{S}^{\text{op}}$-diagram $I \mapsto R\G(X^I, \DR(C(\CA_V, \CM)_{X^I}))$.

For any $I$ the nonzero entries of the spectral sequence (\ref{hypercoho.spec.seq.coeff}) are concentrated in bidegrees $(p, q)$ for which $p \leq q \leq 0$. Therefore total degrees $0$ and $-1$ receive contributions only from terms in bidegree $(p, 0)$ where $p=0,-1$, that is, from
\begin{align*}
H^0(\mathring{C}_I^{p} \otimes_{\C[I]} D_I^\bullet) = \mathring{C}_I^{p} \otimes_{\C[I]} \C.
\end{align*}
The latter group is obtained from $\mathring{C}_I^{p}$ setting all $\lambda_i=0$.
\end{nolabel}

\begin{nolabel}\label{nl:n=2.SS.page}
We now work out $H^0(\mathring{C}_I \otimes_{\C[I]} D_I^\bullet)$, and the associated spectral sequence, somewhat explicitly for the case $I = \{1,2\}$. Firstly $\#Q^*(I) = 3$ and so  $\mathring{C}_I$ has three summands: one in cohomological degree $0$ and two in degree $-1$.

The unique pointed equivalence relation contributing to degree $0$ is $I \twoheadrightarrow *$ and we have
\begin{align*}
\mathring{C}_I^0 = M[\lambda] \otimes_{\C[*]} \C[I] = M \otimes_\C \C[\la_1, \lambda_2]
\end{align*}
(the action of $\lambda_*$ on $M[\lambda]$ is multiplication by $\lambda$). Hence
\begin{align*}
H^0(\mathring{C}_I^0 \otimes_{\C[I]} D_I^\bullet) = M.
\end{align*}

The two summands in degree $-1$ both correspond to the identity map $I \rightarrow I$ and are distinguished only by the choice of marked equivalence class. Let us denote by $R$ the quotient with $\{2\}$ as marked equivalence class. The corresponding summand is
\begin{align*}
\mathring{C}_{I, R} = (\mathring{\G}_1 \otimes V \otimes M[\lambda]) \otimes_{\C[I]} \C[I] = \mathring{\G}_1 \otimes V \otimes M[\lambda],
\end{align*}
in which the action of $\lambda_2$ is multiplication by $\lambda$, and the action of $\lambda_1$ is $\partial_{z_1} \otimes 1 \otimes 1 + 1 \otimes T \otimes 1$ which we abuse notation by writing as $\partial_{z_1} + T$. It follows that
\begin{align*}
H^0(\mathring{C}_{I, R} \otimes_{\C[I]} D_I^\bullet) = \frac{\mathring{\G}_1 \otimes V \otimes M}{\left<\partial_{z_1} + T\right>}.
\end{align*}
The differential $d_{[R], [*]}$ on $\mathring{C}_I$ is given by
\begin{align*}
f(z_1) \otimes a \otimes m \otimes 1 \mapsto \res_{z_1=0} f(z_1) e^{\vec{\lambda}_1 z_1} a(z_1)m \otimes 1.
\end{align*}
The functor $H^0(- \otimes_{\C[I]} D_I^\bullet)$ sets all $\la_i$ to $0$, so the differential becomes
\begin{align}\label{eq:h.of.diff}
f(z_1) \otimes a \otimes m \mapsto \res_{z_1=0} f(z_1) a(z_1)m = a_{(f)}m.
\end{align}
The other summand in degree $-1$ is similar, and $H^0(\mathring{C}_{I} \otimes_{\C[I]} D_I^\bullet)$ is
\begin{align}\label{eq:cx.on.X2}
\left( \frac{\mathring{\G}_1 \otimes V \otimes M}{\left<\partial_{z_1} + T_1\right>} \right)^{\oplus 2} \rightarrow M,
\end{align}
where the differential on both summands in degree $-1$ is given by (\ref{eq:h.of.diff}).

It is easy to check directly that $d$ is well-defined on the quotients using the relation $[Ta](x) = \partial_x a(x)$. Indeed by the formal Leibniz rule
\begin{align}\label{eq:well-def.mod.T}
(Ta)_{(f)}m
= \res_x f(x) (\partial_x a(x) m) = - \res_x (\partial_x f(x)) a(x) m = -a_{(\partial f)}m
\end{align}
as required.
\end{nolabel}

\begin{nolabel}
We illustrate the computation of $H_{0}^{\text{ch}}(X, \CA_V)$. By the quasi-isomorphism (\ref{eq:quasi443}) we have $H_{i}^{\text{ch}}(X, \CA) = H_{i}^{\text{ch}}(X, \CA, V)$. We compute the latter group setting $M=V$ in the formulas of Subsection \ref{nl:n=2.SS.page} above. Now the complex $C_\bullet^{\text{ch}}(X, \CA, M)$ is a colimit over an $\mathcal{S}^{\text{op}}$-diagram, but {\cite[Lemma 4.2.10]{BD.Chiral}} asserts that its cohomology in degrees $i=0,-1,\ldots,-n$ is computed correctly after truncation to the (finite) $\mathcal{S}^{\text{op}}_{\leq n+2}$-subdiagram. For $n=0$ the truncated diagram is
\[
\begin{tikzcd}
{H^0(\mathring{C}_{*} \otimes_{\C[*]} D_*^\bullet)} \ar[r] &  {H^0(\mathring{C}_{I} \otimes_{\C[I]} D_I^\bullet)}\ar[loop right]{}{\sigma}, 
\end{tikzcd}
\]
where $I = \{1,2\}$ and $\sigma$ is the automorphism of the complex (\ref{eq:cx.on.X2}) that exchanges the two summands in degree $-1$. It is easy to see that $H^0(\mathring{C}_{*} \otimes_{\C[*]} D_*^\bullet)$ consists of $V$ concentrated in degree $0$, and the morphism from it to $H^0(\mathring{C}_{I} \otimes_{\C[I]} D_I^\bullet)$ is just the identity on $V$. Upon passage to the colimit the summands exchanged by $\sigma$ are identified, yielding the complex
\begin{align}\label{eq:cx.on.X2.aftercolimit}
\frac{\mathring{\G}_1 \otimes V \otimes V}{\left<\partial_{z_1} + T_1\right>} \rightarrow V
\end{align}
with differential (\ref{eq:h.of.diff}). Therefore
\begin{align*}
H_0^{\text{ch}}(X, \CA_V) \cong \frac{V}{\left< a_{(f)}b \mid a, b \in V, f \in \mathring{\G}_1 \right>},
\end{align*}
recovering the coinvariants (\ref{Hdef}) (vector space dual to conformal blocks).
\end{nolabel}

\begin{nolabel}
Now we compute $\mathring{C}_I$ and its de Rham cohomology for $I = \{1,2,3\}$ in order to produce an explicit complex similar to (\ref{eq:cx.on.X2.aftercolimit}). By {\cite[Lemma 4.2.10]{BD.Chiral}} setting $M=V$ in this complex will compute $H_{i}^{\text{ch}}(X, \CA_V)$ for $i=0,1$. In this case $\#Q^*(I)=10$ and so $\mathring{C}_I$ has $10$ summands: one in cohomological degree $0$, six in degree $-1$ and three in degree $-2$.

The unique pointed equivalence relation contributing to degree $0$ is $I \twoheadrightarrow *$ and we have
\begin{align*}
\mathring{C}_I^0 = M[\lambda] \otimes_{\C[*]} \C[I] = M \otimes_\C \C[\la_1, \lambda_2, \lambda_3]
\end{align*}
(the action of $\lambda_*$ on $M[\lambda]$ is multiplication by $\lambda$). Hence
\begin{align*}
H^0(\mathring{C}_I^0 \otimes_{\C[I]} D_I^\bullet) = M.
\end{align*}

Next we consider two pointed equivalence relations in degree $-1$. The first, which we denote $R_2$ is $\pi : I \twoheadrightarrow \{\alpha, \beta\}$ defined by $\pi(1) = \pi(2) = \alpha$ and $\pi(3) = \beta$, with marked equivalence class $\{3\}$. The second, which we denote $R_2'$ is $\pi' : I \twoheadrightarrow \{\alpha, \beta\}$ defined by $\pi'(2) = \alpha$ and $\pi'(1) = \pi'(3) = \beta$, with marked equivalence class $\{1, 3\}$. All six pointed equivalence relations in degree $-1$ are obtained from $R_2$ and $R_2'$ via permutations of $I$. We now have
\begin{align*}
\mathring{C}_{I, R_2} = (\mathring{\G}_1 \otimes V \otimes M[\lambda]) \otimes_{\C[\lambda_\alpha, \lambda_\beta]} \C[I] = (\mathring{\G}_1 \otimes V \otimes M) \otimes_{\C[\lambda_\alpha]} \C[\lambda_1, \lambda_2, \lambda_3]
\end{align*}
where $\lambda_\alpha = \lambda_1 + \lambda_2$ acts on $\mathring{\G}_1 \otimes V \otimes M$ as $\partial_{z} \otimes 1 \otimes 1 + 1 \otimes T \otimes 1$ and $\lambda_\beta = \lambda_3$ acts as multiplication by $\lambda$. Similarly
\begin{align*}
\mathring{C}_{I, R_2'} = (\mathring{\G}_1 \otimes V \otimes M[\lambda]) \otimes_{\C[\lambda_\alpha, \lambda_\beta]} \C[\la_1, \lambda_2, \lambda_3]
\end{align*}
where $\lambda_\alpha = \lambda_2$ acts on $\mathring{\G}_1 \otimes V \otimes M$ as $\partial_{z} \otimes 1 \otimes 1 + 1 \otimes T \otimes 1$ and $\lambda_\beta = \lambda_1 + \lambda_3$ acts as multiplication by $\lambda$. Setting $\la_i=0$ for $i=1,2,3$ now yields
\begin{align}\label{eq:I3.deg2.summands}
H^0(\mathring{C}_{I, R_2} \otimes_{\C[I]} D_I^\bullet) = H^0(\mathring{C}_{I, R_2'} \otimes_{\C[I]} D_I^\bullet) = \frac{\mathring{\G}_1 \otimes V \otimes M}{\left<\partial_{z} + T_1 \right>}
\end{align}
(here we are writing $T_1$ for $1 \otimes T \otimes 1$).

The three summands in degree $-2$ all correspond to the identity map $I \rightarrow I$ and are distinguished only by the choice of marked equivalence class. Let us denote by $R_3$ the quotient with $\{3\}$ as marked equivalence class. The corresponding summand is
\begin{align*}
\mathring{C}_{I, R_3} = (\mathring{\G}_2 \otimes V \otimes V \otimes M[\lambda]) \otimes_{\C[I]} \C[I] = \mathring{\G}_2 \otimes V \otimes V \otimes M[\lambda],
\end{align*}
in which the action of $\lambda_3$ is multiplication by $\lambda$, and for $i=1,2$ the action of $\lambda_i$ is $\partial_{z_i} + T_i$. Setting $\la_i=0$ for $i=1,2,3$ now yields
\begin{align*}
H^0(\mathring{C}_{I, R_3} \otimes_{\C[I]} D_I^\bullet) = \frac{\mathring{\G}_2 \otimes V \otimes V \otimes M}{\left<\partial_{z_i} + T_i\right>_{i=1,2}}.
\end{align*}

The differential $d_{[R_3], [R_2]}$ is given by
\begin{align*}
f(z_1, z_2) \otimes a \otimes b \otimes m \otimes 1 \mapsto \left( \res_{z_1=z_2} f(z_1, z_2) e^{\vec{\la}_1(z_1-z_2)} a(z_1-z_2)b \right) \otimes m \otimes 1.
\end{align*}
The differential $d_{[R_3], [R_2']}$ is given by
\begin{align*}
f(z_1, z_2) \otimes a \otimes b \otimes m \otimes 1 \mapsto b \otimes \left( \res_{z_1=0} f(z_1, z_2) e^{\vec{\la}_1z_1} a(z_1)m \right) \otimes 1.
\end{align*}
After passing to $H^0(- \otimes_{\C[I]} D_I^\bullet)$ these differentials become
\begin{align}\label{eq:diff.I3.deg2-deg1}
\begin{split}
f(z_1, z_2) \otimes a \otimes b \otimes m &\mapsto \res_{z_1=z_2} f(z_1, z_2) a(z_1-z_2)b \otimes m \\
\text{and} \quad f(z_1, z_2) \otimes a \otimes b \otimes m &\mapsto \res_{z_1=0} f(z_1, z_2) b \otimes a(z_1) m
\end{split}
\end{align}
respectively (cf. (\ref{eq:h.of.diff})).

The chiral homology groups $H_i^{\text{ch}}(X, \CA_V)$, $i=0,1$ are computed from the  colimit of the $\mathcal{S}^{\text{op}}_{\leq 3}$-diagram
\[
\begin{tikzcd}
{H^0(\mathring{C}_{*} \otimes_{\C[*]} D_*^\bullet)} \ar[r] & {H^0(\mathring{C}_{\{1,2\}} \otimes_{\C[\{1,2\}]} D_{\{1,2\}}^\bullet)}\ar[loop above]{}{\sigma} \ar[r] & {H^0(\mathring{C}_{I} \otimes_{\C[I]} D_I^\bullet)}\ar[loop right]{}{\Sigma_3} 
\end{tikzcd}
\]
(all morphisms are those induced by permutations of factors in powers of $X$).

We now describe the colimit $A^\bullet$ of the above diagram concretely, summarising the result in Proposition \ref{prop:chhom.equals.A} below. In degree $0$ the diagram consists of copies of $M$ and identity morphisms, so in the colimit $A^0 = M$. In degree $-1$ the diagram consists of copies of $\mathring{\Gamma}_1 \otimes V \otimes V / \left<\partial_{z_1} + T_1\right>$ which are all identified under the symmetric group actions. Indeed in degrees $0$ and $-1$ the colimit is precisely the complex (\ref{eq:cx.on.X2.aftercolimit}). Finally in degree $-2$ we obtain a quotient of
\[
\frac{\mathring{\G}_2 \otimes V \otimes V \otimes V}{\left<\partial_{x} + T_1, \partial_y + T_2 \right>},
\]
by an action of $\Sigma_2$ which exchanges the sections
\begin{align*}
f(x,y) \cdot a \otimes b \otimes m \quad \text{and} \quad -f(y, x) \cdot b \otimes a \otimes m.
\end{align*}
Having performed these quotients, the components (\ref{eq:diff.I3.deg2-deg1}) become the differential $d_2 : A^{-2} \rightarrow A^{-1}$ defined in (\ref{2.1.coeff}) below. One may confirm that $d_2$ descends to the quotient due to the skew-symmetry identity (\ref{skewsymmetry.identity}). Indeed
\begin{align*}
\res_{y=x} f(y, x) a(y-x)b
&= \res_{t} f(x+t, x) a(t)b \\
&= \res_{t} f(x+t, x) e^{tT} b(-t)a \\
&= -\res_{t} f(x-t, x) e^{-tT} b(t)a \\
&\equiv -\res_{t} f(x-t, x) e^{t\partial_x} b(t)a \\
&= -\res_{t} f(x, x+t) b(t)a \\
&= - \res_{y=x} f(x, y) b(y-x)a,
\end{align*}
so in $A^{-1}$ we have
\begin{align}\label{eq:skew-kills}
d\left( f(x,y) \cdot a \otimes b \otimes m + f(y,x) \cdot b \otimes a \otimes m \right) = 0.
\end{align}
The well-definedness of $d_1$ and $d_2$, although it follows by construction, may be verified by calculations similarly to (\ref{eq:well-def.mod.T}). The fact that $d_1 \circ d_2 = 0$ may also be verified directly from the Jacobi identity (\ref{VAJacobi.identity}). 
In summary we have proved the following proposition.
\end{nolabel}

\begin{prop}\label{prop:chhom.equals.A}
Let $\mathring{\G}_1$ and $\mathring{\G}_2$ be the spaces of meromorphic functions on powers of the elliptic curve $X = X_q$ as defined in (\ref{eq:G.ring.def}) above. Let $V$ be a quasiconformal vertex algebra and let $A^\bullet$ be the complex given by
\begin{align}\label{eq:A.comp.def}
A^0 &= V, \qquad
A^{-1} = \frac{\mathring{\G}_1 \otimes V \otimes V}{\left<\partial_{x} + T_1\right>} \quad \text{and} \quad
A^{-2} = \frac{\mathring{\G}_2 \otimes V \otimes V \otimes V}{\left<\partial_{x} + T_1, \partial_y + T_2 \right>},
\end{align}
with the differentials $d_1 : A^{-1} \rightarrow A^0$ and $d_2 : A^{-2} \rightarrow A^{-1}$ given by
\begin{align}
d_1\left( f(x) a \otimes m \right) = {} & \res_x f(x)a(x)m = a_{(f)} m, \label{1.0.coeff} \\
\begin{split}
d_2 \left( f(x,y) a \otimes b \otimes m \right) = {} & 
 \res_{y} f(x,y) a \otimes b(y)m \\
& - \res_{y=x} f(y,x) a(y-x) b \otimes m - \res_y f(y,x) b \otimes a(y) m. 
\end{split}\label{2.1.coeff}
\end{align}
Then for $i=0, 1$ there exist isomorphisms
\begin{align*}
H_i^{\text{ch}}(X, \CA_V) \cong H^{-i}(A^\bullet).
\end{align*}
\end{prop}

%
%

\section{Elliptic Functions}\label{sec:elliptic.functions}

In this section we recall some background material on elliptic functions as well as several identities which will be used in later sections to study the complex $A^\bullet$ in greater detail.

\begin{nolabel}\label{no:apa.1}
We recall the Eisenstein series $G_{k} \in \C[[q]]$ defined for $k \geq 1$ to be
\begin{align*}
G_{k} = (2\pi i)^{k} \left( -\frac{B_{k}}{k!} + \frac{2}{(k-1)!} \sum_{n=1}^\infty n^{k-1} \frac{q^n}{1-q^n} \right)
\end{align*}
for $k$ even, and $0$ for $k$ odd. Here $B_n$ are the Bernoulli numbers, defined by $t / (e^t-1) = \sum_{n=1}^\infty B_n t^n / n!$. For $k \geq 4$ the Eisenstein series $G_{k}$ is a modular form of weight $k$, while $G_2$ is a quasimodular form of weight $2$.

We also recall the Weierstrass elliptic function $\wp$. For us it will be convenient to put
\begin{align}\label{eq:def.Weierstrass.p}
\wp(z, q) = z^{-2} + \sum_{k=0}^\infty (2k+1) G_{2k+2} z^{2k},
\end{align}
which differs from the standard normalisation by the additive constant $G_2$. The Weierstrass quasielliptic function $\zeta$ is similarly defined to be
\begin{equation}
\begin{aligned}
\zeta(z, q) = z^{-1} - \sum_{k=0}^\infty G_{2k+2} z^{2k+1}.
\end{aligned}
\label{eq:aliased}
\end{equation}
We clearly have the relation $\wp(z) = -\partial_z \zeta(z)$. Note that $\wp$ is an even function of $z$ and $\zeta$ an odd function of $z$.

Let $\tau \in \HH$ the upper half complex plane and put $q = e^{2\pi i \tau}$. Then the two series above may be viewed as Laurent series expansions of meromorphic functions $\wp(z)$ and $\zeta(z)$ with poles for $z$ in the lattice $\Lambda_\tau = \Z 1 + \Z \tau$. We have
\begin{align*}
\wp(z+1, q) = \wp(z+\tau, q) = \wp(z, q),
\end{align*}
in other words $\wp$ is an elliptic function, and
\begin{align}\label{eq:quasielliptic}
\zeta(z+1, q) = \zeta(z, q), \quad \text{while} \quad \zeta(z+\tau, q) = \zeta(z, q) - 2\pi i.
\end{align}

The following identities are valid:
\begin{align*}
\zeta(z, q) &= 2\pi i \left(\frac{e^{2\pi i z}}{e^{2\pi i z}-1} - \frac{1}{2} - \sum_{n=1}^\infty \frac{q^n}{1-q^n} \left[ e^{2\pi i n z} - e^{-2\pi i n z}\right] \right) \\
\text{and} \quad
\wp(z, q) &= (2\pi i)^2 \left(\frac{e^{2\pi i z}}{(e^{2\pi i z}-1)^2} + \sum_{n=1}^\infty \frac{n q^n}{1-q^n} \left[ e^{2\pi i n z} + e^{-2\pi i n z}\right] \right),
\end{align*}
and from them it is clear that
\begin{align}\label{eq:zeta.to.f.reduction}
\begin{split}
\zeta(z, q=0) &= 2\pi i \frac{e^{2\pi i z}}{e^{2\pi i z}-1} - \pi i = f(z) - \pi i \\
\quad \text{and} \quad
\wp(z, q=0) &= (2\pi i)^2 \frac{e^{2\pi i z}}{(e^{2\pi i z}-1)^2} = g(z).
\end{split}
\end{align}
Here $f(z)$ and $g(z)$ are as defined in equation (\ref{f.and.g.def}).
\end{nolabel}

\begin{lem}\label{lem:deq}
The following identity holds
\begin{equation}
-(2 \pi i)^2 q \frac{d \zeta}{dq} = \zeta \wp + \frac{1}{2} \wp'. 
\label{eq:deq}
\end{equation}
\end{lem}
\begin{proof}
It is proved in \cite[Lemma 4.6]{Zhu.Note} that the difference between the two sides of (\ref{eq:deq}) is an elliptic function. The claimed equality follows upon computing polar parts at $z=0$ of both sides and confirming that they agree.
\end{proof}

\begin{nolabel}\label{sec:ZZZ}
In this section we introduce some elliptic functions of three variables and some useful identities among them. Firstly it is clear that the function
\begin{align*}
\zuta(x, y, z) = \zeta(x-y) + \zeta(y-z) + \zeta(z-x)
\end{align*}
is elliptic in the three variables $x, y$ and $z$. Next we define
\begin{align*}
\ZZZ(x, y, z) = \wp(x-y) \left[ \zeta(x-y) + \zeta(y-z) + \zeta(z-x) \right] + \frac{1}{2}\wp'(x-y).
\end{align*}
\end{nolabel}

\begin{lem}\label{surprising.symmetry}
The function $\ZZZ(x, y, z)$ is cyclically symmetric in the variables $x$, $y$ and $z$.
\end{lem}
\begin{proof}
The Weierstrass function $\sigma$ is defined by
\begin{align*}
\sigma(z) = z \prod_{\omega \in \Lambda_\tau \backslash 0} \left(1 - \frac{z}{\omega}\right) e^{z/\omega + \frac{1}{2}(z/\omega)^2},
\end{align*}
and satisfies
\begin{align*}
\frac{\sigma'(z)}{\sigma(z)} = \zeta(z) + G_2 z.
\end{align*}
Now we recall the identity {\cite[p. 243]{Lang.Book.elliptic}}
\begin{align*}
\wp(u) - \wp(v) = - \frac{\sigma(u+v)\sigma(u-v)}{\sigma^2(u)\sigma^2(v)}.
\end{align*}
Taking the logarithmic derivative (in $u$) of this identity gives
\begin{align*}
  \frac{\wp'(u)}{\wp(u) - \wp(v)}
  &=  \frac{d}{du} \left[ \log{\sigma(u+v)} + \log{\sigma(u-v)} - 2 \log{\sigma(u)} \right] \\
&= \zeta(u+v) + \zeta(u-v) - 2 \zeta(u).
\end{align*}
Combining this with a similar logarithmic derivative in $v$ yields
\begin{align*}
\frac{\wp'(u) - \wp'(v)}{\wp(u) - \wp(v)} = 2\zeta(u+v) - 2 \zeta(u) - 2\zeta(v),
\end{align*}
which establishes the desired cyclic symmetry.
\end{proof}

\begin{nolabel}
The cyclic symmetry of $\ZZZ(x, y, z)$ is made manifest by the identity asserted in Lemma \ref{manifest.symmetry} below. To state the identity we need to recall the Jacobi theta functions {\cite[Chapter V]{Chandra}}
\begin{align*}
\theta(z, \tau) &= -i \cdot \sum_{n \in \Z} (-1)^n q^{(n+1/2)^2/2} e^{2\pi i (n+1/2) z}, \\
\theta_1(z, \tau) &= \sum_{n \in \Z} q^{(n+1/2)^2/2} e^{2\pi i(n+1/2) z}, \\
\theta_2(z, \tau) &= \sum_{n \in \Z} (-1)^n q^{n^2/2} e^{2\pi i n z}, \\
\theta_3(z, \tau) &= \sum_{n \in \Z} q^{n^2/2} e^{2\pi i n z}.
\end{align*}
and the Jacobi elliptic function {\cite[p. 100]{Chandra}}
\begin{align*}
\sn(z, \tau) = \frac{1}{\pi \cdot \theta_1(0, \tau) \cdot \theta_3(0, \tau)} \cdot \frac{\theta(z, \tau)}{\theta_2(z, \tau)}.
\end{align*}
One has the following identity relating the Weierstrass and Jacobi elliptic functions {\cite[p. 102]{Chandra}}
\begin{align}\label{Weierstrass.Jacobi}
\sn(z, \tau)^2 = \frac{1}{\wp(z, \tau) - e(\tau)},
\end{align}
where the half period value $e(\tau) = \wp(\tau/2, \tau)$ can in fact be expressed explicitly as
\begin{align*}
e(\tau) = -8\pi^2 \sum_{n=1}^\infty \frac{n \, q^{n/2}}{1-q^{n}}.
\end{align*}

\end{nolabel}
\begin{lem}\label{manifest.symmetry}
The functions
\begin{align}\label{eq:man.symm}
\ZZZ(x, y, z) - e(\tau)\zuta(x, y, z) \quad \text{and} \quad -\frac{1}{\sn(x-y)\sn(y-z)\sn(z-x)}
\end{align}
differ by a constant.
\end{lem}

\begin{proof}
Equation (\ref{Weierstrass.Jacobi}) implies that the function $\sn(z)^{-1}$ has a simple pole at $z=0$ with residue $1$. This can also be seen from Jacobi's formula $\theta'(0, \tau) = \pi \cdot \theta_1(0, \tau) \cdot \theta_2(0, \tau) \cdot \theta_3(0, \tau)$ (here $\theta'$ denotes the derivative with respect to the first variable).

Consequently the right-hand side of (\ref{eq:man.symm}), regarded as a function of $x$, has a simple pole at $x=y$ with residue $\sn(y-z)^{-2}$ (here we have used that $\sn(z)$ is an odd function of $z$). On the other hand $\ZZZ(x, y, z)$, regarded as a function of $x$, has a simple pole at $x=y$ with residue $\wp(y-z)$.

The identity (\ref{Weierstrass.Jacobi}) shows that the difference, call it $\beta$, between the two sides of (\ref{eq:man.symm}) is regular at $x=y$. Similarly $\beta$ is regular at $x=z$, and hence is an elliptic function of $y-z$. By Lemma \ref{surprising.symmetry} and the manifest symmetry of the right-hand side of (\ref{eq:man.symm}), $\beta$ is cyclically symmetric in $x, y, z$. As a cyclically symmetric function independent of $x$, $\beta$ is constant as claimed.
\end{proof}

\section{Configuration Spaces of Elliptic Curves}\label{sec:totaro}

In Section \ref{sec:ch.hom.ell} we produced the complex $A^\bullet$ which computes the chiral homology groups $H^{\text{ch}}_i(X, \CA_V)$ of $X$ an elliptic curve for $i=0,1$. To aid in explicit calculations with $A^\bullet$ it is useful to study in more detail the spaces of functions $\mathring{\G}_I$, introduced in (\ref{eq:G.ring.def}), which figure in its definition. This is carried out in the present section.

\begin{nolabel}
Let $X$ be a smooth complex projective algebraic curve, and $H^{\bullet}(X^n)$ the cohomology ring of $X^n$. Let $\pi_i : H^{\bullet}(X) \rightarrow H^{\bullet}(X^n)$ denote the pullback morphism associated with the projection $X^n \rightarrow X$ to the $i^{\text{th}}$ component, and similarly $\pi_{i,j} : H^{\bullet}(X^2) \rightarrow H^{\bullet}(X^n)$ for $i \neq j$. The class in $H^2(X^2)$ represented by the diagonal copy of $X$ in $X^2$ is denoted $\D$. The subvariety $U^{(n)} \subset X^n$ is defined as the complement of all diagonal divisors.

A special case of a theorem of Totaro {\cite[Theorem 4]{Totaro}} (see also \cite{brownlevin,knudsen,maguire,sch16}) asserts that $H^{\bullet}(U^{(n)})$ is isomorphic to the cohomology of the graded-commutative algebra
\begin{align*}
\left(H^{\bullet}(X^n) \otimes \C[G_{i, j}]_{1 \leq i, j \leq n, i \neq j}\right) / J,
\end{align*}
with degree $1$ generators $G_{i, j}$, differential $d$ defined by $d(G_{i,j}) = \pi_{i,j}(\D)$, and where $J$ is the ideal generated by the elements
\begin{align*}
&G_{i,j} - G_{j,i}, \\
&G_{i,j}G_{i,k} + G_{j,k}G_{j,i} + G_{k,i}G_{k,j}, \\
&[\pi_i(x) - \pi_j(x)] G_{i,j}, \quad \text{for all $x \in H^\bullet(X)$}.
\end{align*}

We now pass to the case of $X$ an elliptic curve. We recall the notation $\mx = X \backslash 0$, and we write $\mathring{U}^{(n)}$ for the complement of all diagonal divisors in $\mx^n$. The group structure of the elliptic curve $X$ yields an isomorphism of varieties
\begin{align}\label{proj.to.aff}
U^{(n+1)} \rightarrow X \times \mathring{U}^{(n)},
\end{align}
for any $n \geq 0$, given by
\begin{align*}
(x_0, x_1, \ldots, x_n) \mapsto (x_0, x_1-x_0, \ldots, x_n-x_0).
\end{align*}
The Betti numbers $h^k(U^{(n)}) = \dim_\C(H^k(U^{(n)}))$ may be computed for small values of $n$ using Totaro's presentation, and using (\ref{proj.to.aff}) we may infer the Betti numbers of $\mathring{U}^{(n)}$ from those of $U^{(n)}$. We are particularly interested in the top Betti number $h^n(\mathring{U}^{(n)})$ which we compute for small values of $n$, obtaining the values listed in the following table.
\begin{align*}
\begin{array}{c|cccccc}
n & 0 & 1 & 2 & 3 & 4 & 5 \\
\hline 
h^{n}(\mathring{U}^{(n)}) & 1 & 2 & 5 & 18 & 79 & 432 \\
\end{array}
\end{align*}
These computations were performed with the computer algebra system \texttt{SAGE} \cite{sagemath}.

Now let $Y$ be a smooth algebraic variety of dimension $n$. At the level of complexes of vector spaces the equality
\begin{align}\label{derham.betti}
H^\bullet(Y)[n] = H^\bullet(R\G_{\DR}(\om_Y))
\end{align}
holds, and so $\dim_\C(H^k(R\G_{\DR}(\om_Y))) = h^{k+n}(Y)$. Now we study
\begin{align*}
\mathring\G_n = \Gamma(\mx^n, \mathring{j}_*\mathring{j}^*\OO_{\mx^n}),
\end{align*}
where $j$ is the embedding $\mathring{U}^{(n)} \rightarrow \mx^n$, and we put $Y = \mathring{U}^{(n)}$. The sheaf $\om_Y$ is a free $\OO_Y$-module of rank $1$ with generator $\nu = dx_1 \wedge \ldots \wedge dx_n$ where $dx_i = \pi_i^*(dz)$ is the pullback of the global $1$-form $dz$ on $X$ (as in Section \ref{sec:conformal.blocks}). The action (\ref{eq:Lie.derivative}) of $\Theta_Y$ on $\om_Y$ is given by
\begin{align*}
(f \nu) \cdot g \partial_{x_i} = -\partial_{x_i}(fg) \nu,
\end{align*}
and so
\begin{align*}
H^0(R\G_{\DR}(j_*j^*\om_{\mx^n})) \cong H^0(R\G_{\DR}(\om_{Y})) \cong \mathring\G_n / \left<\partial_{x_i}\mathring\G_n\right>_{i = 1,\ldots, n}.
\end{align*}
By (\ref{derham.betti}) if follows that
\begin{align*}
\CF_{n} := \mathring\G_n / \left<\partial_{x_i}\mathring\G_n\right>_{i = 1,\ldots, n}
\end{align*}
has dimension $h^n(\mathring{U}^{(n)})$.

\end{nolabel}
\begin{nolabel}\label{funct.induc.sec}
We construct bases of the vector spaces $\CF_{n}$ inductively. If we view $f(x_1, \ldots, x_{n+1}) \in \mathring{\G}_{n+1}$ as a function of $x_{n+1}$ with possible poles at $x_1, \ldots x_n$ and $0$ then by Liouville's theorem we have
\begin{align}\label{eq:Fn.decomp}
\begin{split}
f(x_1, \ldots x_{n+1}) = {} & c(x_1,\ldots x_n) + \sum_{i=0}^n \alpha_{i}(x_1,\ldots x_n) \zeta(x_{n+1} - x_i) \\
&+ \sum_{k \in \Z_{\geq 0}} \sum_{i=0}^n \beta_{i, k}(x_1,\ldots x_n) \wp^{(k)}(x_i-x_{n+1})
\end{split}
\end{align}
 for some collection of functions $c, \al_i, \beta_{i, k} \in \mathring\G_n$. Here $x_0$ stands for $0$, for notational simplicity, and $\wp^{(k)}(z)$ denotes $\partial_z^{(k)} \wp(z) = \tfrac{1}{k!}\partial_z^k \wp(z)$. Now for $k > 0$ we have that
\begin{align*}
\beta(x_1,\ldots x_n) \wp^{(k)}(x_i-x_{n+1}) = -\tfrac{1}{k} \partial_{x_{n+1}} \left( \beta(x_1,\ldots x_n) \wp^{(k-1)}(x_i-x_{n+1}) \right),
\end{align*}
is a total derivative and so vanishes in the quotient $\CF_{n+1}$. Hence, at the level of $\CF_{n+1}$, all terms with $k>0$ in the second summation of (\ref{eq:Fn.decomp})  can be discarded.

In the first summation of (\ref{eq:Fn.decomp}) the condition $\sum_{i=0}^n \alpha_i = 0$ must hold by the residue theorem, and so the sum may be replaced by one of the form
\begin{align*}
\sum_{i=0}^{n-1} \gamma_{i}(x_1,\ldots x_n) \zuta(x_i,x_{i+1},x_{n+1}),
\end{align*}
with a corresponding modification of $c$. Next we note that the difference
\begin{align*}
\beta(x_1,\ldots x_n) \left( \wp(x_i-x_{n+1}) - \wp(x_{j}-x_{n+1}) \right) = \partial_{x_{n+1}} \left(\beta(x_1,\ldots x_n) \zuta(x_i, x_j, x_{n+1}) \right)
\end{align*}
is a total derivative. So we may write
\begin{align}\label{eq:Fn.decomp.2}
\begin{split}
f(x_1, \ldots x_{n+1}) = {} & c(x_1,\ldots x_n) + \sum_{i=0}^{n-1} \gamma_{i}(x_1,\ldots x_n) \zuta(x_i,x_{i+1},x_{n+1}) \\
&+ \beta(x_1,\ldots x_n) \wp(x_n-x_{n+1}).
\end{split}
\end{align}
It is straightforward to show that modification of any of $c, \beta, \alpha_i$ in (\ref{eq:Fn.decomp.2}) by a total derivative modifies $f$ by a total derivative.

\end{nolabel}
\begin{lem}\label{funct.induc}
A set of generators of $\CF_{n+1}$ is furnished by the functions $c$, $\beta \wp(x_n-x_{n+1})$ $\gamma_i \zuta(x_i,x_{i+1},x_{n+1})$, where $c, \beta$, and $\gamma_i$, $i=0,\ldots n-1$ run over any set of generators of $\CF_n$.
\end{lem}
\begin{nolabel}
Let $b_n = h^{n}(\mathring{U}^{(n)})$. From Lemma \ref{funct.induc} it follows that $b_n \leq (n+1)b_{n-1}$. In fact the construction of Section \ref{funct.induc.sec} may be refined to yield $b_n \leq (n+1)b_{n-1} - b_{n-2}$.
\end{nolabel}
\begin{lem}
The classes in $\CF_2$ of the the five functions
\begin{align*}
1, \quad \zuta(0, x_1, x_2), \quad \wp(x_1), \quad \ZZZ(0, x_1, x_2), \quad \text{and} \quad \wp(x_1)\wp(x_2)
\end{align*}
constitute a basis.
\end{lem}
\begin{proof}
Applying the construction of Lemma \ref{funct.induc} to the basis $\{1, \wp(x_1)\}$ of $\CF_1$ yields the set of functions
\begin{align*}
&1, &&\wp(x_1) \\
&\zuta(0, x_1, x_2), &&\wp(x_1)\zuta(0, x_1, x_2), \\
&\wp(x_2), &&\wp(x_1)\wp(x_2).
\end{align*}
We may omit $\wp(x_2)$ since it is equivalent to $\wp(x_1)$ modulo a total derivative. Since $h^2(\mathring{U}^{(2)}) = 5$ the remaining functions form a basis of $\CF_2$. The lemma follows immediately.
\end{proof}

\section{Explicit Differentials in low Degree}\label{sec:diff}

\begin{nolabel}
We have introduced the complex $A^\bullet = A^\bullet(q)$ in Proposition \ref{prop:chhom.equals.A}, and we have shown that its cohomology computes the chiral homology groups $H_i^{\text{ch}}(X, \CA_V)$ for $i=0,1$. We now combine the material of Sections \ref{sec:ch.hom.ell}, \ref{sec:elliptic.functions} and \ref{sec:totaro} to write $A^\bullet$ in more explicit terms. Recall
\begin{align*}
A^{-n} = \frac{\mathring{\G}_n \otimes V^{\otimes n+1}}{\left<\partial_{x_i}+T_i\right>_{i=1,\ldots,n}}
\end{align*}
for $n=0,1,2$. We thus denote elements of $A^{-1}$ by a representative of the form $f(x) \cdot a \otimes m$ and elements of $A^{-2}$ in the form $f(x, y) \cdot a \otimes b \otimes m$. For $n=1,2$ the relations in $A^{-n}$ may be used to `trade' a total derivative in $\mathring{\G}_n$ for a copy of $-T$ acting on one of the first $n$ tensor factors of $V^{\otimes n+1}$. In $A^{-1}$ for example $1 \cdot Ta \otimes m = 0$ and in fact as vector spaces
\begin{align}\label{eq:Am1.reduced}
A^{-1} \cong 1 \cdot (V/TV) \otimes V \oplus \wp(x) \cdot V \otimes V.
\end{align}
Similarly $A^{-2}$ is a certain quotient of the vector space $\C^5 \otimes V \otimes V \otimes V$ where the basis of $\C^5$ is labelled by the symbols
\begin{align*}
1, \quad \zuta(x,y,0), \quad \wp(x), \quad \ZZZ(x,y,0), \quad \text{and} \quad \wp(x)\wp(y).
\end{align*}
At certain points it will be convenient to work with terms of the form $\wp(x-y) \cdot a \otimes b \otimes m$ in $A^{-2}$ as well, even though they are expressible in terms in the five generators listed above. Indeed since $\partial_x \zuta(x, y, 0) = \wp(x)-\wp(x-y)$ we have
\begin{align}\label{eq:relation.in.A2}
\wp(x-y) \cdot a \otimes b \otimes m = \wp(x) \cdot a \otimes b \otimes m + \zuta(x, y, 0) \cdot Ta \otimes b \otimes m.
\end{align}

Since the differential $d_1 : A^{-1} \rightarrow A^0$ is given by (\ref{1.0.coeff}) it follows that the $0^{\text{th}}$ cohomology of $A^\bullet$ is
\begin{align*}
H_0^{\text{ch}}(X, \CA_V) = H^0(A^\bullet) = \frac{V}{V(0)V + V_{(\wp)}V}.
\end{align*}
In the following lemma we use (\ref{2.1.coeff}) to compute the images under the differential $d_2 : A^{-2} \rightarrow A^{-1}$ of our chosen set of generators of $A^{-2}$.
\end{nolabel}

\begin{lem}\label{lem:2.1.diff}
The differential $d_2 : A^{-2} \rightarrow A^{-1}$ is given explicitly on generators as follows.
\begin{align} \label{1.diff}
\begin{split}
d_2\left(1 \cdot a \otimes b \otimes m\right)
= {} & a \otimes b(0)m - a(0)b \otimes m - b \otimes a(0)m,
\end{split}
\end{align}

\begin{align}\label{wp1.diff}
\begin{split}
d_2\left(\wp(x-y) \cdot a \otimes b \otimes m\right)
= {} & - a_{(\wp)}b \otimes m \\
&\hspace{-2cm}+ \wp(x) \sum_{j \in \Z_{\geq 0}} \left( T^{(j)} a \otimes b{(j)}m - T^{(j)}b \otimes a{(j)}m \right),
\\
\end{split}
\end{align}

\begin{align}\label{wp2.diff}
\begin{split}
 d_2\left(\wp(x) \cdot a \otimes b \otimes m\right)
= {} & -b \otimes a_{(\wp)}m \\
&+ \wp(x) \left( a \otimes b{(0)}m + b(0)a \otimes m \right),
\\
\end{split}
\end{align}

\begin{align}\label{zeta.diff}
\begin{split}
 d_2\left(-\zuta(x,y,0)) \cdot a \otimes b \otimes m\right) 
= {} & a_{(\zeta)}b \otimes m - b \otimes a_{(\zeta)}m - a \otimes b_{(\zeta)}m \\
&\hspace{-2cm}+ \wp(x) \int{\{a(0)b\}} \otimes m \\
&\hspace{-2cm}-\wp(x) \sum_{j \in \Z_{\geq 0}} \frac{1}{j+1} \left(T^{(j)}a \otimes b(j+1)m + T^{(j)}b \otimes a(j+1)m \right),
\\
\end{split}
\end{align}

\begin{align}\label{ZZZ.diff}
\begin{split}
 d_2\left(\ZZZ(x, y, 0) \cdot a \otimes b \otimes m\right) 
= {} & -\wp(x) \int{\{a_{(\wp)}b\}} \otimes m \\
&\hspace{-2cm}+\wp(x) \sum_{j \in \Z_{\geq 0}} \frac{1}{j+1} \left(T^{(j)}a \otimes b_{(x^{j+1}\wp)}m + T^{(j)}b \otimes a_{(x^{j+1}\wp)}m \right) \\
&\hspace{-2cm}+(2\pi i)^2 q \frac{d}{dq} \left( a_{(\zeta)}b \otimes m - a \otimes b_{(\zeta)}m - b \otimes a_{(\zeta)}m \right),
\\
\end{split}
\end{align}

\begin{align}\label{wpwp.diff}
\begin{split}
 d_2\left(\wp(x)\wp(y) \cdot a \otimes b \otimes m\right)
= {} & \wp(x) \left(  a \otimes b_{(\wp)}m - b \otimes a_{(\wp)}m \right)
- \wp(x) \int \left\{ a_{(\wp')}b \right\} \otimes m \\
&\hspace{-0cm}- \wp(x) \left(  a_{(\wp)}b \otimes m - b_{(\wp)}a \otimes m \right) \\
&\hspace{-0cm}-(2 \pi i)^{2} q \frac{d}{dq} a_{(\wp)}b \otimes m.
\end{split}
\end{align}
\end{lem}

We remark that equation (\ref{wp1.diff}) is redundant, in view of the relation (\ref{eq:relation.in.A2}). But it is convenient to record it for later use.

\begin{proof}
Let $f(t)$ be an elliptic function meromorphic with possible pole at $t=0$. Then by (\ref{2.1.coeff}) the differentials of $f(x) \cdot a \otimes b \otimes m$, $f(y) \cdot a \otimes b \otimes m$ and $f(x-y) \cdot a \otimes b \otimes m$ are, respectively,
\begin{align}\label{indecomp.diffs}
\begin{split}
&f(x) a \otimes b(0)m
- \sum_{j \in \Z_{\geq 0}} \partial^{(j)}f(x) a(j)b \otimes m
- b \otimes a_{(f)}m, \\
&a \otimes b_{(f)}m
- f(x) a(0)b \otimes m
- f(x) b \otimes a(0)m, \\
\text{and} \quad &\sum_{j \in \Z_{\geq 0}} (-1)^j \partial^{(j)}f(x) a \otimes b(j)m
- a_{(f)}b \otimes m
- \sum_{j \in \Z_{\geq 0}} \partial^{(j)}f(-x) b \otimes a(j)m.
\end{split}
\end{align}
From (\ref{indecomp.diffs}) we get (\ref{1.diff}) immediately. We also recover (\ref{wp1.diff}) easily using the relation $\partial_x + T_1 = 0$ and the fact that $\wp(t)$ is an even function of $t$. To prove (\ref{wp2.diff}) we recall, from (\ref{eq:skew-kills}) that
\[
d\left(\wp(x) \cdot a \otimes b \otimes m\right) = - d\left(\wp(y) \cdot b \otimes a \otimes m\right).
\]
The result follows directly from (\ref{indecomp.diffs}) as before.

By (\ref{indecomp.diffs}) the differential of $(\zeta(x) - \zeta(y) - \zeta(x-y)) \cdot a \otimes b \otimes m$ is
\begin{align*}
&a_{(\zeta)}b \otimes m - b \otimes a_{(\zeta)}m - a \otimes b_{(\zeta)}m \\
&-\sum_{j \geq 1} \partial^{(j)}\zeta(x) a(j)b \otimes m
-\sum_{j \geq 1} (-1)^j \partial^{(j)}\zeta(x) a \otimes b(j)m
-\sum_{j \geq 1} (-1)^j \partial^{(j)}\zeta(x) b \otimes a(j)m.
\end{align*}
Here we have used that $\zeta(t)$ is an odd function of $t$. Using Lemma \ref{lem:skew-symmetry} the expression above is reduced to (\ref{zeta.diff}).

We now compute the differential of $\ZZZ(x, y, 0) a \otimes b \otimes m$ using Lemma \ref{surprising.symmetry} to make the calculations more comfortable. The first term $\res_{y=0}f(x, y) a \otimes b(y)m$ of (\ref{2.1.coeff}) becomes
\begin{equation}\label{eq:5.1.1}
a \otimes b_{(\wp \zeta)}m + \frac{1}{2} a \otimes b_{(\wp')} m + \sum_{j \geq 1} (-1)^j \partial^{(j)}\zeta(x) a \otimes b_{(x^j \wp)} m 
\end{equation}
The second term $\res_{y=x}f(y, x) a(y-x)b \otimes m$ becomes
\begin{equation}\label{eq:5.1.3}
a_{(\wp \zeta)} b \otimes m + \frac{1}{2} a_{(\wp')}b \otimes m - \sum_{j \geq 1} \partial^{(j)}\zeta(x) a_{(x^{j} \wp)}b \otimes m 
\end{equation}
The third term $\res_{y=0}f(y, x) b \otimes a(y)m$ becomes 
\begin{align}\label{eq:5.1.3b}
-b \otimes a_{(\wp \zeta)}m - \frac{1}{2} b \otimes a_{(\wp')} m + \sum_{j \geq 1} \partial^{(j)}\zeta(-x) b \otimes a_{(x^j \wp)} m. 
\end{align}
By Lemma \ref{lem:deq}, equation (\ref{eq:5.1.3}) becomes
\begin{equation}\label{eq:5.1.3deq}
-(2\pi i)^2 q \frac{d}{dq} a_{(\zeta)}b \otimes m + \wp(x) \int \left\{ a_{(\wp)}b \right\} \otimes m.
\end{equation}
Similar reductions are performed on (\ref{eq:5.1.1}) and (\ref{eq:5.1.3b}). Collecting, we obtain (\ref{ZZZ.diff}) as (\ref{eq:5.1.1}) $-$ (\ref{eq:5.1.3}) $-$ (\ref{eq:5.1.3b}).

Finally we come to the image of $\wp(x) \wp(y) a \otimes b \otimes m$. The first and third terms of \eqref{2.1.coeff} together yield
\begin{align}\label{eq:wp.wp.1.3.contrib}
\wp(x) \left( a \otimes b_{(\wp)}m - b \otimes a_{(\wp)}m \right)
\end{align}
immediately. Computation of the remaining term $\res_{y=x} \wp(x) \wp(y) a(y-x)b \otimes m$ is more involved. To facilitate the calculation we exploit Lemma \ref{surprising.symmetry} once again, as we did in proving (\ref{ZZZ.diff}). Indeed
\begin{align}
\wp(x)\wp(y)
= {} & - \partial_x \left(\wp(y) \zeta(x) \right) \nonumber \\
= {} & \partial_x \left( \wp(y) \left[ \zeta(x-y) - \zeta(x) + \zeta(y) \right] + \frac{1}{2} \wp'(y) - \frac{1}{2} \wp'(y) - \wp(y) \zeta(y) - \wp(y) \zeta(x-y) \right) \nonumber \\
= {} & \partial_x \ZZZ(y, 0, x) + \wp(y)\wp(x-y) \nonumber \\
= {} & \partial_x \ZZZ(x, y, 0) + \wp(y)\wp(x-y) \nonumber \\
= {} & \partial_x \left(  \wp(x-y) \left[ \zeta(x-y) - \zeta(x) + \zeta(y) \right] + \frac{1}{2} \wp'(x-y) \right) + \wp(y)\wp(x-y) \nonumber \\
= {} &\wp'(x-y) \zeta(x-y) - \wp(x-y)^2 + \frac{1}{2} \wp''(x-y) \nonumber \\
&+ \wp'(x-y) \left[ \zeta(y) - \zeta(x) \right] + \wp(x-y) \left[ \wp(x) + \wp(y) \right] \nonumber \\
\begin{split}\label{eq:rho.rho.rewrite}
= {} &\wp'(y-x) \zeta(y-x) - \wp(y-x)^2 + \frac{1}{2} \wp''(y-x) \\
&+ \wp'(y-x) \left[ \zeta(x) - \zeta(y) \right] + \wp(y-x) \left[ \wp(x) + \wp(y) \right].
\end{split}
\end{align}
The expression (\ref{eq:rho.rho.rewrite}) consists of five terms. The first three terms contribute
\begin{align*}
a_{(\wp' \zeta)} b \otimes m - a_{(\wp^2)} b \otimes m + \frac{1}{2} a_{(\wp'')} b \otimes m
\end{align*}
to $\res_{y=x} \wp(x) \wp(y) a(y-x)b \otimes m$. The derivative (in the first entry) of the equality (\ref{eq:deq}) is
\begin{align*}
(2\pi i)^2 q \frac{d\wp}{dq} = \zeta \wp' - \wp^2 + \frac{1}{2}\wp''.
\end{align*}
The contribution of the first three terms, written above, therefore becomes
\begin{align*}
(2\pi i)^2 q \frac{d}{dq} a_{(\wp)}b \otimes m.
\end{align*}
The fourth term of (\ref{eq:rho.rho.rewrite}) contributes
\begin{align*}
\wp(x) \sum_{j \in \Z_{\geq 0}} \frac{1}{j+1} (-T)^{(j)} (a_{(x^{j+1} \wp')} b) \otimes m = \wp(x) \int \left\{ a_{(\wp')}b \right\} \otimes m.
\end{align*}
Using (\ref{eq:skew-kills}) the contribution of the fifth term of (\ref{eq:rho.rho.rewrite}) becomes
\begin{align*}
&\res_{y=x} \wp(y-x) \wp(x) a(y-x)b \otimes m - \res_{y=x} \wp(y-x) \wp(x) b(y-x)a \otimes m \\
&\,\,\,\,\,\,\, = \wp(x) a_{(\wp)}b \otimes m - \wp(x) b_{(\wp)}a \otimes m.
\end{align*}
Collecting terms, we see that $d \left( \wp(x) \wp(y) \cdot a \otimes b \otimes m \right)$ equals (\ref{eq:wp.wp.1.3.contrib}) minus
\begin{align*}
(2 \pi i)^{2} q \frac{d}{dq} a_{(\wp)}b \otimes m + \wp(x) \int \left\{ a_{(\wp')}b \right\} \otimes m + \wp(x) a_{(\wp)}b \otimes m - \wp(x) b_{(\wp)}a \otimes m,
\end{align*}
as required.
\end{proof}

\section{Passage to the $q=0$ limit}\label{sec:limit}
\begin{nolabel}
It turns out that the chiral homologies $H^{\text{ch}}_i(X, \CA)$, for $i=0,1$, of the elliptic curve $X$ simplify at $q=0$, which corresponds geometrically to the limit in which $X$ degenerates to a nodal cubic. We shall see that at $q=0$ the homologies are related to the Zhu algebra of $V$.

First we introduce subgroups $B^{-n} \subset A^{-n}$ for $n=0,1,2$, and then we show that at $q=0$ they constitute a subcomplex. We define
\begin{align}\label{q=0.subcomplex.B0}
B^0 = 1 \cdot V_{(\wp)}V \subset 1 \cdot V = A^0,
\end{align}
next we define $B^{-1}$ to be the image in
\[
A^{-1} \cong 1 \cdot (V/TV) \otimes V + \wp(x) \cdot V \otimes V
\]
of
\begin{align}\label{q=0.subcomplex.B1}
1 \cdot V_{(\wp)}V \otimes V + 1 \cdot V \otimes V_{(\wp)}V + \wp(x) \cdot V \otimes V.
\end{align}
Let us note here that $TV \subset V_{(\wp)}V$ since $a_{(\wp)}\vac = Ta$. Finally we define $B^{-2}$ to be the image in $A^{-2}$ of
\begin{align}\label{q=0.subcomplex.B2}
\begin{split}
\wp(x-y) \cdot V \otimes V \otimes V + & \wp(x) \cdot V \otimes V \otimes V \\
& {} + \ZZZ(x,y,0) \cdot V \otimes V \otimes V + \wp(x) \wp(y) \cdot V \otimes V \otimes V.
\end{split}
\end{align}
\end{nolabel}
\begin{lem}
The subgroups $B^i$ defined above form a subcomplex $B^\bullet \subset A^\bullet(q=0)$.
\end{lem}
\begin{proof}
First we demonstrate $d B^{-1} \subset B^0$. It is obvious that $d(\wp(x) \cdot V \otimes V) \subset B^0$. Now
\begin{align*}
d(1 \cdot a \otimes b_{(\wp)}m)
&= a(0)(b_{(\wp)}m)
= a(0)(\res_z \wp(z)b(z)m ) \\
&= \res_z \wp(z) \left( [a(0), b(z)]m + b(z)a(0)m \right) \\
&= (a(0)b)_{(\wp)}m + b_{(\wp)}(a(0)m) \in B^0
\end{align*}
(using here that $[a(0), b(z)] = [a(0)b](z)$). On the other hand
\begin{align*}
d(1 \cdot a_{(\wp)}b \otimes m) 
&= (a_{(\wp)}b)(0)m 
= \res_x\res_w \wp(x)[a(x)b](w)m \\ 
&= \res_x\res_z\res_w \wp(x) \left( a(z)b(w)mi_{z,w} - b(w)a(z)mi_{z,w} \right) \delta(x, z-w) \\
&= \res_z\res_w \left( a(z)b(w)mi_{z,w} - b(w)a(z)mi_{z,w} \right)  \wp(z-w) \\
&= \res_z\res_w\sum_{j \in \Z_{\geq 0}} \left( a(z)b(w)m (-w)^j \partial^{(j)}\wp(z)
- b(w)a(z)m z^j \partial^{(j)}\wp(-w) \right) \\
&= \sum_{j \in \Z_{\geq 0}} (-1)^j \left( a_{(\partial^{(j)}\wp)}(b(j)m) - b_{(\partial^{(j)}\wp)}(a(j)m) \right).
\end{align*}
Since $(Ta)_{(f)}b = -a_{(\partial f)}b$ we see that $d(1 \cdot a_{(\wp)}b \otimes m) \in B^0$.

We have not yet used the condition $q=0$, we use it to guarantee $dB^{-2} \subset B^{-1}$. More precisely the inclusion follows from the explicit formulas of Lemma \ref{lem:2.1.diff}; the terms of the form $q \frac{d}{dq}(\cdots)$ vanish at $q=0$, and all remaining terms manifestly lie in $B^{-1}$.
\end{proof}

\begin{nolabel}
Now let us put $Q^\bullet = A^\bullet(q=0) / B^\bullet$ and consider the long exact sequence in cohomology associated with $0 \rightarrow B^\bullet \rightarrow A^\bullet(q=0) \rightarrow Q^\bullet \rightarrow 0$, namely
\begin{align}\label{long.exact.sequence}
\begin{split}
\xymatrix{
H^{0}(B^\bullet) \ar@{->}[r] & H^{0}(A^\bullet(q=0)) \ar@{->}[r] & H^{0}(Q^\bullet) \ar@{->}[r] & 0 \\
  H^{-1}(B^\bullet) \ar@{->}[r] & H^{-1}(A^\bullet(q=0)) \ar@{->}[r] & H^{-1}(Q^\bullet) \ar@{->}[ull] \\
 & \cdots \ar@{->}[r] & H^{-2}(Q^\bullet) \ar@{->}[ull] \\  
}
\end{split}
\end{align}
Clearly $d : B^{-1} \rightarrow B^0$ is surjective. Therefore we have an isomorphism
\begin{align*}
H^{0}(A^\bullet(q=0)) \rightarrow H^{0}(Q^\bullet)
\end{align*}
and a surjection
\begin{align*}
H^{-1}(A^\bullet(q=0)) \rightarrow H^{-1}(Q^\bullet)
\end{align*}
whose kernel is a quotient of $H^{-1}(B^\bullet)$. In the next section we compute $H^{-1}(B^\bullet)$ using the spectral sequence associated with a filtration.
\end{nolabel}
\begin{nolabel}
As we saw in equation (\ref{eq:zeta.to.f.reduction}) the specialisations at $q=0$ of $\zeta(z)$ and $\wp(z)$ are the functions $f(z)-\pi i$ and $g(z)$ of (\ref{f.and.g.def}). By the second of these facts we see that as vector spaces
\begin{align*}
Q^0 = A^0(q=0) / B^0 \cong V / V_{(g)}V = \zhu(V).
\end{align*}
We also have
\begin{align*}
Q^{-1} = A^{-1}(q=0) / B^{-1}
\cong \frac{V \otimes V}{V_{(g)}V \otimes V + V \otimes V_{(g)}V}
\cong \zhu(V) \otimes \zhu(V)
\end{align*}
as vector spaces. In the middle isomorphism here we have used that $TV \subset V_{(\wp)}V$. The differential $Q^{-1} \rightarrow Q^0$ is given by
\begin{align*}
a \otimes b \mapsto a(0)b = \frac{1}{2\pi i} \left( ab - ba \right),
\end{align*}
where $ab = a_{(f)}b$ is the Zhu product. Here we have used (\ref{eq:Zhu.commutator}). Up to a sign, then, the differential coincides with that of the bar complex of $\zhu(V)$ according to the convention (\ref{eq:bar.complex.conv}) that we have adopted. It follows that
\begin{align*}
H^0(Q^\bullet) \cong \Hoch_0(\zhu(V)).
\end{align*}
Now we will show that
\begin{align}\label{Q.zhu.isom}
H^{-1}(Q^\bullet) \cong \Hoch_1(\zhu(V)).
\end{align}
We start by noting that $H^{-1}(Q^\bullet)$ is the quotient of $Q^{-1}$ by the images of $1 \cdot V \otimes V \otimes V$ and $\zuta(x, y, 0) \cdot V \otimes V \otimes V$. According to (\ref{1.diff}) the image of $1 \cdot a \otimes b \otimes m$ is
\begin{align}\label{raw.Lie.diff}
a \otimes b(0)m - a(0)b \otimes m - b \otimes a(0)m = \frac{1}{2\pi i} \left( a \otimes [b, m] - [a, b] \otimes m - b \otimes [a, m] \right),
\end{align}
and according to (\ref{zeta.diff}) the image of $-\zuta(x, y, 0) \cdot a \otimes b \otimes m$ is
\begin{align}\label{raw.zhu.hoch}
a_{(\zeta)}b \otimes m - b \otimes a_{(\zeta)}m - a \otimes b_{(\zeta)}m.
\end{align}
By the relation $\zeta(z, q=0) = f(z) - \pi i$ together with (\ref{eq:Zhu.commutator}) again, we see (at $q=0$) that
\begin{align*}
a_{(\zeta)}b = a_{(f - \pi i)}b = ab - \frac{1}{2}(ab - ba) = \frac{1}{2}(ab + ba).
\end{align*}
Therefore the expression (\ref{raw.zhu.hoch}) reduces to one half times
\begin{align}\label{impure.zhu.hoch}
(ab + ba) \otimes m - b \otimes (am + ma) - a \otimes (bm + mb). 
\end{align}
Now we recall the Hochschild differential
\begin{align}\label{pure.zhu.hoch}
d_{\text{Hoch}}(a \otimes b \otimes m) &= b \otimes ma - ab \otimes m + a \otimes bm
\end{align}
and we rewrite (\ref{impure.zhu.hoch}) as
\begin{align*}
&\left( ab \otimes m - b \otimes ma - a \otimes bm \right)
+ \left( ba \otimes m - b \otimes am - a \otimes mb \right) \\
= {} & -d_{\text{Hoch}}(a \otimes b \otimes m) - d_{\text{Hoch}}(b \otimes a \otimes m).
\end{align*}
On the other hand the right hand side of (\ref{raw.Lie.diff}) is (ignoring an overall factor of $2\pi i$)
\begin{align*}
&\left( a \otimes bm - ab \otimes m + b \otimes ma \right)
+ \left( -a \otimes mb + ba \otimes m - b \otimes am \right) \\
= {} & d_{\text{Hoch}}(a \otimes b \otimes m) - d_{\text{Hoch}}(b \otimes a \otimes m).
\end{align*}
It follows that $d_2(B^{-2})$, which is the sum of the images of (\ref{impure.zhu.hoch}) and (\ref{raw.Lie.diff}), coincides with (\ref{pure.zhu.hoch}) the image of the Hochschild differential. Thus we have proved (\ref{Q.zhu.isom}).
\end{nolabel}

\section{Li's Filtrations}\label{sec:filtr}

In the previous section we have seen that there is a surjection
\begin{align*}
H^{-1}(A^\bullet(q=0)) \rightarrow \Hoch_1(\zhu(V))
\end{align*}
whose kernel is controlled by the cohomology group $H^{-1}(B^\bullet)$. In this section we introduce a filtration on $B^\bullet$ and analyse the cohomology through the associated spectral sequence.

\begin{nolabel}
Let $V$ be a vertex algebra. The \emph{Li filtration} \cite{Lifilt} $F^\bullet{V}$ is the decreasing filtration on $V$ defined by putting $F^pV$ to be the span of the vectors
\begin{align*}
a^1{(-n_1-1)}\cdots a^r{(-n_r-1)}b,
\end{align*}
where $r \geq 0$, $a^1, \ldots, a^r, b \in V$, and $n_1, \ldots, n_r \in \Z_{+}$, such that $\sum_j n_j \geq p$.

The product $a \cdot b = a{(-1)}b$ induces a commutative associative algebra structure on the associated graded $\gr^FV$ with $\vac$ as unit. Following \cite{Arakawa.lisse} the \emph{singular support} of $V$ is by definition the scheme
\begin{align*}
\sing(V) = \Spec{(\gr^FV)}.
\end{align*}
The translation operator $T$ turns $\gr^FV$ into a differential algebra, in other words $\sing(V)$ comes equipped with a canonical vector field. Clearly $\gr^F_0{V}$ coincides with Zhu's $C_2$-algebra $R_V$ of Section \ref{sec:c2.def}.
\end{nolabel}

\begin{nolabel}
Now suppose that $V$ is (quasi)conformal. In \cite{Li.PVA} Li introduced another filtration on $V$, this one increasing. Let $\{a^i | i \in I\}$ be a strong generating set of $V$, recall this means that the vectors
\begin{align}\label{eq:VA.PBW.monom}
a^{i_1}{(-n_1)}\cdots a^{i_r}{(-n_r)}\vac,
\end{align}
where $i_1, \ldots, i_r \in I$ and $n_1, \ldots, n_r \in \Z_{> 0}$, span $V$. We assume furthermore that the $a^i$ are elements of homogeneous conformal weight. The \emph{standard filtration} $G_\bullet V$, relative to the choice of strong generating set, is the increasing filtration on $V$ defined by putting $G_p V$ to be the span of the vectors (\ref{eq:VA.PBW.monom}) such that $\sum_j \D(a^{i_j}) \leq p$, where $\D(a)$ denotes the conformal weight of $a$.

Let $a \in G_pV$ and $b \in G_qV$, then for all $n \in \Z$ one has $a(n)b \in G_{p+q}V$. In particular we have, for any Laurent series $f$, that
\begin{align}\label{eq:G.filtr.property.1}
a_{(f)}b \in G_{p+q}V.
\end{align}
Furthermore
\begin{align}\label{eq:G.filtr.property.2}
Ta \in G_pV
\end{align}
and
\begin{align}\label{eq:G.filtr.property.3}
a(n)b \in G_{p+q-1}V \quad \text{for $n \in \Z_{\geq 0}$.}
\end{align}
See \cite[Proposition 4.2, Theorem 4.6]{Li.PVA}.

For a (quasi)conformal vertex algebra both the Li filtration and the standard filtration are compatible with the conformal weight grading. Let $F^pV_\D = V_\D \cap F^pV$ and $G_pV_\D = V_\D \cap G_{p}V$. According to {\cite[Proposition 2.6.1]{Arakawa.lisse}} one has
\begin{align}\label{F.G.same}
F^pV_\D = G_{\D-p}V_\D
\end{align}
for all $\D, p \in \Z$. The equalities (\ref{F.G.same}) induce an isomorphism
\begin{align}\label{F.G.isom}
\gr^GV \rightarrow \gr^FV
\end{align}
of commutative differential algebras (even of Poisson vertex algebras). Although both sides of (\ref{F.G.isom}) are naturally $\Z_{\geq 0}$-graded, the isomorphism does not respect the grading.


The vertex algebra $V$ is said to be finitely strongly generated if it possesses a finite set of strong generators. It is well known \cite{GN03} that if $V$ is finitely strongly generated then $R_V = \gr_0^F V$ is an algebra of finite type.
\end{nolabel}

\begin{nolabel}
From equations (\ref{wp1.diff}) and (\ref{wp2.diff}) we see that elements of $1 \cdot V_{(\wp)}V \otimes V$ and of $1 \cdot V \otimes V_{(\wp)}V$ are equivalent, modulo boundaries of elements of $B^{-2}$, to elements of $\wp(x) \cdot V \otimes V$.  Thus it suffices to analyse the kernel of the restriction of $d : B^{-1} \rightarrow B^0$ to $\wp(x) \cdot V \otimes V$.

The standard filtration on $V$ induces an increasing filtration on $A^\bullet$, namely $G_p A^{-n}$ is defined to be the image in the quotient of the vector subspace generated by
\begin{align*}
\bigcup_{\sum p_i \leq p} \mathring{\G}_n \cdot G_{p_1}V \otimes \cdots \otimes G_{p_{n+1}}V.
\end{align*}
This is well-defined since the filtration $G$ is $T$-stable, i.e., because of (\ref{eq:G.filtr.property.2}). The filtration is compatible with the differential because of (\ref{eq:G.filtr.property.1}). We endow $B^\bullet$ with the induced filtration.

We recall the form (\ref{eq:def.Weierstrass.p}) of the Weierstrass function $\wp(x)$ and the property (\ref{eq:G.filtr.property.3}) of the standard filtration. It follows that the induced differential $d : \gr^GB^{-1} \rightarrow \gr^GB^0$ is given by
\begin{align*}
\wp(x) \cdot a \otimes b \mapsto 1 \cdot a{(-2)}b.
\end{align*}
The image of $\ZZZ(x, y, 0) \cdot a \otimes b \otimes m$, which at $q=0$ is
\begin{align*}
-\wp(x) \cdot \int{\{a_{(\wp)}b\}} \otimes m
+\wp(x) \cdot \sum_{j \in \Z_{\geq 0}} \frac{1}{j+1} \left(T^{(j)}a \otimes b_{(x^{j+1}\wp)}m + T^{(j)}b \otimes a_{(x^{j+1}\wp)}m \right),
\end{align*}
similarly becomes
\begin{align*}
-\wp(x) \cdot \left( a{(-1)}b \otimes m - a \otimes b{(-1)}m - b \otimes a{(-1)}m \right).
\end{align*}
in $\gr^GB^{-1}$. Finally the image of $\wp(x)\wp(y) a \otimes b \otimes m$ at $q=0$ is
\begin{align*}
\wp(x) \cdot \left( a \otimes b_{(\wp)}m - b \otimes a_{(\wp)}m \right) - \wp(x) \cdot \int \left\{ a_{(\wp')}b \right\} \otimes m 
- \wp(x) \cdot \left(  a_{(\wp)}b - b_{(\wp)}a \right) \otimes m.
\end{align*}
In $\gr^GB^{-1}$ the term $\wp(x) \cdot \left( a \otimes b_{(\wp)}m - b \otimes a_{(\wp)}m \right)$ becomes (we throw away all terms of the form $a(j)b$ for $j \geq 0$)
\begin{align*}
\wp(x) \cdot \left( a \otimes b(-2)m - b \otimes a(-2)m \right).
\end{align*}
Similarly the integral term becomes 
\begin{align*}
\wp(x) \cdot \left( a_{(x\wp')}b - \frac{1}{2}T(a_{(x^2\wp')}b) \right) \otimes m = \wp(x) \cdot \left( -2 a(-2)b + T(a(-1)b) \right) \otimes m.
\end{align*}
The remaining term becomes
\begin{align*}
\wp(x) \cdot \left( a(-2)b - b(-2)a \right) \otimes m 
= \wp(x) \cdot \left( 2 a(-2)b - T(a(-1)b)\right) \otimes m,
\end{align*}
where we have used skew-symmetry to rewrite $b(-2)a$ as $-a(-2)b + T(a(-1)b)$ plus terms that vanish in the associated graded. Clearly the integral term and subsequent term cancel, leaving us with
\begin{align*}
\wp(x) \cdot \left( a \otimes b(-2)m - b \otimes a(-2)m \right).
\end{align*}


Thus $H^{-1}(\gr^G B^\bullet)$ is the kernel of the map $\gr^GV \otimes \gr^GV \rightarrow \gr^GV$ defined by
\begin{align} \label{aTb}
a \otimes b \mapsto (Ta) \cdot b,
\end{align}
modulo the subspace generated by the elements
\begin{align}
&a \cdot b \otimes m - a \otimes b \cdot m - b \otimes a \cdot m, \label{element.to.kill.1} \\
\text{and} \quad &a \otimes (Tb) \cdot m - b \otimes (Ta) \cdot m. \label{element.to.kill.2}
\end{align}
In Section \ref{sec:conclusion} below we shall interpret these formulas as differentials in a Koszul complex.
\end{nolabel}

\begin{nolabel}
Associated with the increasing filtration $G_pB^\bullet$ is the spectral sequence in which the first page is $E^{p, q}_1 = H^{p+q}(\gr^{-p}B^{\bullet})$ and the differential of bidegree $(+1, 0)$ is induced by the connecting morphisms of
\begin{align*}
0 \rightarrow \frac{G_{p-1} V}{G_{p-2} V} \rightarrow \frac{G_p V}{G_{p-2} V} \rightarrow \frac{G_p V}{G_{p-1} V} \rightarrow 0.
\end{align*}
Since the filtration is exhaustive and bounded below, we have convergence
\begin{align}\label{ss.converge}
E^{p, q}_1 \Rightarrow H^{p+q}(B^\bullet).
\end{align}
\end{nolabel}

\section{K\"{a}hler Differentials and Arc Spaces}\label{sec:hoch}

\begin{nolabel}
Let $A$ be a commutative $\C$-algebra. The space $\Om_{A/\C}$ of K\"{a}hler differentials is the free $A$-module generated by symbols $df$ for $f \in A$, modulo the relations $d\alpha = 0$ for $\alpha \in \C$, $d(f+g) = df+dg$ and $d(fg) = f \cdot dg + g \cdot df$. It is well known that
\begin{align*}
\Hoch_0(A) = A \quad \text{and} \quad \Hoch_1(A) \cong \Om_{A/\C},
\end{align*}
where the latter isomorphism is induced by $a \otimes b \mapsto b \cdot da$ (the swap of factors is a consequence of our conventions for Hochschild homology). There is an isomorphism
\begin{align}\label{der.isom}
\Hom_A(\Om_{A/\C}, A) \cong \Der_\C(A),
\end{align}
associating to $\tau \in \Der_\C(A)$ the homomorphism $\iota_\tau : \Om_{A/\C} \rightarrow A$ defined by $\iota_\tau(f \cdot dg) = f \tau(g)$.

Let us fix $\tau \in \Der_\C(A)$ and consider the exterior algebra
\begin{align}\label{K.complex}
K^A_\bullet = \Sym_A(\Om_{A/\C}[1])
\end{align}
of the $A$-module $\Om_{A/\C}$. The morphism $\iota_\tau : \Om_{A/\C} \rightarrow A$ of $A$-modules uniquely extends, much as in Section \ref{sec:lie.homology}, to a differential on $K^A_\bullet$ which we also denote $\iota_\tau$.
\end{nolabel}

\begin{nolabel}
The canonical morphism $A \rightarrow \Om_{A/\C}$, sending $a$ to $da$, extends uniquely to a derivation of $K^A_\bullet$, known as the de Rham differential and denoted $d$. The Lie derivative $\Lie_\tau = [d, \iota_\tau]$ (given explicitly in formula (\ref{eq:Lie.derivative}) above) enhances $K^A_\bullet$ to a complex of differential $A$-modules.
\label{no:differential}
\end{nolabel}

\begin{nolabel}
Now let $A = \bigoplus_{n \in \Z_{\geq 0}} A^n$ be a $\mathbb{Z}_{\geq 0}$-graded commutative algebra, with a derivation $\tau$ of degree $+1$. The $A$-module $\Omega_{A/\C}$ acquires a natural $\Z_{\geq 0}$-grading by declaring $\deg (a \cdot db) = \deg(a) + \deg(b) + 1$. This grading in turn extends to a grading $K_\bullet = K^A_\bullet = \bigoplus_{n \in \Z_{\geq 0}} K_\bullet^n$ compatible with the differential $\iota_\tau$, and the homology $H_\bullet(K_\bullet, \iota_\tau) = \bigoplus_{n \in \Z_{\geq 0}} H_\bullet(K_\bullet^n, \iota_\tau)$ becomes a $\mathbb{Z}_{\geq 0}$-graded $A$-module. With respect to the internal $\mathbb{Z}_{\geq 0}$ gradation the operators $\iota_\tau$, $d$ and $\Lie_\tau$ are homogeneous of degree $0$, $+1$ and $+1$ respectively.
\label{no:jet-1}
\end{nolabel}

\begin{nolabel}\label{no:ideal}
Before proceeding we recall the notion of arc space. We shall need to consider only arc spaces of affine schemes in this work. For background theory we refer the reader to \cite{EinMustata}.

Let $X = \Spec A^0$ be an affine scheme of finite type. The arc space of $X$ consists of a scheme $JX = \Spec JA^0$ equipped with a vector field $\partial \in \Der(JA^0)$ and a morphism $JX \rightarrow X$ characterised by the following universal property, which we formulate in terms of coordinate rings: For any algebra morphism $A^0 \rightarrow A$ to a commutative algebra $A$ endowed with a derivation $\tau$, there exists a unique morphism $(JA^0, \partial) \rightarrow (A, \tau)$ of differential algebras such that the diagram
\begin{align} 
\begin{split}
\xymatrix{
 & JA^0 \ar[d] \\ 
A^0 \ar[ur] \ar[r] & A \\
}
\end{split}\label{eq:universal-jet}
\end{align}
commutes. In the present case the arc space can be constructed quite explicitly: If $A^0 \cong \C[x^1, \ldots, x^n]/(f_1, \ldots, f_s)$ then $JA^0$ is the quotient of the differential algebra $\C[x^i_j]_{i=1,\ldots,n, \,\, j \in \Z_{\geq 0}}$, with differential $\partial$ defined by $\partial(x^i_j) = x^i_{j+1}$, by the differential ideal generated by $x^1, \ldots, x^n$. The algebra $JA^0$ is naturally $\Z_{\geq 0}$-graded by setting $\deg x^i_j = j$ so that the differential $\partial$ has degree $+1$. Clearly $(JA^0)^0 = A^0$.
\end{nolabel}

\begin{nolabel}\label{no:euler}
We return to our discussion of the complex $K_\bullet$ associated with $A = \bigoplus_{n \in \Z_{\geq 0}} A^n$. Let us assume from now on that each component $A^n$ of $A$ is finitely generated as an $A^0$-module. Then each subcomplex $K^n_\bullet$ is a complex of finitely generated $A^0$-modules, concentrated in cohomological degrees $-n, \ldots, 0$. We also assume from now on that $(A, \tau)$ is generated as a differential algebra by $A^0$. Therefore we have a morphism $JA^0 \rightarrow A$ of graded commutative unital differential algebras which, by our assumption on $A$, is surjective. It also follows easily that $(K, \Lie_\tau)$ is generated as a differential algebra by its subalgebra $\widetilde{A} \cong \Sym_{A^0}(\Omega_{A^0/k}[1])$ of total degree $0$, so we have a surjection
\begin{align*}
(J\widetilde{A}, \partial) \twoheadrightarrow (K, \Lie_\tau)
\end{align*}
of differential algebras. Writing $A^+ = \bigoplus_{j > 0} A^j$, it is clear that
\begin{align}\label{homo.in.deg.0}
H_0(K_\bullet, \iota_\tau) = A / A^+ \cong A^0.
\end{align}
\end{nolabel}

\begin{nolabel}
Geometrically we may think of the superscheme $\widetilde{X} = \Spec \widetilde{A}$ as the shifted tangent bundle $T[-1]X$ of $X$, and its arc space $J\widetilde{X}$ is the shifted tangent bundle $JT[-1]X = \Spec J\widetilde{A} =  T[-1]JX$ of the arc space of $X$. Let us write $Y = \Spec{A}$, then by assumption we have an embedding $Y \hookrightarrow JX$, and therefore
\begin{align*}
T[-1]Y = \Spec K \hookrightarrow T[-1]JX = JT[-1]X.
\end{align*}
\end{nolabel}

\begin{thm}\label{thm:ssjets}
Let $A = \bigoplus_{n \in \Z_{\geq 0}} A^n$ be a $\mathbb{Z}_{\geq 0}$-graded commutative algebra with a derivation $\tau$ of degree $+1$, and let $(K^A_\bullet, \iota_\tau)$ be the Koszul complex associated with $A$ as above. We assume $A$ is generated by $A^0$ as a differential algebra, and that $A^0$ is an algebra of finite type. Then $H_{-1}(K_\bullet^A, \iota_\tau) = 0$  if and only if the canonical morphism $JA^0 \rightarrow A$ is an isomorphism.
\label{thm:whole-arc-thm}
\end{thm}
\begin{proof}
We first prove the implication ($\Leftarrow$). We begin by considering the case $A^0 = k[x^1,\ldots,x^n]$ and $A = JA^0 = k[x^i_j]_{i=1,\ldots,n, \,\, j \in \Z_{\geq 0}}$. We choose a total order for the generators in which $x^{i_1}_{j_1} < x^{i_2}_{j_2}$ whenever $j_1 < j_2$ and consider the corresponding regular sequence $\{x^i_j\}_{j \geq 1}$ in $A$. The complex $K^A_\bullet$ is simply the Koszul complex associated with this regular sequence, under the identification of $dx^i_j$ with $[x^i_{j-1}]$. Hence $H_n(K^A_\bullet, \iota_\partial) = 0$ for $n \neq 0$ by the Koszul vanishing theorem. Although our regular sequence is infinite, imposing an upper limit $j \leq j_0$ yields a finite regular sequence for the subalgebra of $A$ generated by $\{x^i_j\}_{j \leq j_0}$. The resulting subcomplex coincides with $K^A_\bullet$ in internal degree less than $j_0$, and so vanishing of homology of $K^A_\bullet$ follows from the Koszul vanishing theorem for finite regular sequences.

Now let $A^0$ be arbitrary of finite type, let $A = JA^0$, and suppose $H_{-1}(K^A_\bullet, \iota_\partial) = 0$. We now put $B^0 = A^0/(f)$ where $(f)$ is the ideal generated by some nonzero element $f$. We shall prove that $H_{-1}(K_\bullet^B, \iota_\partial)=0$. Clearly the statement of the theorem follows since any algebra of finite type is a quotient of a polynomial algebra, and polynomial algebras are Noetherian.

We have $B = JB^0 \cong A / (\partial^j{f})_{j \in \Z_{\geq 0}}$. Now let $\bar\omega \in (K^B)_{-1}^j$ be a nonzero cycle of degree $j$ (since it is nonzero we have $j > 0$), which we write in the form
\begin{align*}
\bar\omega = \sum \bar{a}_i \cdot d\bar{b}_i.
\end{align*}
We choose a representative $\omega = \sum a_i \cdot d b_i$ of $\bar\omega$ in $(K^A)_{-1}^j$ and obtain
\begin{align*}
\iota_\partial \omega = \sum a_i \partial b_i = \sum_{k = 0}^j c_k \partial^k f,
\end{align*}
for some collection of elements $c_k \in A^{j-k}$. The form $\omega - \sum_{k=1}^j c_k \cdot d(\partial^{k-1} f)$ is also a representative of $\bar{\omega}$ in $(K^A)_{-1}^j$ now with the property that
\begin{align*}
\iota_\partial \left( \omega - \sum_{k=1}^j c_k \cdot d(\partial^{k-1} f) \right) = c_0 f.
\end{align*}
Since $\deg(c_0) = j>0$, we may write 
\begin{align*}
c_0 = \sum d_i \partial e_i
\end{align*}
for some collection of $d_i, e_i \in A$ such that $\deg(d_i) + \deg(e_i) = j-1$. It follows that 
\begin{align*} 
\iota_\partial \left(  \omega - \sum_{k =1}^j c_k \cdot d(\partial^{k-1}f) - \sum_i d_i f \cdot de_i \right) = 0.
\end{align*}
Since by assumption $H_{-1}(K^A_\bullet, \iota_\partial) = 0$ the form inside the parentheses is exact, hence there exists $\omega^2 \in (K^A)_{-2}^j$ such that 
\begin{align*}
\iota_\partial \omega^2 = \omega - \sum_{k=1}^j c_k \cdot d(\partial^{k-1} f) - \sum_i d_i f \cdot de_i,
\end{align*}
and consequently the image $\bar{\omega}^2 \in (K^B)_{-1}^j$ of $\omega^2$ satisfies
\begin{align*}
\iota_\partial \bar{\omega}^2 = \bar{\omega},
\end{align*}
proving that $H_{-1}(K^B_\bullet, \iota_\partial) = 0$ as required.

Now we prove the implication ($\Rightarrow$). By the universal property of $JA^0$ we have $A = JA^0 / I$ where $I$ is a homogeneous differential ideal. Let $K_\bullet = K_\bullet^A$ and $\widetilde{K}_\bullet = K_\bullet^{JA^0}$, so that we have a surjective morphism of complexes $\widetilde{K}_\bullet \twoheadrightarrow K_\bullet$. Let $f \in I$ be non-zero and homogeneous of minimal degree. Since $(JA^0)^0 = A^0$ it follows that $I^0 = 0$ and so $j = \deg(f) > 0$.

It follows from the minimal degree condition on $f$ that it cannot be expressed as a linear combination of terms of the form $a \partial b$ where either $a \in I$ or $b \in I$ since $\deg(a \partial b)$ is strictly greater than both $\deg(a)$ and $\deg(b)$.

Since $H_0(\widetilde{K}_\bullet^j, \iota_\partial) = 0$ (see (\ref{homo.in.deg.0})) there exists $\omega \in \widetilde{K}_{-1}^j$ such that $\iota_\partial \omega = f$. Let $\bar{\omega}$ be the projection of $\omega$ to $K_{-1}^j$, then we have $\iota_\tau \bar\omega = 0$. By assumption $H_{-1}(K_\bullet, \iota_\tau) = 0$, so $\bar{\omega}$ is exact, and so we choose $\bar{\omega}^2 \in K_{-2}^j$ such that $\iota_\tau \bar{\omega}^2 = \bar{\omega}$. Now we let $\omega^2$ be a preimage of $\bar\omega^2$ in $\widetilde{K}_{-2}^{-j}$. It follows that
\begin{align*}
\iota_\partial \omega^2 = \omega + \sum a_i \cdot db_i
\end{align*}
where, for each $i$, either $a_i \in I$ or $b_i \in I$. On the other hand $\deg(a_i) + \deg(b_i) = j-1$, and applying $\iota_\partial$ once more we obtain
\begin{align*}
f = - \sum a_i \partial b_i,
\end{align*}
which contradicts our hypothesis on $f$. Therefore $I = 0$ and so $A = JA^0$.
\end{proof}

\section{Main Theorem}\label{sec:conclusion}


\begin{nolabel}
We now return to the setting of Section \ref{sec:filtr}, so $V$ is a quasiconformal vertex algebra and $A = \gr^F V$. By (\ref{F.G.isom}) we have $A \cong \gr^G V$. We observe that
\begin{align*}
H^{-1}(\gr^G B^\bullet) \cong H_{-1}(K^A_\bullet, \iota_T).
\end{align*}
Indeed we have computed $H^{-1}(\gr^G B^\bullet)$ as the kernel in $\gr^G V \otimes \gr^G V$ of the map (\ref{aTb}) modulo the subspace generated by terms (\ref{element.to.kill.1}) and (\ref{element.to.kill.2}). But the quotient of $\gr^G V \otimes \gr^G V \cong A \otimes A$ by (\ref{element.to.kill.1}) is $\Hoch_1(A) \cong \Om_{A/\C}$, the map (\ref{aTb}) is just $\iota_T$, and (\ref{element.to.kill.2}) is the image of $\iota_T$ in $\Om_{A/\C}$. We are now in a position to state the main theorem.
\end{nolabel}

\begin{thm}\label{thm:main-theorem}
Let $V$ be a vertex algebra and $A^\bullet(q)$ the complex (\ref{eq:A.comp.def}). We assume $V$ is quasiconformal and finitely strongly generated. Let $A = \bigoplus_{n \in \Z_{\geq 0}} A^n$ denote the associated graded $\gr^F{V}$ of $V$ with respect to the Li filtration. For $q \neq 0$ we have
\begin{align*}
H^{-1}(A^\bullet(q)) \cong H_1^{\text{ch}}(X_{q}, \CA_V),
\end{align*}
and if the canonical morphism $JA^0 \rightarrow A$ is an isomorphism then
\begin{align*}
H^{-1}(A^\bullet(q=0)) \cong \Hoch_1(\zhu(V)).
\end{align*}
\end{thm}

\begin{proof}
The first claim is Proposition \ref{prop:chhom.equals.A}. To prove the second part we let $K_\bullet^A$ be the complex (\ref{K.complex}) with differential $\iota_T$. It is well known \cite[Section 4]{Lifilt} that the canonical morphism $JA^0 \rightarrow A$ is surjective, so since we have assumed $V$ to be finitely strongly generated we may apply Theorem \ref{thm:ssjets} and the condition that $JA^0 \rightarrow A$ be an isomorphism guarantees (indeed is equivalent to)
\begin{align*}
H_{-1}(K_\bullet^A, \iota_T) = 0.
\end{align*}
Thus we have  $H^{-1}(\gr^G B^\bullet) = 0$. Now by convergence of the spectral sequence (\ref{ss.converge}) we infer $H^{-1}(B^\bullet) = 0$, and in turn by the long exact sequence (\ref{long.exact.sequence}) we have
\begin{align*}
H^{-1}(A^\bullet(q=0)) \cong H^{-1}(Q^\bullet).
\end{align*}
Combining this with the isomorphism (\ref{Q.zhu.isom}) yields the result.
\end{proof}

\begin{rem}
In Section \ref{sec:examples} below we show that the condition appearing in Theorem \ref{thm:main-theorem}, namely that $JA^0 \rightarrow A$ be an isomorphism, is satisfied for many well-known vertex algebras. It might be possible to strengthen Theorem \ref{thm:main-theorem} by removing or weakening this condition by way of an analysis not only of the $E^1$ page of the spectral sequence (\ref{ss.converge}) but also of its $E^2$ page. The Poisson vertex algebra structure of $\gr^F{V}$ is invisible on the $E^1$ page but in principle appears in the structure of the $E^2$ page. It would be very interesting to see if the theory of Poisson homology \cite{Brylinski.Poisson} or Poisson vertex algebra homology \cite{DK.variational} plays some role. In \cite{arakawasetsu} the Poisson homology of $A^0$ is indeed seen to be related to the representation theory of $V$.
\end{rem}
\begin{rem}
With the notation as in  Theorem \ref{thm:main-theorem}, to any quasi-conformal vertex algebra we have associated a Poisson algebra $A^0=A/(A \cdot T A)$. Let $I$ be the kernel of the surjection $JA^0 \rightarrow A$. Then it is straightforward to see that $I/(I \cdot T I)$ is a non-unital Poisson algebra. Moreover, it is a Poisson module over $A^0$. Namely, if we let $J = A^0 \oplus I/(I \cdot  TI)$, $X = \Spec A^0$, $Y = \Spec J$, we have two Poisson affine schemes $X$ and $Y$ canonically associated to any quasiconformal vertex algebra, together with a Poisson map $Y \twoheadrightarrow X$ and a section $X \hookrightarrow Y$. The condition in the theorem is that these two Poisson schemes be equal. As in the previous remark, it would be interesting to see if the second page of the spectral sequence can be described in terms of the Poisson homologies of these Poisson algebras. 
\label{no:finite-poisson}
\end{rem}

\section{Examples}\label{sec:examples}

In this section we discuss Theorem \ref{thm:main-theorem} in the context of two well-known classes of vertex algebras: the Virasoro vertex algebra and the affine vertex algebras. Background material can be found in the books \cite{Kac.VA.Book} and \cite{FBZ.Book}.

\begin{nolabel}
Fix a constant $c \in \C$. We consider the highest weight $\vir$-module
\[
\vir^c = U(\vir) \otimes_{U(\vir_+)} \C v,
\]
here $\vir_+$ denotes the span in $\vir$ of $C$ and $L_n$ for $n \geq -1$, and the action of $\vir_+$ on $\C v$ is by $L_{n}v = 0$ for $n \geq -1$ and $Cv = cv$. This module carries the natural structure of a conformal vertex algebra of central charge $c$ in which $\vac = 1 \otimes v$ and the quantum field associated with $L_{-2}\vac$ is $L(z) = \sum L_n z^{-n-2}$. The conformal weight grading on $\vir^c$ is determined by $\D(\vac)=0$ and $\D(L_{m}b) = m + \D(b)$. It follows from (\ref{F.G.same}) that $F^p\vir^c_\D$ is the span of the monomials $L_{-n_1-2} \cdots L_{-n_s-2}\vac$ for which $\sum n_i \geq p$ and $\sum(n_i+2)=\D$. Hence as differential commutative algebras
\[
\gr^F{\vir^c} = \C[L_{-2}, L_{-3}, \ldots],
\]
where $\deg(L_{-n-2}) = n$ and $TL_{-n} = nL_{-n-1}$. In particular $R_{\vir^c} \cong \C[x]$ where $x = [L_{-2}]$, and the natural surjection $JR_{\vir^c} \rightarrow \gr^F{\vir^c}$ is an isomorphism.
\end{nolabel}

\begin{nolabel}
Fix a constant $k \in \C$. Let $\g$ be a finite dimensional Lie algebra over $\C$ with invariant bilinear form $(\cdot, \cdot)$, and let $\widehat{\g} = \g((t))\oplus \C K$ be the associated affine Lie algebra (the affine Kac-Moody algebra in case $\g$ is simple, the Heisenberg Lie algebra in case $\g$ is abelian)
\[
[at^m, bt^n] = [a, b]t^{m+n} + m (a, b) \delta_{m, -n} K.
\]
We consider the vacuum module
\[
V^k(\g) = U(\widehat{\g}) \otimes_{U(\g[[t]]+ \C K)} \C v,
\]
here $\g[[t]] v = 0$ and $K v = k v$. This module carries the natural structure of a quasiconformal vertex algebra. If $k \neq -h^\vee$, where $h^\vee$ is the dual Coxeter number of $\g$, then $V^k(\g)$ is furthermore conformal of central charge $c = \frac{k \dim(\g)}{k + h^\vee}$. As above $\vac = 1 \otimes v$, and now the quantum field associated with $a_{-1}\vac$ is $a(z) = \sum a_n z^{-n-1}$, where we have written $a_n = at^n$. The conformal weight grading on $V^k(\g)$ is determined by $\D(\vac)=0$ and $\D(a_{m}b) = m + \D(b)$. As above we see that $F^pV^k(\g)_\D$ is the span of the monomials $a^{1}_{-n_1-1} \cdots a^{s}_{-n_s-1}\vac$ for which $\sum n_i \geq p$ and $\sum(n_i+1)=\D$. Hence as differential commutative algebras
\[
\gr^F{V^k(\g)} = S(t^{-1}\g[t^{-1}]), 
\]
where $\deg(at^n) = -n-1$ and $T(at^{-n}) = nat^{-n-1}$. In particular $R_{V^k(\g)} \cong S(\g)$, and the natural surjection $JR_{V^k(\g)} \rightarrow \gr^F{V^k(\g)}$ is an isomorphism.
\end{nolabel}

\begin{nolabel}
Let $\g$ be a simple Lie algebra and $f \in \g$ a nonzero nilpotent element. We denote by $\CS$ the associated Slodowy slice, defined by embedding $f$ into an $\mathfrak{sl}_2$ triple $\{e, h, f\} \subset \g$ and putting $\CS = \{f + x | [x, e] = 0\} \subset \g$. The universal affine $W$-algebra is constructed as the quantised Drinfeld-Sokolov reduction of $V^k(\g)$. See \cite{KRW} for the construction. It was proved in {\cite{DK06}} that $R_{W^k(\g, f)} \cong \C[\CS]$, and it was proved in {\cite[Theorem 4.17]{A.assoc.var}} that $\gr^F{W^k(\g, f)} \cong \C[J \CS]$. Thus the surjection $J R_{W^k(\g, f)} \rightarrow \gr^F{W^k(\g, f)}$ is an isomorphism (at arbitrary level $k$: critical or non-critical).
\end{nolabel}

\begin{nolabel}
Now we consider the Virasoro minimal models. Let $p, p' \geq 2$ be two coprime integers, and let
\[
c = c_{p, p'} = 1 - 6 \frac{(p-p')^2}{pp'}.
\]
It is well known that the simple quotient $\vir_{p, p'}$ of $\vir^{c}$ is a rational vertex algebra \cite{Wang.Vir} (see also \cite{DMZ.vir}) known as a minimal model. Its (unnormalised) character is given by the formula \cite{FF.coinvar} \cite{KW88}
\begin{align}\label{bosonic.vir}
\chi_{\vir_{p, p'}}(q) = \frac{1}{\prod_{m=1}^\infty(1-q^m)} \sum_{n \in \Z} \left[ q^{\frac{(2pp'n + p-p')^2 - (p-p')^2}{4pp'}} - q^{\frac{(2pp'n + p+p')^2 - (p-p')^2}{4pp'}} \right].
\end{align}
The maximal ideal $I_{p, p'} \subset \vir^c$ is generated by a singular vector $v_{p, p'}$ of conformal weight $(p-1)(p'-1)$. In general $v_{p,p'}$ is a linear combination of monomials $L_{-n_1}^{i_1} \cdots L_{-n_s}^{i_s}\vac$ of total degree $(p-1)(p'-1)$, and the only one of these to survive in the quotient $R_{\vir^c}$ is $L_{-2}^{(p-1)(p'-1)/2}\vac$. The coefficient of this latter monomial is nonzero {\cite{FF1}} {\cite[Lemma 4.3]{Wang.Vir}}. The algebra $R_{\vir_{p, p'}}$ is obtained as the quotient of $R_{\vir^c} \cong \C[x]$ by the ideal generated by the image of $v_{p,p'}$, therefore $R_{\vir_{p, p'}} \cong \C[x] / (x^{(p-1)(p'-1)/2})$.
\end{nolabel}

\begin{nolabel}\label{subsec:ordering}
Before proceeding we recall some standard material on monomial orders, Hilbert series and Gr\"{o}bner bases \cite{CLOS}, particularly in connection with arc spaces. Let $R$ be the graded $\C$-algebra of polynomials on a (finite or infinite) countable totally ordered set of variables
\begin{align}\label{y.tot.ord}
y_0 > y_1 > y_2 > \ldots
\end{align}
whose degrees are compatible with the order in the sense that $\deg(y_{i+1}) \geq \deg(y_i)$ for all $i$. The graded reverse lexicographic (grevlex) order on the set of monomials of $R$ is defined as follows. We put
\[
\prod y_i^{\al_i} > \prod y_i^{\beta_i}
\]
if $\sum \al_i \deg(y_i) > \sum \beta_i \deg(y_i)$, or else if these two sums are equal and the product $\prod y_i^{\al_i}$ comes \emph{later} than $\prod y_i^{\beta_i}$ in the lexicographic ordering. For example if $\deg(y_i) = 1$ for $i=0,1,2$ then $y_1^2 > y_0 y_2$. We write $\LT(f)$ for the leading term of the polynomial $f$, i.e., $\LT(f) = \LC(f) \cdot \LM(f)$ where $\LM(f)$ is the highest monomial occuring in $f$ relative to the grevlex order, and $\LC(f)$ is the coefficient with which it occurs in $f$. The leading term ideal $\LT(I)$ of $I$ is the ideal of $R$ generated by $\LT(f)$ as $f$ runs over $I$.

From now on we assume $\deg(x_0) > 0$ and that $I \subset R$ is a homogeneous ideal. Hence $(R/I)_n$ is well defined and finite dimensional. The Hilbert series of $I$ is
\[
H_I(q) = \sum_{n=0}^\infty \dim(R/I)_n q^n.
\]
Clearly if $I_1 \subset I_2$ then we have $\dim(R/I_2)_n \leq \dim(R/I_1)_n$ for all $n$. Generally for two power series $f(q)$ and $g(q)$ we shall write $f(q) \leq g(q)$ if $f_n \leq g_n$ for all $n \in \Z_{\geq 0}$. Thus $H_{I_2}(q) \leq H_{I_1}(q)$.
In general we write $\chi_V(q)$ for the graded dimension $\sum_{n \in \Z_{\geq 0}} \dim{V_n} q^n$ of a $\Z_{\geq 0}$-graded vector space, so $H_I(q) = \chi_{R/I}(q)$.
\end{nolabel}
\begin{prop}[{\cite[p. 472]{CLOS}}]\label{equal.H}
The Hilbert series of $I$ and ${\LT(I)}$ coincide.
\end{prop}

\begin{nolabel}
Let $f \in R$ and $g_1, g_2, \ldots, g_s \in R$. We say that ``$f$ reduces to zero modulo $\{g_i\}$'' if there exist $a_1,a_2,\ldots, a_s \in R$ such that
\begin{itemize}
\item  $f - \sum_i a_i g_i = 0$,

\item $\LM(a_i g_i) \leq \LM(f)$ for all $i \in \{1,2,\ldots,s\}$.
\end{itemize}
Let $f, g \in R$ and let $H$ be the least common multiple of their leading terms. By definition the \emph{$S$-polynomial} of $f$ and $g$ is
\[
S(f, g) = \frac{H}{\LM(f)} \cdot f - \frac{H}{\LM(g)} \cdot g.
\]

Let $I \subset R$ be an ideal, and $\{g_i\}$ a countable set of generators of $I$. Then $\{g_i\}$ is said to be a \emph{Gr\"{o}bner basis} of $I$ if $S(g_j, g_k)$ reduces to zero modulo $\{g_i\}$ for any $j, k$.
\end{nolabel}
\begin{lem}[{\cite[p. 88]{CLOS}}]\label{LCMlem}
If $\LM(f)$ and $\LM(g)$ are relatively prime then $S(f, g)$ reduces to zero modulo the set $\{f, g\}$.
\end{lem}

\begin{prop}[{\cite[p. 91]{CLOS}}]
If $\{g_i\}$ is a Gr\"{o}bner basis of $I$ then $\LT(I)$ is generated by the set of monomials $\{\LT(g_i)\}$.
\end{prop}

\begin{nolabel}\label{x.series.sec}
We now turn our attention to the morphism $J R_V \rightarrow \gr^F V$ for $V = \vir_{p, p'}$ the Virasoro minimal model. As we saw above $R_V \cong \C[x] / (x^s)$ where $s = (p-1)(p'-1)/2$. We write $R$ for the differential algebra
\[
J \C[x] = \C[x_0, x_1, x_2, \ldots]
\]
with derivation $T$ defined by $T(x_n) = x_{n+1}$. We equip the set of variables with the total order
\begin{align}\label{x.tot.ord}
x_0 > x_1 > x_2 > \ldots.
\end{align}
We fix $s \geq 2$ and denote by $I$ the ideal generated by $x_0^s$ and all its derivatives, so that $R / I \cong J(\C[x]/(x^s))$. We identify $x_0$ with $x$ and assign $\deg(x_i) = i+2$ so that the Hilbert series of $I$ is compatible with conformal degree in $J R_V$.
\end{nolabel}

\begin{prop}[{\cite[Proposition 5.2 and Lemma 5.3]{mourtada}}]\label{mour.prop}
Let $R$ and $I$ be as above. Relative to the grevlex monomial order induced by the total order (\ref{x.tot.ord}), the set $\{T^k(x_0^s)\}_{k \in \Z_{\geq 0}}$ is a Gr\"{o}bner basis of $I$. For each $k \in \Z_{\geq 0}$ the leading monomial of $T^k(x_0^s)$ is of the form $x_{i}^{\alpha} x_{i+1}^{s-\alpha}$ for some $i \geq 0$ and $0 \leq \alpha \leq s$.
\end{prop}

\begin{nolabel}
Our objective is to use Proposition (\ref{mour.prop}) to determine the Hilbert series of the ideal $I$. This is already done in \cite{mourtada}, but we repeat some details of the computation since the (conformal) grading we work with is slightly different than that used in \emph{loc. cit}.

We parametrise the monomials in $R$ of degree $n$ by partitions $(2^{i_2}, \ldots, N^{i_N})$ of $n$ into parts of size at least $2$ (the listed partition corresponds to $x_0^{i_2} \cdots x_{N-2}^{i_N}$). By Proposition (\ref{mour.prop}) the quotient $R / \LT(I)$ is spanned by the monomials corresponding to those partitions in which $i_k + i_{k+1} \leq s-1$ for $k=2,\ldots, N-1$. Gordon's generalisation {\cite[Theorem 7.5]{Andrews.Book}} of the Rogers-Ramanujan identity establishes a product formula for the generating function of the number of such partitions. Namely
\begin{align}\label{hilbert.vir}
H_I(q) = H_{\LT(I)}(q) =\prod_{\substack{m \geq 1, m \not\equiv 0, \pm 1 \\ \mod(2s+1)}} \frac{1}{1-q^m}.
\end{align}

In the case $(p, p') = (2, 2s+1)$ for some $s \geq 1$ the minimal model character (\ref{bosonic.vir}) has been shown to be precisely the right-hand side of (\ref{hilbert.vir}) \cite{KW88} \cite{Rocha-Caridi} \cite{KW2017}. Evidently then, the character of $\vir_{p, p'}$ coincides with the Hilbert series of $JR_{\vir_{p,p'}}$ when $(p, p') = (2, 2s+1)$. It follows that the surjection $JR_{\vir_{2,2s+1}} \rightarrow \gr^F{\vir_{2,2s+1}}$ is an isomorphism.

For $p, p' \geq 3$ the character (\ref{bosonic.vir}) differs from the Hilbert series (\ref{hilbert.vir}) and so the surjection $JR_{\vir_{p,p'}} \rightarrow \gr^F{\vir_{p, p'}}$ is not an isomorphism in these cases. For example for the Ising model $V = \vir_{3, 4}$ the dimensions of graded pieces of $JR_V$ and $\gr^FV$ agree up to conformal weight $\D = 8$, but disagree for $\D \geq 9$. In summary, we have proved the following theorem.
\end{nolabel} 
\begin{thm} \label{thm:minimal}
Let $V$ be the Virasoro minimal model $\vir_{p,p'}$. If $(p,p')=(2,2k+1)$ where $k \in \Z_{\geq 1}$ then the natural surjection $JR_V \rightarrow \gr^F V$ is an isomorphism. In other words the embedding $\sing(V) \hookrightarrow JX_V$ is an isomorphism of schemes. If $p,p' \geq 3$ then these morphisms are not isomorphisms, however the reduced schemes of $\sing(V)$ and $JX_V$ are isomorphic, consisting both of a single closed point.
\end{thm}

\begin{cor}\label{cor:vanish.H.vir}
Let $V$ be the Virasoro minimal model $\vir_{2,2k+1}$ where $k \in \Z_{\geq 1}$, and let $A^\bullet(q)$ be the complex (\ref{eq:A.comp.def}). For $q \neq 0$ we have
\begin{align*}
H^{-1}(A^\bullet(q)) \cong H_1^{\text{ch}}(X_{q}, \CA_V),
\end{align*}
and 
\begin{align*}
H^{-1}(A^\bullet(q=0)) = 0.
\end{align*}
\end{cor}

\begin{proof}
The isomorphism condition appearing in Theorem \ref{thm:main-theorem} is satisfied by $V$ by Theorem \ref{thm:minimal}. It remains to verify that $\Hoch_1(\zhu(V)) = 0$, but this holds because $V$ is rational and so in turn $\zhu(V)$ is semisimple.
\end{proof}

\begin{nolabel}
We remark that the algebras $R/I$ have also been studied in \cite{Kanade} in connection with principal subspaces of $\widehat{\mathfrak{sl}}_2$-modules.

It is interesting to note that two different minimal models may have isomorphic associated schemes. This is the case for $\vir_{2,7}$ and the Ising model $\vir_{3, 4}$, for instance. We thus obtain an embedding
\[
\sing(\vir_{3,4}) \hookrightarrow \sing(\vir_{2,7}),
\]
or equivalently a surjective morphism of the associated Poisson vertex algebras. We do not know if there is a relation between the corresponding vertex algebras which explains this morphism.
\label{no:explain-morphism}
\end{nolabel}

\begin{nolabel}\label{subsec:aff.c2}
Now we consider the simple affine vertex algebras. Let $\g$ be a simple Lie algebra and $k \in \Z_{\geq 0}$. The simple quotient of $V^k(\g)$ is denoted $V_k(\g)$ and is a rational vertex algebra. The maximal ideal of $V^k(\g)$ coincides with its maximal $\widehat{\g}$-submodule and is generated by the singular vector $v_k = {e_\theta}(-1)^{k+1}\vac$, where $e_\theta$ is the highest root vector of $\g$ \cite{Kac.VA.Book}. We consider the adjoint action of $\g$ on $S(\g)$ and we denote by $W$ the $\g$-submodule of $S(\g)$ generated under this action by $e_\theta^{k+1}$. Since $v_k$ is a highest weight vector for the $\widehat{\g}$-action on $V^k(\g)$, it follows that $R_{V_k(\g)}$ is the quotient of $R_{V^k(\g)} \cong S(\g)$ by the commutative algebra ideal generated by $W$.
\end{nolabel}

\begin{nolabel}
Now we pass to the special case $\g = \mathfrak{sl}_2$, using $\{e, h, f\}$ to denote its standard basis. In this case $W \subset S(\g)$ is the span of the vectors $\ad(f)^ie^{k+1}$ for $i = 0, 1, \ldots, 2k+2$.
\end{nolabel}

\begin{nolabel}\label{subsec:sl2grob}
We write $R$ for the differential algebra
\[
J \bigl( \C[e, h, f] \bigr) = \C[e_0, h_0, f_0, e_1, h_1, \ldots]
\]
with derivation $T$ defined by $T(e_n) = e_{n+1}$ and similarly for $h$ and $f$. We equip the set of variables with the total order
\begin{align}\label{var.tot.ord.16}
e_0 > h_0 > f_0 > e_1 > h_1 > f_1 > e_2 > \cdots,
\end{align}
we assign the (conformal weight) degree $\D(x_i) = i+1$ where $x$ here stands for $e$, $h$ or $f$, and we equip the set of monomials of $R$ with the corresponding grevlex order.

We introduce $e(t) = \sum_{k \in \Z_{\geq 0}} e_k \frac{t^k}{k!}$, and similarly $h(t)$ and $f(t)$, and for any polynomial $a$ in three variables we write
\[
a(e(t), h(t), f(t)) = \sum_{k \in \Z_{\geq 0}} [a(e,h,f)]_k \frac{t^k}{k!},
\]
so that, for instance, $[eh^2]_k = k! \sum_{p+q+r=k} e_p h_q h_r / (p!q!r!)$.

In general, with reference to a given monomial of $R$, we use the symbol $\#(e_i)$ to denote the power to which $e_i$ appears in the monomial, etc. Similarly $\#(e)$ denotes $\sum_{i \in \Z_{\geq 0}} \#(e_i)$ and similar symbols such as $\#(e_{\geq i})$, etc., are self-explanatory.

The algebra $R$ comes with three distinct gradings, in addition to the conformal weight grading already introduced, all of which we shall have occasion to use in the proof of Lemma \ref{K.inside} below. For a given monomial we define
\begin{align*}
\text{charge} &= \#(e) - \#(f), \\
\text{degree} &= \#(e) + \#(h) + \#(f), \\
\text{weight} &= \sum_{m \in \Z_{\geq 0}} m(\#(e_m) + \#(h_m) + \#(f_m)).
\end{align*}
In fact the conformal weight is now given by $\D = \text{degree} + \text{weight}$.
\end{nolabel}

\begin{nolabel}\label{nola:I.def}
We define $I \subset R$ to be the ideal generated by the following set of polynomials
\begin{align}\label{ad.generators}
[\ad(f)^{i}e^{k+1}]_m \quad \text{where $m \in \Z_{\geq 0}$ and $i = 0, 1, \ldots, 2k+2$}.
\end{align}
Of course $JR_{V_k(\mathfrak{sl}_2)} \cong R/I$. On the other hand we consider the set of those monomials of $R$ satisfying one (or more) of the conditions
\begin{align}\label{MPRels}
\begin{split}
\begin{alignedat}{2}
&\#(h_{m+1}) + \#(f_{m}) + \#(f_{m+1}) = k+1, 
\quad\quad & &\#(e_{m+1}) + \#(h_{m+1}) + \#(f_{m}) = k+1, \\
&\#(e_{m+1}) + \#(h_{m}) + \#(f_{m}) = k+1,
\quad\quad & &\#(e_{m}) + \#(e_{m+1}) + \#(h_{m}) = k+1,
\end{alignedat}
\end{split}
\end{align}
for some value(s) of $m \in \Z_{\geq 0}$. We set $K \subset R$ to be the ideal generated by the monomials so obtained.
\end{nolabel}

\begin{lem}\label{K.inside}
Every monomial of $R$ satisfying at least one of the conditions (\ref{MPRels}) is the leading monomial of some polynomial from the set (\ref{ad.generators}). In particular $K \subset \LT(I)$.
\end{lem}

\begin{proof}
We begin by observing that in general
\begin{align}\label{adfe.sum}
\ad(f)^ie^{k+1} = \sum_{p, q, r} \al_p^{(k)} e^p h^{q} f^{r},
\end{align}
where the sum runs over all $p, q, r \in \Z_{\geq 0}$ such that $p+q+r = k+1$ and $q+2r=i$, and the $\al_p^{(k)}$ are certain rational coefficients. Using the commutation relations of $\mathfrak{sl}_2$ to write down a recurrence relation for the coefficients, it is straightforward to deduce that every monomial term in the sum of (\ref{adfe.sum}) appears with nonzero coefficient.

We now show that monomials satisfying $\#(h_{m+1}) + \#(f_{m}) + \#(f_{m+1}) = k+1$ lie in $\LT(I)$. More specifically, for given $q, r_-, r_+ \in \Z_{\geq 0}$ satisfying $q+r_-+r_+=k+1$, we consider the monomial
\[
M = h_{m+1}^qf_m^{r_-}f_{m+1}^{r_+}.
\]
The charge of $M$ is $-(r_++r_-)$, the degree is $k+1$ and the weight is
\[
w = \text{weight}(M) = (m+1)(k+1)-r_-.
\]
We claim that $M$ is the leading monomial of
\[
X = [\ad(f)^{q+2(r_-+r_+)}e^{k+1}]_{w},
\]
here we have chosen the power of $\ad(f)$ so that $X$ is homogeneous of charge $-(r_-+r_+)$, degree $k+1$ and weight $w$.

We consider an arbitrary monomial term of $X$ and we assume it to be no lower than $M$ in the grevlex monomial order. We therefore have $\#(e_i) = \#(h_i) = 0$ for $i \leq m$, $\#(f_i) = 0$ for $i \leq m-1$ and $\#(f_m) \leq r_-$. Let us write
\[
C = \#(e_{\geq m+2})+\#(h_{\geq m+2})+\#(f_{\geq m+2}).
\]
The weight $w$ of the monomial is clearly at least
\begin{align*}
 (m+2)&\left( \#(e_{\geq m+2})+\#(h_{\geq m+2})+\#(f_{\geq m+2})\right) \\
&+ (m+1)\left( \#(e_{m+1})+\#(h_{m+1})+\#(f_{m+1})\right)+m\#(f_m) \\
%
%
= {} & C + (m+1) \left( \#(e_{\geq m+1})+\#(h_{\geq m+1})+\#(f_{\geq m}) \right) - \#(f_m) \\
= {} & C + (m+1)(k+1) - \#(f_m).
\end{align*}
From this calculation and the relation $w = (m+1)(k+1)-r_-$ we deduce
\[
\#(f_m) \geq r_- + C.
\]
Combining this with $\#(f_m) \leq r_-$ yields $C = 0$ and $\#(f_m) = r_-$. Given that $\#(f_m) = r_-$ and that $M$ does not contain the term $e_{m+1}$, the presence of the term $e_{m+1}$ in our monomial would also result in a contradiction, so we must have $\#(e) = \#(e_{m+1}) = 0$. Now the condition on the charge implies that $\#(f) = r_-+r_+$ and hence $\#(f_{m+1}) = r_+$. So in fact the monomial we started with must be $M$ itself.

Similar arguments can be employed to demonstrate that monomials satisfying the other conditions (\ref{MPRels}) also lie in $\LT(I)$. We omit the details.

\end{proof}

\begin{thm}\label{thm:affine}
Let $k \in \Z_{\geq 0}$ and let $V$ denote the simple affine vertex algebra $V_k(\mathfrak{sl}_2)$. The natural embedding $\sing(V) \hookrightarrow JX_V$ is an isomorphism.
\end{thm}
\begin{proof}
We consider the quotient $R/I$ defined in Section \ref{subsec:sl2grob} above as a graded ring with grading $\D$. The arc space $JR_V$, with its conformal weight grading, is isomorphic as a graded ring to $R/I$, and so $\chi_{JR_V}(q) = H_I(q)$.

It is a general fact that $\chi_V(q) = \chi_{\gr^F(V)}(q)$. Since $JR_V \rightarrow \gr^F(V)$ is a surjection we have $\chi_{\gr^F(V)}(q) \leq \chi_{JR_V}(q)$. Proposition \ref{equal.H} asserts that $H_I(q) = H_{\LT(I)}(q)$. By Lemma \ref{K.inside} we have $H_{\LT(I)}(q) \leq H_{K}(q)$.


Obviously the quotient $R/K$ possesses a linear basis consisting of all those monomials that satisfy the following four conditions for every $m \geq 0$ 
\begin{align}\label{MPdiffrels}
\begin{split}
\begin{alignedat}{2}
&\#(f_{m+1}) + \#(h_{m+1}) + \#(f_{m}) \leq k, 
\quad\quad & &\#(h_{m+1}) + \#(e_{m+1}) + \#(f_{m}) \leq k, \\
&\#(e_{m+1}) + \#(f_{m}) + \#(h_{m}) \leq k,
\quad\quad & &\#(e_{m+1}) + \#(h_{m}) + \#(e_{m}) \leq k.
\end{alignedat}
\end{split}
\end{align}
We now define a linear morphism $\phi : R/K \rightarrow V$ by specifying its action on monomials as follows. Write the monomial with its variables in increasing order according to (\ref{var.tot.ord.16}). Now replace $e_i$ with $e{(-i-1)}$, $h_i$ with $h{(-i-1)}$ and $f_i$ with $f{(-i-1)}$. Read the result as an element of $V$. By comparing (\ref{MPdiffrels}) with {\cite[equation (3)]{MPbook}} in the special case $k_0 = k$ it follows immediately that the elements of $V$ thus obtained constitute a basis. Hence $\phi$ is a (graded) isomorphism, and we obtain $H_{K}(q) = \chi_V(q)$.

Putting all these observations together yields
\[
\chi_V(q) = \chi_{\gr^F(V)}(q) \leq \chi_{JR_V}(q) = H_I(q) = H_{\LT(I)}(q) \leq H_{K}(q) = \chi_V(q),
\]
from which it follows that $\chi_{\gr^F(V)}(q) = \chi_{JR_V}(q)$ and hence that $JR_V \rightarrow \gr^F(V)$ is an isomorphism.
\end{proof}
We record the following consequence of the proof of Theorem \ref{thm:affine}.
\begin{prop}\label{sl2grob.kgeneral}
The set (\ref{ad.generators}) is a Gr\"{o}bner basis of the ideal $I \subset R$ defined in \ref{nola:I.def} above.
\end{prop}

\begin{proof}
By the proof of Theorem \ref{thm:affine} we have $H_{\LT(I)}(q) = H_{K}(q)$ which, together with Lemma \ref{K.inside}, implies that $K = \LT(I)$. Buchberger's theorem \cite[p. 77, p. 85]{CLOS} asserts that a set $G$ of generators of an ideal $I$ is a Gr\"{o}bner basis if and only if $\LT(I) = \left<\LT(x) | x \in G\right>$. Taking $G$ to be the set of polynomials (\ref{ad.generators}) we obtain the desired result.
\end{proof}



Finally we deduce the following corollary, whose proof is the same as that of Corollary \ref{cor:vanish.H.vir} above.
\begin{cor}
Let $V$ be the simple affine vertex algebra $V_k(\mathfrak{sl}_2)$ where $k \in \Z_{\geq 0}$, and let $A^\bullet(q)$ be the complex (\ref{eq:A.comp.def}). For $q \neq 0$ we have
\begin{align*}
H^{-1}(A^\bullet(q)) \cong H_1^{\text{ch}}(X_{q}, \CA_V),
\end{align*}
and 
\begin{align*}
H^{-1}(A^\bullet(q=0)) = 0.
\end{align*}
\end{cor}

\begin{nolabel}
In \cite{AM} it has been shown that if $V$ is \emph{quasi-lisse}, that is, if the reduced scheme $\Specm(R_V)$ has finitely many symplectic leaves, then the reduced schemes $\Specm(JR_V)$ and $\Specm(\gr^F{V})$ are isomorphic. In general however, as Theorem \ref{thm:minimal} above indicates, the schemes $\Spec(JR_V)$ and $\Spec(\gr^F{V})$ may fail to be isomorphic even when the reduced schemes each consist a single point.
\end{nolabel}
\begin{nolabel}
Let $V$ be the simple quotient of Zamolodchikov's $W_3$-algebra at central charge $c=-2$. In \cite{AL} it has been shown that the surjection $JR_V \rightarrow \gr^F{V}$ is not an isomorphism.
\end{nolabel}

\bibliographystyle{plain}

\end{document}